\documentclass[a4paper,11pt]{article}
\usepackage{MyStyle}
\usepackage{tikz-cd}

\usepackage{hyperref}
\usepackage{leftindex}
\usepackage{shuffle}

\usepackage{algorithm}
\floatname{algorithm}{Method}
\usepackage{algorithmic}
\usepackage{enumerate}
\newcolumntype{C}{>{$}c<{$}}
\usepackage{planarforest}
\newcommand{\forestA}{
\tikz[planar forest ] {

\node [b] at (0.0, 0.0) {  } 
;
}}

\newcommand{\forestB}{
\tikz[planar forest ] {

\node [b] at (0.0, 0.0) {  } 
;
}}

\newcommand{\forestC}{
\tikz[planar forest ] {

\node [b] at (0.0, 0.0) {  } 
;
}}

\newcommand{\forestD}{
\tikz[planar forest ] {

\node [b] at (0.0, 0.0) {  } 
;
}}

\newcommand{\forestE}{
\tikz[planar forest ] {

\node [b] at (0.0, 0.0) {  } 
;
}}

\newcommand{\forestF}{
\tikz[planar forest ] {

\node [b] at (0.0, 0.0) {  } 
;
}}

\newcommand{\forestG}{
\tikz[planar forest ] {

\node [b] at (0.0, 0.0) {  } 
;
}}

\newcommand{\forestH}{
\tikz[planar forest ] {

\node [b] at (0.0, 0.0) {  } 
child {node [b] at (0.0, 1.0) {  }  
}
;
}}

\newcommand{\forestI}{
\tikz[planar forest ] {

\node [b] at (0.0, 0.0) {  } 
child {node [b] at (0.0, 1.0) {  }  
child {node [b] at (0.0, 1.0) {  }  
}
}
;
}}

\newcommand{\forestJ}{
\tikz[planar forest ] {

\node [b] at (0.0, 0.0) {  } 
child {node [b] at (-0.5, 1.0) {  }  
}
child {node [b] at (0.5, 1.0) {  }  
}
;
}}

\newcommand{\forestK}{
\tikz[planar forest ] {

\node [b] at (0.0, 0.0) {  } 
child {node [b] at (0.0, 1.0) {  }  
child {node [b] at (0.0, 1.0) {  }  
child {node [b] at (0.0, 1.0) {  }  
}
}
}
;
}}

\newcommand{\forestL}{
\tikz[planar forest ] {

\node [b] at (0.0, 0.0) {  } 
child {node [b] at (0.0, 1.0) {  }  
child {node [b] at (-0.5, 1.0) {  }  
}
child {node [b] at (0.5, 1.0) {  }  
}
}
;
}}

\newcommand{\forestM}{
\tikz[planar forest ] {

\node [b] at (0.0, 0.0) {  } 
child {node [b] at (-0.5, 1.0) {  }  
}
child {node [b] at (0.5, 1.0) {  }  
child {node [b] at (0.0, 1.0) {  }  
}
}
;
}}

\newcommand{\forestN}{
\tikz[planar forest ] {

\node [b] at (0.0, 0.0) {  } 
child {node [b] at (-0.5, 1.0) {  }  
child {node [b] at (0.0, 1.0) {  }  
}
}
child {node [b] at (0.5, 1.0) {  }  
}
;
}}

\newcommand{\forestO}{
\tikz[planar forest ] {

\node [b] at (0.0, 0.0) {  } 
child {node [b] at (-1.0, 1.0) {  }  
}
child {node [b] at (0.0, 1.0) {  }  
}
child {node [b] at (1.0, 1.0) {  }  
}
;
}}

\newcommand{\forestP}{
\tikz[planar forest ] {

\node [b] at (0.0, 0.0) {  } 
child {node [b] at (0.0, 1.0) {  }  
child {node [b] at (0.0, 1.0) {  }  
child {node [b] at (0.0, 1.0) {  }  
child {node [b] at (0.0, 1.0) {  }  
}
}
}
}
;
}}

\newcommand{\forestQ}{
\tikz[planar forest ] {

\node [b] at (0.0, 0.0) {  } 
child {node [b] at (0.0, 1.0) {  }  
child {node [b] at (0.0, 1.0) {  }  
child {node [b] at (-0.5, 1.0) {  }  
}
child {node [b] at (0.5, 1.0) {  }  
}
}
}
;
}}

\newcommand{\forestR}{
\tikz[planar forest ] {

\node [b] at (0.0, 0.0) {  } 
child {node [b] at (0.0, 1.0) {  }  
child {node [b] at (-0.5, 1.0) {  }  
}
child {node [b] at (0.5, 1.0) {  }  
child {node [b] at (0.0, 1.0) {  }  
}
}
}
;
}}

\newcommand{\forestS}{
\tikz[planar forest ] {

\node [b] at (0.0, 0.0) {  } 
child {node [b] at (0.0, 1.0) {  }  
child {node [b] at (-0.5, 1.0) {  }  
child {node [b] at (0.0, 1.0) {  }  
}
}
child {node [b] at (0.5, 1.0) {  }  
}
}
;
}}

\newcommand{\forestT}{
\tikz[planar forest ] {

\node [b] at (0.0, 0.0) {  } 
child {node [b] at (0.0, 1.0) {  }  
child {node [b] at (-1.0, 1.0) {  }  
}
child {node [b] at (0.0, 1.0) {  }  
}
child {node [b] at (1.0, 1.0) {  }  
}
}
;
}}

\newcommand{\forestU}{
\tikz[planar forest ] {

\node [b] at (0.0, 0.0) {  } 
child {node [b] at (-0.5, 1.0) {  }  
}
child {node [b] at (0.5, 1.0) {  }  
child {node [b] at (0.0, 1.0) {  }  
child {node [b] at (0.0, 1.0) {  }  
}
}
}
;
}}

\newcommand{\forestV}{
\tikz[planar forest ] {

\node [b] at (0.0, 0.0) {  } 
child {node [b] at (-0.5, 1.0) {  }  
child {node [b] at (0.0, 1.0) {  }  
child {node [b] at (0.0, 1.0) {  }  
}
}
}
child {node [b] at (0.5, 1.0) {  }  
}
;
}}

\newcommand{\forestW}{
\tikz[planar forest ] {

\node [b] at (0.0, 0.0) {  } 
child {node [b] at (-0.5, 1.0) {  }  
}
child {node [b] at (0.5, 1.0) {  }  
child {node [b] at (-0.5, 1.0) {  }  
}
child {node [b] at (0.5, 1.0) {  }  
}
}
;
}}

\newcommand{\forestX}{
\tikz[planar forest ] {

\node [b] at (0.0, 0.0) {  } 
child {node [b] at (-0.5, 1.0) {  }  
child {node [b] at (-0.5, 1.0) {  }  
}
child {node [b] at (0.5, 1.0) {  }  
}
}
child {node [b] at (0.5, 1.0) {  }  
}
;
}}

\newcommand{\forestY}{
\tikz[planar forest ] {

\node [b] at (0.0, 0.0) {  } 
child {node [b] at (-0.5, 1.0) {  }  
child {node [b] at (0.0, 1.0) {  }  
}
}
child {node [b] at (0.5, 1.0) {  }  
child {node [b] at (0.0, 1.0) {  }  
}
}
;
}}

\newcommand{\forestAB}{
\tikz[planar forest ] {

\node [b] at (0.0, 0.0) {  } 
child {node [b] at (-1.0, 1.0) {  }  
}
child {node [b] at (0.0, 1.0) {  }  
}
child {node [b] at (1.0, 1.0) {  }  
child {node [b] at (0.0, 1.0) {  }  
}
}
;
}}

\newcommand{\forestBB}{
\tikz[planar forest ] {

\node [b] at (0.0, 0.0) {  } 
child {node [b] at (-1.0, 1.0) {  }  
}
child {node [b] at (0.0, 1.0) {  }  
child {node [b] at (0.0, 1.0) {  }  
}
}
child {node [b] at (1.0, 1.0) {  }  
}
;
}}

\newcommand{\forestCB}{
\tikz[planar forest ] {

\node [b] at (0.0, 0.0) {  } 
child {node [b] at (-1.0, 1.0) {  }  
child {node [b] at (0.0, 1.0) {  }  
}
}
child {node [b] at (0.0, 1.0) {  }  
}
child {node [b] at (1.0, 1.0) {  }  
}
;
}}

\newcommand{\forestDB}{
\tikz[planar forest ] {

\node [b] at (0.0, 0.0) {  } 
child {node [b] at (-1.5, 1.0) {  }  
}
child {node [b] at (-0.5, 1.0) {  }  
}
child {node [b] at (0.5, 1.0) {  }  
}
child {node [b] at (1.5, 1.0) {  }  
}
;
}}

\newcommand{\forestEB}{
\tikz[planar forest ] {

\node [b] at (0.0, 0.0) {  } 
child {node [b] at (-0.5, 1.0) {  }  
}
child {node [b] at (0.5, 1.0) {  }  
child {node [b] at (0.0, 1.0) {  }  
}
}
;
}}

\newcommand{\forestFB}{
\tikz[planar forest ] {

\node [b] at (0.0, 0.0) {  } 
;
}}

\newcommand{\forestGB}{
\tikz[planar forest ] {

\node [b] at (0.0, 0.0) {  } 
;
}}

\newcommand{\forestHB}{
\tikz[planar forest ] {

\node [b] at (0.0, 0.0) {  } 
;
}}

\newcommand{\forestIB}{
\tikz[planar forest ] {

\node [b] at (0.0, 0.0) {  } 
;
}}

\newcommand{\forestJB}{
\tikz[planar forest ] {

\node [b] at (0.0, 0.0) {  } 
child {node [b] at (-0.5, 1.0) {  }  
child {node [b] at (0.0, 1.0) {  }  
}
}
child {node [b] at (0.5, 1.0) {  }  
}
;
}}

\newcommand{\forestKB}{
\tikz[planar forest ] {

\node [b] at (0.0, 0.0) {  } 
;
}}

\newcommand{\forestLB}{
\tikz[planar forest ] {

\node [b] at (0.0, 0.0) {  } 
;
}}

\newcommand{\forestMB}{
\tikz[planar forest ] {

\node [b] at (0.0, 0.0) {  } 
;
}}

\newcommand{\forestNB}{
\tikz[planar forest ] {

\node [b] at (0.0, 0.0) {  } 
;
}}

\newcommand{\forestOB}{
\tikz[planar forest ] {

\node [w] at (0.0, 0.0) {  } 
child {node [w] at (0.0, 1.0) {  }  
}
;
}}

\newcommand{\forestPB}{
\tikz[planar forest ] {

\node [b] at (0.0, 0.0) {  } 
child {node [b] at (-0.5, 1.0) {  }  
child {node [b] at (0.0, 1.0) {  }  
}
}
child {node [b] at (0.5, 1.0) {  }  
}
;
}}

\newcommand{\forestQB}{
\tikz[planar forest ] {

\node [b] at (0.0, 0.0) {  } 
child {node [w] at (-1.0, 1.0) {  }  
child {node [w] at (0.0, 1.0) {  }  
}
}
child {node [b] at (0.0, 1.0) {  }  
child {node [b] at (0.0, 1.0) {  }  
}
}
child {node [b] at (1.0, 1.0) {  }  
}
;
}}

\newcommand{\forestRB}{
\tikz[planar forest ] {

\node [b] at (0.0, 0.0) {  } 
child {node [b] at (-0.5, 1.0) {  }  
child {node [w] at (-0.5, 1.0) {  }  
child {node [w] at (0.0, 1.0) {  }  
}
}
child {node [b] at (0.5, 1.0) {  }  
}
}
child {node [b] at (0.5, 1.0) {  }  
}
;
}}

\newcommand{\forestSB}{
\tikz[planar forest ] {

\node [b] at (0.0, 0.0) {  } 
child {node [b] at (-0.5, 1.0) {  }  
child {node [b] at (0.0, 1.0) {  }  
}
}
child {node [b] at (0.5, 1.0) {  }  
child {node [w] at (0.0, 1.0) {  }  
child {node [w] at (0.0, 1.0) {  }  
}
}
}
;
}}

\newcommand{\forestTB}{
\tikz[planar forest ] {

\node [b] at (0.0, 0.0) {  } 
child {node [b] at (-0.5, 1.0) {  }  
child {node [b] at (0.0, 1.0) {  }  
child {node [w] at (0.0, 1.0) {  }  
child {node [w] at (0.0, 1.0) {  }  
}
}
}
}
child {node [b] at (0.5, 1.0) {  }  
}
;
}}

\newcommand{\forestUB}{
\tikz[planar forest ] {

\node [b] at (0.0, 0.0) {  } 
child {node [b] at (-2.0, 1.0) {  }  
child {node [w] at (0.0, 1.0) {  }  
}
}
child {node [b] at (0.0, 1.0) {  }  
child {node [b] at (-1.0, 1.0) {  }  
child {node [b] at (0.0, 1.0) {  }  
}
}
child {node [e] at (0.0, 1.0) {  }  
}
child {node [w] at (1.0, 1.0) {  }  
}
}
child {node [b] at (1.0, 1.0) {  }  
}
;
}}

\newcommand{\forestVB}{
\tikz[planar forest ] {

\node [b] at (0.0, 0.0) {  } 
;
}}

\newcommand{\forestWB}{
\tikz[planar forest ] {

\node [w] at (0.0, 0.0) {  } 
;
}}

\newcommand{\forestXB}{
\tikz[planar forest ] {

\node [b] at (0.0, 0.0) {  } 
child {node [b] at (-2.0, 1.0) {  }  
child {node [w] at (0.0, 1.0) {  }  
}
}
child {node [b] at (0.0, 1.0) {  }  
child {node [b] at (-1.0, 1.0) {  }  
child {node [b] at (0.0, 1.0) {  }  
}
}
child {node [b] at (0.0, 1.0) {  }  
}
child {node [w] at (1.0, 1.0) {  }  
}
}
child {node [e] at (1.0, 1.0) {  }  
}
;
}}

\newcommand{\forestYB}{
\tikz[planar forest ] {

\node [b] at (0.0, 0.0) {  } 
;
}}

\newcommand{\forestAC}{
\tikz[planar forest ] {

\node [w] at (0.0, 0.0) {  } 
;
}}

\newcommand{\forestBC}{
\tikz[planar forest ] {

\node [b] at (0.0, 0.0) {  } 
;
}}

\newcommand{\forestCC}{
\tikz[planar forest ] {

\node [w] at (0.0, 0.0) {  } 
;
}}

\newcommand{\forestDC}{
\tikz[planar forest ] {

\node [b] at (0.0, 0.0) {  } 
child {node [b] at (-1.0, 1.0) {  }  
}
child {node [e] at (0.0, 1.0) {  }  
}
child {node [b] at (1.0, 1.0) {  }  
}
;
}}

\newcommand{\forestEC}{
\tikz[planar forest ] {

\node [w] at (0.0, 0.0) {  } 
child {node [w] at (0.0, 1.0) {  }  
}
;
}}

\newcommand{\forestFC}{
\tikz[planar forest ] {

\node [b] at (0.0, 0.0) {  } 
child {node [b] at (-1.0, 1.0) {  }  
}
child {node [w] at (0.0, 1.0) {  }  
child {node [w] at (0.0, 1.0) {  }  
}
}
child {node [b] at (1.0, 1.0) {  }  
}
;
}}

\newcommand{\forestGC}{
\tikz[planar forest ] {

\node [b] at (0.0, 0.0) {  } 
child {node [b] at (-1.0, 1.0) {  }  
}
child {node [e] at (0.0, 1.0) {  }  
}
child {node [b] at (1.0, 1.0) {  }  
}
;
}}

\newcommand{\forestHC}{
\tikz[planar forest ] {

\node [w] at (0.0, 0.0) {  } 
child {node [w] at (0.0, 1.0) {  }  
child {node [e] at (0.0, 1.0) {  }  
}
}
;
}}

\newcommand{\forestIC}{
\tikz[planar forest ] {

\node [b] at (0.0, 0.0) {  } 
child {node [b] at (-1.0, 1.0) {  }  
}
child {node [w] at (0.0, 1.0) {  }  
child {node [w] at (0.0, 1.0) {  }  
child {node [e] at (0.0, 1.0) {  }  
}
}
}
child {node [b] at (1.0, 1.0) {  }  
}
;
}}

\newcommand{\forestJC}{
\tikz[planar forest ] {

\node [b] at (0.0, 0.0) {  } 
child {node [b] at (-1.0, 1.0) {  }  
}
child {node [e] at (0.0, 1.0) {  }  
}
child {node [b] at (1.0, 1.0) {  }  
}
;
}}

\newcommand{\forestKC}{
\tikz[planar forest ] {

\node [w] at (0.0, 0.0) {  } 
child {node [w] at (0.0, 1.0) {  }  
}
;
}}

\newcommand{\forestLC}{
\tikz[planar forest ] {

\node [b] at (0.0, 0.0) {  } 
child {node [b] at (-1.5, 1.0) {  }  
}
child {node [w] at (-0.5, 1.0) {  }  
child {node [w] at (0.0, 1.0) {  }  
}
}
child {node [e] at (0.5, 1.0) {  }  
}
child {node [b] at (1.5, 1.0) {  }  
}
;
}}

\newcommand{\forestMC}{
\tikz[planar forest ] {

\node [b] at (0.0, 0.0) {  } 
child {node [b] at (-1.5, 1.0) {  }  
}
child {node [e] at (-0.5, 1.0) {  }  
}
child {node [w] at (0.5, 1.0) {  }  
child {node [w] at (0.0, 1.0) {  }  
}
}
child {node [b] at (1.5, 1.0) {  }  
}
;
}}

\newcommand{\forestNC}{
\tikz[planar forest ] {

\node [b] at (0.0, 0.0) {  } 
child {node [b] at (-1.0, 1.0) {  }  
}
child {node [e] at (0.0, 1.0) {  }  
}
child {node [b] at (1.0, 1.0) {  }  
}
;
}}

\newcommand{\forestOC}{
\tikz[planar forest ] {

\node [w] at (0.0, 0.0) {  } 
child {node [w] at (0.0, 1.0) {  }  
}
;
}}

\newcommand{\forestPC}{
\tikz[planar forest ] {

\node [w] at (0.0, 0.0) {  } 
;
}}

\newcommand{\forestQC}{
\tikz[planar forest ] {

\node [b] at (0.0, 0.0) {  } 
child {node [b] at (-1.5, 1.0) {  }  
}
child {node [w] at (-0.5, 1.0) {  }  
child {node [w] at (0.0, 1.0) {  }  
}
}
child {node [w] at (0.5, 1.0) {  }  
}
child {node [b] at (1.5, 1.0) {  }  
}
;
}}

\newcommand{\forestRC}{
\tikz[planar forest ] {

\node [b] at (0.0, 0.0) {  } 
child {node [b] at (-1.5, 1.0) {  }  
}
child {node [w] at (-0.5, 1.0) {  }  
}
child {node [w] at (0.5, 1.0) {  }  
child {node [w] at (0.0, 1.0) {  }  
}
}
child {node [b] at (1.5, 1.0) {  }  
}
;
}}

\newcommand{\forestSC}{
\tikz[planar forest ] {

\node [b] at (0.0, 0.0) {  } 
child {node [b] at (-1.0, 1.0) {  }  
child {node [b] at (0.0, 1.0) {  }  
}
}
child {node [b] at (1.0, 1.0) {  }  
child {node [b] at (-1.0, 1.0) {  }  
child {node [e] at (0.0, 1.0) {  }  
}
}
child {node [b] at (0.0, 1.0) {  }  
}
child {node [b] at (1.0, 1.0) {  }  
}
}
;
}}

\newcommand{\forestTC}{
\tikz[planar forest ] {

\node [b] at (0.0, 0.0) {  } 
child {node [b] at (-0.5, 1.0) {  }  
child {node [b] at (0.0, 1.0) {  }  
}
}
child {node [e] at (0.5, 1.0) {  }  
}
;
}}

\newcommand{\forestUC}{
\tikz[planar forest ] {

\node [b] at (0.0, 0.0) {  } 
child {node [e] at (-1.0, 1.0) {  }  
}
child {node [b] at (0.0, 1.0) {  }  
}
child {node [b] at (1.0, 1.0) {  }  
}
;
}}

\newcommand{\forestVC}{
\tikz[planar forest ] {

\node [b] at (0.0, 0.0) {  } 
child {node [e] at (0.0, 1.0) {  }  
}
;
}}

\newcommand{\forestWC}{
\tikz[planar forest ] {

\node [b] at (0.0, 0.0) {  } 
child {node [b] at (-0.5, 1.0) {  }  
child {node [b] at (0.0, 1.0) {  }  
}
}
child {node [b] at (0.5, 1.0) {  }  
}
;
}}

\newcommand{\forestXC}{
\tikz[planar forest ] {

\node [b] at (0.0, 0.0) {  } 
child {node [e] at (-1.0, 1.0) {  }  
}
child {node [b] at (0.0, 1.0) {  }  
child {node [b] at (0.0, 1.0) {  }  
}
}
child {node [b] at (1.0, 1.0) {  }  
}
;
}}

\newcommand{\forestYC}{
\tikz[planar forest ] {

\node [b] at (0.0, 0.0) {  } 
child {node [b] at (-0.5, 1.0) {  }  
child {node [e] at (-0.5, 1.0) {  }  
}
child {node [b] at (0.5, 1.0) {  }  
}
}
child {node [b] at (0.5, 1.0) {  }  
}
;
}}

\newcommand{\forestAD}{
\tikz[planar forest ] {

\node [b] at (0.0, 0.0) {  } 
child {node [b] at (-0.5, 1.0) {  }  
child {node [b] at (0.0, 1.0) {  }  
child {node [e] at (0.0, 1.0) {  }  
}
}
}
child {node [b] at (0.5, 1.0) {  }  
}
;
}}

\newcommand{\forestBD}{
\tikz[planar forest ] {

\node [b] at (0.0, 0.0) {  } 
child {node [b] at (-0.5, 1.0) {  }  
child {node [b] at (0.0, 1.0) {  }  
}
}
child {node [b] at (0.5, 1.0) {  }  
child {node [e] at (0.0, 1.0) {  }  
}
}
;
}}

\newcommand{\forestCD}{
\tikz[planar forest ] {

\node [w] at (0.0, 0.0) {  } 
child {node [w] at (0.0, 1.0) {  }  
}
;
}}

\newcommand{\forestDD}{
\tikz[planar forest ] {

\node [b] at (0.0, 0.0) {  } 
child {node [e] at (-1.0, 1.0) {  }  
}
child {node [b] at (0.0, 1.0) {  }  
}
child {node [b] at (1.0, 1.0) {  }  
}
;
}}

\newcommand{\forestED}{
\tikz[planar forest ] {

\node [b] at (0.0, 0.0) {  } 
child {node [e] at (-1.0, 1.0) {  }  
}
child {node [b] at (0.0, 1.0) {  }  
child {node [w] at (0.0, 1.0) {  }  
child {node [w] at (0.0, 1.0) {  }  
}
}
}
child {node [b] at (1.0, 1.0) {  }  
}
;
}}

\newcommand{\forestFD}{
\tikz[planar forest ] {

\node [b] at (0.0, 0.0) {  } 
child {node [e] at (-1.0, 1.0) {  }  
}
child {node [b] at (0.0, 1.0) {  }  
}
child {node [b] at (1.0, 1.0) {  }  
child {node [w] at (0.0, 1.0) {  }  
child {node [w] at (0.0, 1.0) {  }  
}
}
}
;
}}

\newcommand{\forestGD}{
\tikz[planar forest ] {

\node [b] at (0.0, 0.0) {  } 
child {node [e] at (-1.5, 1.0) {  }  
}
child {node [b] at (-0.5, 1.0) {  }  
}
child {node [b] at (0.5, 1.0) {  }  
}
child {node [w] at (1.5, 1.0) {  }  
child {node [w] at (0.0, 1.0) {  }  
}
}
;
}}

\newcommand{\forestHD}{
\tikz[planar forest ] {

\node [l] at (0.0, 0.0) { $\times$ } 
;
}}

\newcommand{\forestID}{
\tikz[planar forest ] {

\node [b] at (0.0, 0.0) {  } 
child {node [l] at (0.0, 1.0) { $\times$ }  
}
;
}}

\newcommand{\forestJD}{
\tikz[planar forest ] {

\node [b] at (0.0, 0.0) {  } 
child {node [b] at (-0.5, 1.0) {  }  
}
child {node [l] at (0.5, 1.0) { $\times$ }  
}
;
}}

\newcommand{\forestKD}{
\tikz[planar forest ] {

\node [b] at (0.0, 0.0) {  } 
child {node [l] at (-0.5, 1.0) { $\times$ }  
}
child {node [b] at (0.5, 1.0) {  }  
}
;
}}

\newcommand{\forestLD}{
\tikz[planar forest ] {

\node [b] at (0.0, 0.0) {  } 
child {node [b] at (-0.5, 1.0) {  }  
}
child {node [l] at (0.5, 1.0) { $\times$ }  
}
;
}}

\newcommand{\forestMD}{
\tikz[planar forest ] {

\node [b] at (0.0, 0.0) {  } 
child {node [l] at (-0.5, 1.0) { $\times$ }  
}
child {node [b] at (0.5, 1.0) {  }  
}
;
}}

\newcommand{\forestND}{
\tikz[planar forest ] {

\node [b] at (0.0, 0.0) {  } 
child {node [l] at (0.0, 1.0) { $\times$ }  
}
;
}}

\newcommand{\forestOD}{
\tikz[planar forest ] {

\node [b] at (0.0, 0.0) {  } 
child {node [l] at (-0.5, 1.0) { $\times$ }  
}
child {node [b] at (0.5, 1.0) {  }  
}
;
}}

\newcommand{\forestPD}{
\tikz[planar forest ] {

\node [b] at (0.0, 0.0) {  } 
child {node [l] at (0.0, 1.0) { $\times$ }  
}
;
}}

\newcommand{\forestQD}{
\tikz[planar forest ] {

\node [b] at (0.0, 0.0) {  } 
child {node [b] at (-0.5, 1.0) {  }  
}
child {node [l] at (0.5, 1.0) { $\times$ }  
}
;
}}

\newcommand{\forestRD}{
\tikz[planar forest ] {

\node [l] at (0.0, 0.0) { $\times$ } 
;
}}

\newcommand{\forestSD}{
\tikz[planar forest ] {

\node [b] at (0.0, 0.0) {  } 
child {node [l] at (-0.5, 1.0) { $\times$ }  
}
child {node [e] at (0.5, 1.0) {  }  
}
;
}}

\newcommand{\forestTD}{
\tikz[planar forest ] {

\node [b] at (0.0, 0.0) {  } 
child {node [b] at (-0.5, 1.0) {  }  
}
child {node [l] at (0.5, 1.0) { $\times$ }  
}
;
}}

\newcommand{\forestUD}{
\tikz[planar forest ] {

\node [b] at (0.0, 0.0) {  } 
child {node [l] at (-0.5, 1.0) { $\times$ }  
}
child {node [b] at (0.5, 1.0) {  }  
}
;
}}

\newcommand{\forestVD}{
\tikz[planar forest ] {

\node [w] at (0.0, 0.0) {  } 
child {node [w] at (0.0, 1.0) {  }  
}
;
}}

\newcommand{\forestWD}{
\tikz[planar forest ] {

\node [b] at (0.0, 0.0) {  } 
child {node [l] at (-0.5, 1.0) { $\times$ }  
}
child {node [w] at (0.5, 1.0) {  }  
child {node [w] at (0.0, 1.0) {  }  
}
}
;
}}

\newcommand{\forestXD}{
\tikz[planar forest ] {

\node [b] at (0.0, 0.0) {  } 
child {node [b] at (-0.5, 1.0) {  }  
}
child {node [l] at (0.5, 1.0) { $\times$ }  
}
;
}}

\newcommand{\forestYD}{
\tikz[planar forest ] {

\node [b] at (0.0, 0.0) {  } 
child {node [l] at (-0.5, 1.0) { $\times$ }  
}
child {node [b] at (0.5, 1.0) {  }  
}
;
}}

\newcommand{\forestAE}{
\tikz[planar forest ] {

\node [b] at (0.0, 0.0) {  } 
child {node [l] at (0.0, 1.0) { $\times$ }  
}
;
}}

\newcommand{\forestBE}{
\tikz[planar forest ] {

\node [b] at (0.0, 0.0) {  } 
child {node [e] at (-0.5, 1.0) {  }  
}
child {node [l] at (0.5, 1.0) { $\times$ }  
}
;
}}

\newcommand{\forestCE}{
\tikz[planar forest ] {

\node [b] at (0.0, 0.0) {  } 
;
}}

\newcommand{\forestDE}{
\tikz[planar forest ] {

\node [b] at (0.0, 0.0) {  } 
child {node [e] at (0.0, 1.0) {  }  
}
;
}}

\newcommand{\forestEE}{
\tikz[planar forest ] {

\node [b] at (0.0, 0.0) {  } 
child {node [l] at (0.0, 1.0) { $\times$ }  
}
;
}}

\newcommand{\forestFE}{
\tikz[planar forest ] {

\node [b] at (0.0, 0.0) {  } 
child {node [b] at (0.0, 1.0) {  }  
}
;
}}

\newcommand{\forestGE}{
\tikz[planar forest ] {

\node [b] at (0.0, 0.0) {  } 
child {node [b] at (0.0, 1.0) {  }  
child {node [e] at (0.0, 1.0) {  }  
}
}
;
}}

\newcommand{\forestHE}{
\tikz[planar forest ] {

\node [b] at (0.0, 0.0) {  } 
child {node [e] at (-0.5, 1.0) {  }  
}
child {node [b] at (0.5, 1.0) {  }  
}
;
}}

\newcommand{\forestIE}{
\tikz[planar forest ] {

\node [b] at (0.0, 0.0) {  } 
child {node [l] at (0.0, 1.0) { $\times$ }  
}
;
}}

\newcommand{\forestJE}{
\tikz[planar forest ] {

\node [b] at (0.0, 0.0) {  } 
child {node [l] at (0.0, 1.0) { $\times$ }  
}
;
}}

\newcommand{\forestKE}{
\tikz[planar forest ] {

\node [b] at (0.0, 0.0) {  } 
child {node [l] at (-0.5, 1.0) { $\times$ }  
}
child {node [b] at (0.5, 1.0) {  }  
}
;
}}

\newcommand{\forestLE}{
\tikz[planar forest ] {

\node [b] at (0.0, 0.0) {  } 
child {node [l] at (0.0, 1.0) { $\times$ }  
}
;
}}

\newcommand{\forestME}{
\tikz[planar forest ] {

\node [b] at (0.0, 0.0) {  } 
;
}}

\newcommand{\forestNE}{
\tikz[planar forest ] {

\node [b] at (0.0, 0.0) {  } 
child {node [l] at (0.0, 1.0) { $\times$ }  
}
;
}}

\newcommand{\forestOE}{
\tikz[planar forest ] {

\node [b] at (0.0, 0.0) {  } 
child {node [e] at (0.0, 1.0) {  }  
}
;
}}

\newcommand{\forestPE}{
\tikz[planar forest ] {

\node [b] at (0.0, 0.0) {  } 
child {node [e] at (-0.5, 1.0) {  }  
}
child {node [l] at (0.5, 1.0) { $\times$ }  
}
;
}}

\newcommand{\forestQE}{
\tikz[planar forest ] {

\node [b] at (0.0, 0.0) {  } 
;
}}

\newcommand{\forestRE}{
\tikz[planar forest ] {

\node [b] at (0.0, 0.0) {  } 
child {node [l] at (0.0, 1.0) { $\times$ }  
}
;
}}

\newcommand{\forestSE}{
\tikz[planar forest ] {

\node [b] at (0.0, 0.0) {  } 
child {node [l] at (0.0, 1.0) { $\times$ }  
}
;
}}

\newcommand{\forestTE}{
\tikz[planar forest ] {

\node [b] at (0.0, 0.0) {  } 
child {node [b] at (-0.5, 1.0) {  }  
}
child {node [l] at (0.5, 1.0) { $\times$ }  
}
;
}}

\newcommand{\forestUE}{
\tikz[planar forest ] {

\node [b] at (0.0, 0.0) {  } 
child {node [b] at (0.0, 1.0) {  }  
child {node [b] at (0.0, 1.0) {  }  
}
}
;
}}

\newcommand{\forestVE}{
\tikz[planar forest ] {

\node [b] at (0.0, 0.0) {  } 
child {node [b] at (0.0, 1.0) {  }  
child {node [b] at (0.0, 1.0) {  }  
child {node [e] at (0.0, 1.0) {  }  
}
}
}
;
}}

\newcommand{\forestWE}{
\tikz[planar forest ] {

\node [b] at (0.0, 0.0) {  } 
child {node [b] at (0.0, 1.0) {  }  
child {node [e] at (-0.5, 1.0) {  }  
}
child {node [b] at (0.5, 1.0) {  }  
}
}
;
}}

\newcommand{\forestXE}{
\tikz[planar forest ] {

\node [b] at (0.0, 0.0) {  } 
child {node [e] at (-0.5, 1.0) {  }  
}
child {node [b] at (0.5, 1.0) {  }  
child {node [b] at (0.0, 1.0) {  }  
}
}
;
}}

\newcommand{\forestYE}{
\tikz[planar forest ] {

\node [b] at (0.0, 0.0) {  } 
child {node [l] at (0.0, 1.0) { $\times$ }  
}
;
}}

\newcommand{\forestAF}{
\tikz[planar forest ] {

\node [b] at (0.0, 0.0) {  } 
child {node [l] at (0.0, 1.0) { $\times$ }  
}
;
}}

\newcommand{\forestBF}{
\tikz[planar forest ] {

\node [b] at (0.0, 0.0) {  } 
child {node [l] at (0.0, 1.0) { $\times$ }  
}
;
}}

\newcommand{\forestCF}{
\tikz[planar forest ] {

\node [b] at (0.0, 0.0) {  } 
child {node [l] at (0.0, 1.0) { $\times$ }  
}
;
}}

\newcommand{\forestDF}{
\tikz[planar forest ] {

\node [b] at (0.0, 0.0) {  } 
child {node [l] at (-0.5, 1.0) { $\times$ }  
}
child {node [b] at (0.5, 1.0) {  }  
}
;
}}

\newcommand{\forestEF}{
\tikz[planar forest ] {

\node [b] at (0.0, 0.0) {  } 
child {node [l] at (-0.5, 1.0) { $\times$ }  
}
child {node [b] at (0.5, 1.0) {  }  
child {node [b] at (0.0, 1.0) {  }  
}
}
;
}}

\newcommand{\forestFF}{
\tikz[planar forest ] {

\node [b] at (0.0, 0.0) {  } 
child {node [b] at (-0.5, 1.0) {  }  
}
child {node [b] at (0.5, 1.0) {  }  
}
;
}}

\newcommand{\forestGF}{
\tikz[planar forest ] {

\node [b] at (0.0, 0.0) {  } 
child {node [b] at (-0.5, 1.0) {  }  
child {node [e] at (0.0, 1.0) {  }  
}
}
child {node [b] at (0.5, 1.0) {  }  
}
;
}}

\newcommand{\forestHF}{
\tikz[planar forest ] {

\node [b] at (0.0, 0.0) {  } 
child {node [b] at (-0.5, 1.0) {  }  
}
child {node [b] at (0.5, 1.0) {  }  
child {node [e] at (0.0, 1.0) {  }  
}
}
;
}}

\newcommand{\forestIF}{
\tikz[planar forest ] {

\node [b] at (0.0, 0.0) {  } 
child {node [e] at (-1.0, 1.0) {  }  
}
child {node [b] at (0.0, 1.0) {  }  
}
child {node [b] at (1.0, 1.0) {  }  
}
;
}}

\newcommand{\forestJF}{
\tikz[planar forest ] {

\node [b] at (0.0, 0.0) {  } 
child {node [l] at (-0.5, 1.0) { $\times$ }  
}
child {node [b] at (0.5, 1.0) {  }  
}
;
}}

\newcommand{\forestKF}{
\tikz[planar forest ] {

\node [b] at (0.0, 0.0) {  } 
child {node [l] at (0.0, 1.0) { $\times$ }  
}
;
}}

\newcommand{\forestLF}{
\tikz[planar forest ] {

\node [b] at (0.0, 0.0) {  } 
child {node [b] at (-0.5, 1.0) {  }  
}
child {node [l] at (0.5, 1.0) { $\times$ }  
}
;
}}

\newcommand{\forestMF}{
\tikz[planar forest ] {

\node [b] at (0.0, 0.0) {  } 
child {node [l] at (0.0, 1.0) { $\times$ }  
}
;
}}

\newcommand{\forestNF}{
\tikz[planar forest ] {

\node [b] at (0.0, 0.0) {  } 
child {node [l] at (-1.0, 1.0) { $\times$ }  
}
child {node [b] at (0.0, 1.0) {  }  
}
child {node [b] at (1.0, 1.0) {  }  
}
;
}}

\newcommand{\forestOF}{
\tikz[planar forest ] {

\node [b] at (0.0, 0.0) {  } 
child {node [l] at (0.0, 1.0) { $\times$ }  
}
;
}}

\newcommand{\forestPF}{
\tikz[planar forest ] {

\node [b] at (0.0, 0.0) {  } 
child {node [b] at (0.0, 1.0) {  }  
}
;
}}

\newcommand{\forestQF}{
\tikz[planar forest ] {

\node [b] at (0.0, 0.0) {  } 
child {node [l] at (0.0, 1.0) { $\times$ }  
}
;
}}

\newcommand{\forestRF}{
\tikz[planar forest ] {

\node [b] at (0.0, 0.0) {  } 
child {node [b] at (0.0, 1.0) {  }  
child {node [e] at (0.0, 1.0) {  }  
}
}
;
}}

\newcommand{\forestSF}{
\tikz[planar forest ] {

\node [b] at (0.0, 0.0) {  } 
child {node [l] at (0.0, 1.0) { $\times$ }  
}
;
}}

\newcommand{\forestTF}{
\tikz[planar forest ] {

\node [b] at (0.0, 0.0) {  } 
child {node [e] at (-0.5, 1.0) {  }  
}
child {node [b] at (0.5, 1.0) {  }  
}
;
}}

\newcommand{\forestUF}{
\tikz[planar forest ] {

\node [b] at (0.0, 0.0) {  } 
child {node [e] at (-0.5, 1.0) {  }  
}
child {node [l] at (0.5, 1.0) { $\times$ }  
}
;
}}

\newcommand{\forestVF}{
\tikz[planar forest ] {

\node [b] at (0.0, 0.0) {  } 
child {node [b] at (0.0, 1.0) {  }  
}
;
}}

\newcommand{\forestWF}{
\tikz[planar forest ] {

\node [b] at (0.0, 0.0) {  } 
child {node [l] at (0.0, 1.0) { $\times$ }  
}
;
}}

\newcommand{\forestXF}{
\tikz[planar forest ] {

\node [b] at (0.0, 0.0) {  } 
child {node [l] at (0.0, 1.0) { $\times$ }  
}
;
}}

\newcommand{\forestYF}{
\tikz[planar forest ] {

\node [b] at (0.0, 0.0) {  } 
child {node [l] at (0.0, 1.0) { $\times$ }  
}
;
}}

\newcommand{\forestAG}{
\tikz[planar forest ] {

\node [b] at (0.0, 0.0) {  } 
child {node [l] at (0.0, 1.0) { $\times$ }  
}
;
}}

\newcommand{\forestBG}{
\tikz[planar forest ] {

\node [b] at (0.0, 0.0) {  } 
child {node [l] at (-0.5, 1.0) { $\times$ }  
}
child {node [b] at (0.5, 1.0) {  }  
}
;
}}

\newcommand{\forestCG}{
\tikz[planar forest ] {

\node [b] at (0.0, 0.0) {  } 
child {node [b] at (-0.5, 1.0) {  }  
child {node [b] at (0.0, 1.0) {  }  
}
}
child {node [l] at (0.5, 1.0) { $\times$ }  
}
;
}}

\newcommand{\forestDG}{
\tikz[planar forest ] {

\node [b] at (0.0, 0.0) {  } 
child {node [l] at (-0.5, 1.0) { $\times$ }  
}
child {node [b] at (0.5, 1.0) {  }  
}
;
}}

\newcommand{\forestEG}{
\tikz[planar forest ] {

\node [b] at (0.0, 0.0) {  } 
;
}}

\newcommand{\forestFG}{
\tikz[planar forest ] {

\node [b] at (0.0, 0.0) {  } 
child {node [l] at (-0.5, 1.0) { $\times$ }  
}
child {node [b] at (0.5, 1.0) {  }  
}
;
}}

\newcommand{\forestGG}{
\tikz[planar forest ] {

\node [b] at (0.0, 0.0) {  } 
child {node [e] at (0.0, 1.0) {  }  
}
;
}}

\newcommand{\forestHG}{
\tikz[planar forest ] {

\node [b] at (0.0, 0.0) {  } 
child {node [l] at (-0.5, 1.0) { $\times$ }  
}
child {node [b] at (0.5, 1.0) {  }  
child {node [e] at (0.0, 1.0) {  }  
}
}
;
}}

\newcommand{\forestIG}{
\tikz[planar forest ] {

\node [b] at (0.0, 0.0) {  } 
;
}}

\newcommand{\forestJG}{
\tikz[planar forest ] {

\node [b] at (0.0, 0.0) {  } 
child {node [e] at (-1.0, 1.0) {  }  
}
child {node [l] at (0.0, 1.0) { $\times$ }  
}
child {node [b] at (1.0, 1.0) {  }  
}
;
}}

\newcommand{\forestKG}{
\tikz[planar forest ] {

\node [b] at (0.0, 0.0) {  } 
;
}}

\newcommand{\forestLG}{
\tikz[planar forest ] {

\node [b] at (0.0, 0.0) {  } 
child {node [l] at (-0.5, 1.0) { $\times$ }  
}
child {node [b] at (0.5, 1.0) {  }  
}
;
}}

\newcommand{\forestMG}{
\tikz[planar forest ] {

\node [b] at (0.0, 0.0) {  } 
child {node [l] at (0.0, 1.0) { $\times$ }  
}
;
}}

\newcommand{\forestNG}{
\tikz[planar forest ] {

\node [b] at (0.0, 0.0) {  } 
child {node [l] at (-0.5, 1.0) { $\times$ }  
}
child {node [b] at (0.5, 1.0) {  }  
child {node [b] at (0.0, 1.0) {  }  
}
}
;
}}

\newcommand{\forestOG}{
\tikz[planar forest ] {

\node [b] at (0.0, 0.0) {  } 
child {node [b] at (-1.0, 1.0) {  }  
}
child {node [l] at (0.0, 1.0) { $\times$ }  
}
child {node [b] at (1.0, 1.0) {  }  
}
;
}}

\newcommand{\forestPG}{
\tikz[planar forest ] {

\node [b] at (0.0, 0.0) {  } 
child {node [b] at (-0.5, 1.0) {  }  
}
child {node [l] at (0.5, 1.0) { $\times$ }  
}
;
}}

\newcommand{\forestQG}{
\tikz[planar forest ] {

\node [b] at (0.0, 0.0) {  } 
;
}}

\newcommand{\forestRG}{
\tikz[planar forest ] {

\node [b] at (0.0, 0.0) {  } 
child {node [b] at (-0.5, 1.0) {  }  
}
child {node [l] at (0.5, 1.0) { $\times$ }  
}
;
}}

\newcommand{\forestSG}{
\tikz[planar forest ] {

\node [b] at (0.0, 0.0) {  } 
child {node [e] at (0.0, 1.0) {  }  
}
;
}}

\newcommand{\forestTG}{
\tikz[planar forest ] {

\node [b] at (0.0, 0.0) {  } 
child {node [b] at (-0.5, 1.0) {  }  
child {node [e] at (0.0, 1.0) {  }  
}
}
child {node [l] at (0.5, 1.0) { $\times$ }  
}
;
}}

\newcommand{\forestUG}{
\tikz[planar forest ] {

\node [b] at (0.0, 0.0) {  } 
;
}}

\newcommand{\forestVG}{
\tikz[planar forest ] {

\node [b] at (0.0, 0.0) {  } 
child {node [e] at (-1.0, 1.0) {  }  
}
child {node [b] at (0.0, 1.0) {  }  
}
child {node [l] at (1.0, 1.0) { $\times$ }  
}
;
}}

\newcommand{\forestWG}{
\tikz[planar forest ] {

\node [b] at (0.0, 0.0) {  } 
;
}}

\newcommand{\forestXG}{
\tikz[planar forest ] {

\node [b] at (0.0, 0.0) {  } 
child {node [b] at (-0.5, 1.0) {  }  
}
child {node [l] at (0.5, 1.0) { $\times$ }  
}
;
}}

\newcommand{\forestYG}{
\tikz[planar forest ] {

\node [b] at (0.0, 0.0) {  } 
child {node [l] at (0.0, 1.0) { $\times$ }  
}
;
}}

\newcommand{\forestAH}{
\tikz[planar forest ] {

\node [b] at (0.0, 0.0) {  } 
child {node [b] at (-0.5, 1.0) {  }  
child {node [b] at (0.0, 1.0) {  }  
}
}
child {node [l] at (0.5, 1.0) { $\times$ }  
}
;
}}

\newcommand{\forestBH}{
\tikz[planar forest ] {

\node [b] at (0.0, 0.0) {  } 
child {node [b] at (-1.0, 1.0) {  }  
}
child {node [b] at (0.0, 1.0) {  }  
}
child {node [l] at (1.0, 1.0) { $\times$ }  
}
;
}}

\newcommand{\forestCH}{
\tikz[planar forest ] {

\node [b] at (0.0, 0.0) {  } 
child {node [l] at (0.0, 1.0) { $\times$ }  
}
;
}}

\newcommand{\forestDH}{
\tikz[planar forest ] {

\node [b] at (0.0, 0.0) {  } 
child {node [l] at (0.0, 1.0) { $\times$ }  
}
;
}}

\newcommand{\forestEH}{
\tikz[planar forest ] {

\node [b] at (0.0, 0.0) {  } 
;
}}

\newcommand{\forestFH}{
\tikz[planar forest ] {

\node [b] at (0.0, 0.0) {  } 
child {node [l] at (0.0, 1.0) { $\times$ }  
}
;
}}

\newcommand{\forestGH}{
\tikz[planar forest ] {

\node [b] at (0.0, 0.0) {  } 
child {node [l] at (0.0, 1.0) { $\times$ }  
}
;
}}

\newcommand{\forestHH}{
\tikz[planar forest ] {

\node [b] at (0.0, 0.0) {  } 
child {node [e] at (0.0, 1.0) {  }  
}
;
}}

\newcommand{\forestIH}{
\tikz[planar forest ] {

\node [b] at (0.0, 0.0) {  } 
child {node [l] at (0.0, 1.0) { $\times$ }  
}
;
}}

\newcommand{\forestJH}{
\tikz[planar forest ] {

\node [b] at (0.0, 0.0) {  } 
child {node [e] at (-0.5, 1.0) {  }  
}
child {node [l] at (0.5, 1.0) { $\times$ }  
}
;
}}

\newcommand{\forestKH}{
\tikz[planar forest ] {

\node [b] at (0.0, 0.0) {  } 
;
}}

\newcommand{\forestLH}{
\tikz[planar forest ] {

\node [b] at (0.0, 0.0) {  } 
child {node [l] at (0.0, 1.0) { $\times$ }  
}
;
}}

\newcommand{\forestMH}{
\tikz[planar forest ] {

\node [b] at (0.0, 0.0) {  } 
child {node [l] at (0.0, 1.0) { $\times$ }  
}
;
}}

\newcommand{\forestNH}{
\tikz[planar forest ] {

\node [b] at (0.0, 0.0) {  } 
child {node [l] at (0.0, 1.0) { $\times$ }  
}
;
}}

\newcommand{\forestOH}{
\tikz[planar forest ] {

\node [b] at (0.0, 0.0) {  } 
child {node [l] at (0.0, 1.0) { $\times$ }  
}
;
}}

\newcommand{\forestPH}{
\tikz[planar forest ] {

\node [b] at (0.0, 0.0) {  } 
child {node [b] at (-0.5, 1.0) {  }  
}
child {node [l] at (0.5, 1.0) { $\times$ }  
}
;
}}

\newcommand{\forestQH}{
\tikz[planar forest ] {

\node [b] at (0.0, 0.0) {  } 
child {node [l] at (0.0, 1.0) { $\times$ }  
}
;
}}

\newcommand{\forestRH}{
\tikz[planar forest ] {

\node [b] at (0.0, 0.0) {  } 
child {node [l] at (0.0, 1.0) { $\times$ }  
}
;
}}

\newcommand{\forestSH}{
\tikz[planar forest ] {

\node [b] at (0.0, 0.0) {  } 
;
}}

\newcommand{\forestTH}{
\tikz[planar forest ] {

\node [b] at (0.0, 0.0) {  } 
child {node [l] at (0.0, 1.0) { $\times$ }  
}
;
}}

\newcommand{\forestUH}{
\tikz[planar forest ] {

\node [b] at (0.0, 0.0) {  } 
child {node [l] at (0.0, 1.0) { $\times$ }  
}
;
}}

\newcommand{\forestVH}{
\tikz[planar forest ] {

\node [b] at (0.0, 0.0) {  } 
child {node [e] at (0.0, 1.0) {  }  
}
;
}}

\newcommand{\forestWH}{
\tikz[planar forest ] {

\node [b] at (0.0, 0.0) {  } 
child {node [l] at (0.0, 1.0) { $\times$ }  
}
;
}}

\newcommand{\forestXH}{
\tikz[planar forest ] {

\node [b] at (0.0, 0.0) {  } 
child {node [e] at (-0.5, 1.0) {  }  
}
child {node [l] at (0.5, 1.0) { $\times$ }  
}
;
}}

\newcommand{\forestYH}{
\tikz[planar forest ] {

\node [b] at (0.0, 0.0) {  } 
;
}}

\newcommand{\forestAI}{
\tikz[planar forest ] {

\node [b] at (0.0, 0.0) {  } 
child {node [l] at (0.0, 1.0) { $\times$ }  
}
;
}}

\newcommand{\forestBI}{
\tikz[planar forest ] {

\node [b] at (0.0, 0.0) {  } 
child {node [l] at (0.0, 1.0) { $\times$ }  
}
;
}}

\newcommand{\forestCI}{
\tikz[planar forest ] {

\node [b] at (0.0, 0.0) {  } 
child {node [l] at (0.0, 1.0) { $\times$ }  
}
;
}}

\newcommand{\forestDI}{
\tikz[planar forest ] {

\node [b] at (0.0, 0.0) {  } 
child {node [l] at (0.0, 1.0) { $\times$ }  
}
;
}}

\newcommand{\forestEI}{
\tikz[planar forest ] {

\node [b] at (0.0, 0.0) {  } 
child {node [b] at (-0.5, 1.0) {  }  
}
child {node [l] at (0.5, 1.0) { $\times$ }  
}
;
}}

\newcommand{\forestFI}{
\tikz[planar forest ] {

\node [b] at (0.0, 0.0) {  } 
;
}}

\newcommand{\forestGI}{
\tikz[planar forest ] {

\node [b] at (0.0, 0.0) {  } 
child {node [b] at (0.0, 1.0) {  }  
}
;
}}

\newcommand{\forestHI}{
\tikz[planar forest ] {

\node [b] at (0.0, 0.0) {  } 
;
}}

\newcommand{\forestII}{
\tikz[planar forest ] {

\node [b] at (0.0, 0.0) {  } 
child {node [b] at (0.0, 1.0) {  }  
child {node [e] at (0.0, 1.0) {  }  
}
}
;
}}

\newcommand{\forestJI}{
\tikz[planar forest ] {

\node [b] at (0.0, 0.0) {  } 
;
}}

\newcommand{\forestKI}{
\tikz[planar forest ] {

\node [b] at (0.0, 0.0) {  } 
child {node [e] at (-0.5, 1.0) {  }  
}
child {node [b] at (0.5, 1.0) {  }  
}
;
}}

\newcommand{\forestLI}{
\tikz[planar forest ] {

\node [b] at (0.0, 0.0) {  } 
child {node [e] at (0.0, 1.0) {  }  
}
;
}}

\newcommand{\forestMI}{
\tikz[planar forest ] {

\node [b] at (0.0, 0.0) {  } 
child {node [b] at (0.0, 1.0) {  }  
}
;
}}

\newcommand{\forestNI}{
\tikz[planar forest ] {

\node [b] at (0.0, 0.0) {  } 
child {node [l] at (0.0, 1.0) { $\times$ }  
}
;
}}

\newcommand{\forestOI}{
\tikz[planar forest ] {

\node [b] at (0.0, 0.0) {  } 
child {node [b] at (-0.5, 1.0) {  }  
}
child {node [l] at (0.5, 1.0) { $\times$ }  
}
;
}}

\newcommand{\forestPI}{
\tikz[planar forest ] {

\node [b] at (0.0, 0.0) {  } 
child {node [l] at (0.0, 1.0) { $\times$ }  
}
;
}}

\newcommand{\forestQI}{
\tikz[planar forest ] {

\node [b] at (0.0, 0.0) {  } 
child {node [l] at (-0.5, 1.0) { $\times$ }  
}
child {node [b] at (0.5, 1.0) {  }  
}
;
}}

\newcommand{\forestRI}{
\tikz[planar forest ] {

\node [b] at (0.0, 0.0) {  } 
child {node [b] at (-1.0, 1.0) {  }  
}
child {node [l] at (0.0, 1.0) { $\times$ }  
}
child {node [b] at (1.0, 1.0) {  }  
}
;
}}

\newcommand{\forestSI}{
\tikz[planar forest ] {

\node [b] at (0.0, 0.0) {  } 
child {node [l] at (-1.0, 1.0) { $\times$ }  
}
child {node [b] at (0.0, 1.0) {  }  
}
child {node [b] at (1.0, 1.0) {  }  
}
;
}}

\newcommand{\forestTI}{
\tikz[planar forest ] {

\node [b] at (0.0, 0.0) {  } 
child {node [l] at (-0.5, 1.0) { $\times$ }  
}
child {node [b] at (0.5, 1.0) {  }  
child {node [b] at (0.0, 1.0) {  }  
}
}
;
}}

\newcommand{\forestUI}{
\tikz[planar forest ] {

\node [b] at (0.0, 0.0) {  } 
child {node [b] at (-0.5, 1.0) {  }  
child {node [b] at (0.0, 1.0) {  }  
}
}
child {node [l] at (0.5, 1.0) { $\times$ }  
}
;
}}

\newcommand{\forestVI}{
\tikz[planar forest ] {

\node [b] at (0.0, 0.0) {  } 
;
}}

\newcommand{\forestWI}{
\tikz[planar forest ] {

\node [b] at (0.0, 0.0) {  } 
;
}}

\newcommand{\forestXI}{
\tikz[planar forest ] {

\node [l] at (0.0, 0.0) { $\times$ } 
;
}}

\newcommand{\forestYI}{
\tikz[planar forest ] {

\node [b] at (0.0, 0.0) {  } 
;
}}

\newcommand{\forestAJ}{
\tikz[planar forest ] {

\node [b] at (0.0, 0.0) {  } 
;
}}

\newcommand{\forestBJ}{
\tikz[planar forest ] {

\node [b] at (0.0, 0.0) {  } 
;
}}

\newcommand{\forestCJ}{
\tikz[planar forest ] {

\node [b] at (0.0, 0.0) {  } 
;
}}

\newcommand{\forestDJ}{
\tikz[planar forest ] {

\node [b] at (0.0, 0.0) {  } 
child {node [b] at (0.0, 1.0) {  }  
}
;
}}

\newcommand{\forestEJ}{
\tikz[planar forest ] {

\node [b] at (0.0, 0.0) {  } 
child {node [b] at (0.0, 1.0) {  }  
child {node [b] at (0.0, 1.0) {  }  
}
}
;
}}

\newcommand{\forestFJ}{
\tikz[planar forest ] {

\node [b] at (0.0, 0.0) {  } 
child {node [b] at (-0.5, 1.0) {  }  
}
child {node [b] at (0.5, 1.0) {  }  
}
;
}}

\newcommand{\forestGJ}{
\tikz[planar forest ] {

\node [b] at (0.0, 0.0) {  } 
;
}}

\newcommand{\forestHJ}{
\tikz[planar forest ] {

\node [b] at (0.0, 0.0) {  } 
child {node [b] at (0.0, 1.0) {  }  
}
;
}}

\newcommand{\forestIJ}{
\tikz[planar forest ] {

\node [b] at (0.0, 0.0) {  } 
child {node [b] at (-0.5, 1.0) {  }  
}
child {node [b] at (0.5, 1.0) {  }  
}
;
}}

\newcommand{\forestJJ}{
\tikz[planar forest ] {

\node [b] at (0.0, 0.0) {  } 
child {node [l] at (0.0, 1.0) { $\times$ }  
}
;
}}

\newcommand{\forestKJ}{
\tikz[planar forest ] {

\node [b] at (0.0, 0.0) {  } 
child {node [b] at (0.0, 1.0) {  }  
}
;
}}

\newcommand{\forestLJ}{
\tikz[planar forest ] {

\node [b] at (0.0, 0.0) {  } 
child {node [l] at (-0.5, 1.0) { $\times$ }  
}
child {node [b] at (0.5, 1.0) {  }  
}
;
}}

\newcommand{\forestMJ}{
\tikz[planar forest ] {

\node [b] at (0.0, 0.0) {  } 
;
}}

\newcommand{\forestNJ}{
\tikz[planar forest ] {

\node [b] at (0.0, 0.0) {  } 
child {node [l] at (0.0, 1.0) { $\times$ }  
}
;
}}

\newcommand{\forestOJ}{
\tikz[planar forest ] {

\node [b] at (0.0, 0.0) {  } 
child {node [l] at (0.0, 1.0) { $\times$ }  
}
;
}}

\newcommand{\forestPJ}{
\tikz[planar forest ] {

\node [b] at (0.0, 0.0) {  } 
;
}}

\newcommand{\forestQJ}{
\tikz[planar forest ] {

\node [b] at (0.0, 0.0) {  } 
;
}}

\newcommand{\forestRJ}{
\tikz[planar forest ] {

\node [b] at (0.0, 0.0) {  } 
child {node [b] at (0.0, 1.0) {  }  
}
;
}}

\newcommand{\forestSJ}{
\tikz[planar forest ] {

\node [b] at (0.0, 0.0) {  } 
child {node [b] at (0.0, 1.0) {  }  
child {node [b] at (-0.5, 1.0) {  }  
}
child {node [b] at (0.5, 1.0) {  }  
}
}
;
}}

\newcommand{\forestTJ}{
\tikz[planar forest ] {

\node [b] at (0.0, 0.0) {  } 
child {node [b] at (-0.5, 1.0) {  }  
}
child {node [b] at (0.5, 1.0) {  }  
child {node [b] at (0.0, 1.0) {  }  
}
}
;
}}

\newcommand{\forestUJ}{
\tikz[planar forest ] {

\node [b] at (0.0, 0.0) {  } 
child {node [l] at (0.0, 1.0) { $\times$ }  
}
;
}}

\newcommand{\forestVJ}{
\tikz[planar forest ] {

\node [b] at (0.0, 0.0) {  } 
child {node [b] at (0.0, 1.0) {  }  
child {node [b] at (0.0, 1.0) {  }  
}
}
;
}}

\newcommand{\forestWJ}{
\tikz[planar forest ] {

\node [b] at (0.0, 0.0) {  } 
child {node [l] at (-0.5, 1.0) { $\times$ }  
}
child {node [b] at (0.5, 1.0) {  }  
child {node [b] at (0.0, 1.0) {  }  
}
}
;
}}

\newcommand{\forestXJ}{
\tikz[planar forest ] {

\node [b] at (0.0, 0.0) {  } 
;
}}

\newcommand{\forestYJ}{
\tikz[planar forest ] {

\node [b] at (0.0, 0.0) {  } 
child {node [l] at (0.0, 1.0) { $\times$ }  
}
;
}}

\newcommand{\forestAK}{
\tikz[planar forest ] {

\node [b] at (0.0, 0.0) {  } 
child {node [l] at (-0.5, 1.0) { $\times$ }  
}
child {node [b] at (0.5, 1.0) {  }  
}
;
}}

\newcommand{\forestBK}{
\tikz[planar forest ] {

\node [b] at (0.0, 0.0) {  } 
;
}}

\newcommand{\forestCK}{
\tikz[planar forest ] {

\node [b] at (0.0, 0.0) {  } 
child {node [l] at (0.0, 1.0) { $\times$ }  
}
;
}}

\newcommand{\forestDK}{
\tikz[planar forest ] {

\node [b] at (0.0, 0.0) {  } 
child {node [l] at (0.0, 1.0) { $\times$ }  
}
;
}}

\newcommand{\forestEK}{
\tikz[planar forest ] {

\node [b] at (0.0, 0.0) {  } 
child {node [l] at (0.0, 1.0) { $\times$ }  
}
;
}}

\newcommand{\forestFK}{
\tikz[planar forest ] {

\node [b] at (0.0, 0.0) {  } 
;
}}

\newcommand{\forestGK}{
\tikz[planar forest ] {

\node [b] at (0.0, 0.0) {  } 
;
}}

\newcommand{\forestHK}{
\tikz[planar forest ] {

\node [b] at (0.0, 0.0) {  } 
child {node [b] at (0.0, 1.0) {  }  
child {node [b] at (0.0, 1.0) {  }  
}
}
;
}}

\newcommand{\forestIK}{
\tikz[planar forest ] {

\node [b] at (0.0, 0.0) {  } 
child {node [b] at (-0.5, 1.0) {  }  
}
child {node [b] at (0.5, 1.0) {  }  
child {node [b] at (0.0, 1.0) {  }  
}
}
;
}}

\newcommand{\forestJK}{
\tikz[planar forest ] {

\node [b] at (0.0, 0.0) {  } 
child {node [b] at (-0.5, 1.0) {  }  
child {node [b] at (0.0, 1.0) {  }  
}
}
child {node [b] at (0.5, 1.0) {  }  
}
;
}}

\newcommand{\forestKK}{
\tikz[planar forest ] {

\node [b] at (0.0, 0.0) {  } 
child {node [b] at (-1.0, 1.0) {  }  
}
child {node [b] at (0.0, 1.0) {  }  
}
child {node [b] at (1.0, 1.0) {  }  
}
;
}}

\newcommand{\forestLK}{
\tikz[planar forest ] {

\node [b] at (0.0, 0.0) {  } 
child {node [l] at (0.0, 1.0) { $\times$ }  
}
;
}}

\newcommand{\forestMK}{
\tikz[planar forest ] {

\node [b] at (0.0, 0.0) {  } 
child {node [b] at (-0.5, 1.0) {  }  
}
child {node [b] at (0.5, 1.0) {  }  
}
;
}}

\newcommand{\forestNK}{
\tikz[planar forest ] {

\node [b] at (0.0, 0.0) {  } 
child {node [b] at (0.0, 1.0) {  }  
child {node [b] at (-0.5, 1.0) {  }  
}
child {node [b] at (0.5, 1.0) {  }  
}
}
;
}}

\newcommand{\forestOK}{
\tikz[planar forest ] {

\node [b] at (0.0, 0.0) {  } 
child {node [l] at (-1.0, 1.0) { $\times$ }  
}
child {node [b] at (0.0, 1.0) {  }  
}
child {node [b] at (1.0, 1.0) {  }  
}
;
}}

\newcommand{\forestPK}{
\tikz[planar forest ] {

\node [b] at (0.0, 0.0) {  } 
;
}}

\newcommand{\forestQK}{
\tikz[planar forest ] {

\node [b] at (0.0, 0.0) {  } 
child {node [l] at (-0.5, 1.0) { $\times$ }  
}
child {node [b] at (0.5, 1.0) {  }  
}
;
}}

\newcommand{\forestRK}{
\tikz[planar forest ] {

\node [b] at (0.0, 0.0) {  } 
child {node [l] at (0.0, 1.0) { $\times$ }  
}
;
}}

\newcommand{\forestSK}{
\tikz[planar forest ] {

\node [b] at (0.0, 0.0) {  } 
;
}}

\newcommand{\forestTK}{
\tikz[planar forest ] {

\node [b] at (0.0, 0.0) {  } 
child {node [b] at (-0.5, 1.0) {  }  
}
child {node [l] at (0.5, 1.0) { $\times$ }  
}
;
}}

\newcommand{\forestUK}{
\tikz[planar forest ] {

\node [b] at (0.0, 0.0) {  } 
child {node [l] at (0.0, 1.0) { $\times$ }  
}
;
}}

\newcommand{\forestVK}{
\tikz[planar forest ] {

\node [b] at (0.0, 0.0) {  } 
;
}}

\newcommand{\forestWK}{
\tikz[planar forest ] {

\node [b] at (0.0, 0.0) {  } 
;
}}

\newcommand{\forestXK}{
\tikz[planar forest ] {

\node [b] at (0.0, 0.0) {  } 
child {node [b] at (-0.5, 1.0) {  }  
}
child {node [b] at (0.5, 1.0) {  }  
}
;
}}

\newcommand{\forestYK}{
\tikz[planar forest ] {

\node [b] at (0.0, 0.0) {  } 
child {node [l] at (0.0, 1.0) { $\times$ }  
}
;
}}

\newcommand{\forestAL}{
\tikz[planar forest ] {

\node [b] at (0.0, 0.0) {  } 
child {node [b] at (-0.5, 1.0) {  }  
}
child {node [b] at (0.5, 1.0) {  }  
}
;
}}

\newcommand{\forestBL}{
\tikz[planar forest ] {

\node [b] at (0.0, 0.0) {  } 
child {node [b] at (-0.5, 1.0) {  }  
}
child {node [l] at (0.5, 1.0) { $\times$ }  
}
;
}}

\newcommand{\forestCL}{
\tikz[planar forest ] {

\node [b] at (0.0, 0.0) {  } 
child {node [b] at (0.0, 1.0) {  }  
}
;
}}

\newcommand{\forestDL}{
\tikz[planar forest ] {

\node [b] at (0.0, 0.0) {  } 
child {node [l] at (0.0, 1.0) { $\times$ }  
}
;
}}

\newcommand{\forestEL}{
\tikz[planar forest ] {

\node [b] at (0.0, 0.0) {  } 
child {node [l] at (0.0, 1.0) { $\times$ }  
}
;
}}

\newcommand{\forestFL}{
\tikz[planar forest ] {

\node [b] at (0.0, 0.0) {  } 
child {node [b] at (0.0, 1.0) {  }  
}
;
}}

\newcommand{\forestGL}{
\tikz[planar forest ] {

\node [b] at (0.0, 0.0) {  } 
child {node [b] at (-0.5, 1.0) {  }  
child {node [b] at (0.0, 1.0) {  }  
}
}
child {node [l] at (0.5, 1.0) { $\times$ }  
}
;
}}

\newcommand{\forestHL}{
\tikz[planar forest ] {

\node [b] at (0.0, 0.0) {  } 
;
}}

\newcommand{\forestIL}{
\tikz[planar forest ] {

\node [b] at (0.0, 0.0) {  } 
child {node [l] at (0.0, 1.0) { $\times$ }  
}
;
}}

\newcommand{\forestJL}{
\tikz[planar forest ] {

\node [b] at (0.0, 0.0) {  } 
child {node [l] at (-0.5, 1.0) { $\times$ }  
}
child {node [b] at (0.5, 1.0) {  }  
}
;
}}

\newcommand{\forestKL}{
\tikz[planar forest ] {

\node [b] at (0.0, 0.0) {  } 
;
}}

\newcommand{\forestLL}{
\tikz[planar forest ] {

\node [b] at (0.0, 0.0) {  } 
child {node [l] at (0.0, 1.0) { $\times$ }  
}
;
}}

\newcommand{\forestML}{
\tikz[planar forest ] {

\node [b] at (0.0, 0.0) {  } 
child {node [l] at (0.0, 1.0) { $\times$ }  
}
;
}}

\newcommand{\forestNL}{
\tikz[planar forest ] {

\node [b] at (0.0, 0.0) {  } 
child {node [l] at (0.0, 1.0) { $\times$ }  
}
;
}}

\newcommand{\forestOL}{
\tikz[planar forest ] {

\node [b] at (0.0, 0.0) {  } 
;
}}

\newcommand{\forestPL}{
\tikz[planar forest ] {

\node [b] at (0.0, 0.0) {  } 
child {node [l] at (0.0, 1.0) { $\times$ }  
}
;
}}

\newcommand{\forestQL}{
\tikz[planar forest ] {

\node [b] at (0.0, 0.0) {  } 
;
}}

\newcommand{\forestRL}{
\tikz[planar forest ] {

\node [b] at (0.0, 0.0) {  } 
child {node [b] at (0.0, 1.0) {  }  
}
;
}}

\newcommand{\forestSL}{
\tikz[planar forest ] {

\node [b] at (0.0, 0.0) {  } 
child {node [b] at (-0.5, 1.0) {  }  
child {node [b] at (0.0, 1.0) {  }  
}
}
child {node [b] at (0.5, 1.0) {  }  
}
;
}}

\newcommand{\forestTL}{
\tikz[planar forest ] {

\node [b] at (0.0, 0.0) {  } 
child {node [b] at (-0.5, 1.0) {  }  
child {node [b] at (0.0, 1.0) {  }  
}
}
child {node [l] at (0.5, 1.0) { $\times$ }  
}
;
}}

\newcommand{\forestUL}{
\tikz[planar forest ] {

\node [b] at (0.0, 0.0) {  } 
;
}}

\newcommand{\forestVL}{
\tikz[planar forest ] {

\node [b] at (0.0, 0.0) {  } 
child {node [l] at (-1.0, 1.0) { $\times$ }  
}
child {node [b] at (0.0, 1.0) {  }  
}
child {node [b] at (1.0, 1.0) {  }  
}
;
}}

\newcommand{\forestWL}{
\tikz[planar forest ] {

\node [b] at (0.0, 0.0) {  } 
;
}}

\newcommand{\forestXL}{
\tikz[planar forest ] {

\node [b] at (0.0, 0.0) {  } 
child {node [l] at (0.0, 1.0) { $\times$ }  
}
;
}}

\newcommand{\forestYL}{
\tikz[planar forest ] {

\node [b] at (0.0, 0.0) {  } 
child {node [l] at (-0.5, 1.0) { $\times$ }  
}
child {node [b] at (0.5, 1.0) {  }  
}
;
}}

\newcommand{\forestAM}{
\tikz[planar forest ] {

\node [b] at (0.0, 0.0) {  } 
;
}}

\newcommand{\forestBM}{
\tikz[planar forest ] {

\node [b] at (0.0, 0.0) {  } 
child {node [b] at (-0.5, 1.0) {  }  
}
child {node [b] at (0.5, 1.0) {  }  
child {node [b] at (0.0, 1.0) {  }  
}
}
;
}}

\newcommand{\forestCM}{
\tikz[planar forest ] {

\node [b] at (0.0, 0.0) {  } 
child {node [l] at (-0.5, 1.0) { $\times$ }  
}
child {node [b] at (0.5, 1.0) {  }  
child {node [b] at (0.0, 1.0) {  }  
}
}
;
}}

\newcommand{\forestDM}{
\tikz[planar forest ] {

\node [b] at (0.0, 0.0) {  } 
;
}}

\newcommand{\forestEM}{
\tikz[planar forest ] {

\node [b] at (0.0, 0.0) {  } 
child {node [b] at (-1.0, 1.0) {  }  
}
child {node [l] at (0.0, 1.0) { $\times$ }  
}
child {node [b] at (1.0, 1.0) {  }  
}
;
}}

\newcommand{\forestFM}{
\tikz[planar forest ] {

\node [b] at (0.0, 0.0) {  } 
;
}}

\newcommand{\forestGM}{
\tikz[planar forest ] {

\node [b] at (0.0, 0.0) {  } 
child {node [l] at (0.0, 1.0) { $\times$ }  
}
;
}}

\newcommand{\forestHM}{
\tikz[planar forest ] {

\node [b] at (0.0, 0.0) {  } 
child {node [b] at (-0.5, 1.0) {  }  
}
child {node [l] at (0.5, 1.0) { $\times$ }  
}
;
}}

\newcommand{\forestIM}{
\tikz[planar forest ] {

\node [b] at (0.0, 0.0) {  } 
;
}}

\newcommand{\forestJM}{
\tikz[planar forest ] {

\node [b] at (0.0, 0.0) {  } 
;
}}

\newcommand{\forestKM}{
\tikz[planar forest ] {

\node [b] at (0.0, 0.0) {  } 
child {node [b] at (0.0, 1.0) {  }  
child {node [b] at (0.0, 1.0) {  }  
}
}
;
}}

\newcommand{\forestLM}{
\tikz[planar forest ] {

\node [b] at (0.0, 0.0) {  } 
;
}}

\newcommand{\forestMM}{
\tikz[planar forest ] {

\node [b] at (0.0, 0.0) {  } 
child {node [b] at (-0.5, 1.0) {  }  
}
child {node [b] at (0.5, 1.0) {  }  
}
;
}}

\newcommand{\forestNM}{
\tikz[planar forest ] {

\node [b] at (0.0, 0.0) {  } 
child {node [l] at (0.0, 1.0) { $\times$ }  
}
;
}}

\newcommand{\forestOM}{
\tikz[planar forest ] {

\node [b] at (0.0, 0.0) {  } 
;
}}

\newcommand{\forestPM}{
\tikz[planar forest ] {

\node [b] at (0.0, 0.0) {  } 
child {node [b] at (0.0, 1.0) {  }  
}
;
}}

\newcommand{\forestQM}{
\tikz[planar forest ] {

\node [b] at (0.0, 0.0) {  } 
;
}}

\newcommand{\forestRM}{
\tikz[planar forest ] {

\node [b] at (0.0, 0.0) {  } 
;
}}

\newcommand{\forestSM}{
\tikz[planar forest ] {

\node [b] at (0.0, 0.0) {  } 
child {node [b] at (0.0, 1.0) {  }  
}
;
}}

\newcommand{\forestTM}{
\tikz[planar forest ] {

\node [b] at (0.0, 0.0) {  } 
;
}}

\newcommand{\forestUM}{
\tikz[planar forest ] {

\node [b] at (0.0, 0.0) {  } 
child {node [b] at (0.0, 1.0) {  }  
}
;
}}

\newcommand{\forestVM}{
\tikz[planar forest ] {

\node [b] at (0.0, 0.0) {  } 
child {node [b] at (0.0, 1.0) {  }  
child {node [b] at (0.0, 1.0) {  }  
}
}
;
}}

\newcommand{\forestWM}{
\tikz[planar forest ] {

\node [b] at (0.0, 0.0) {  } 
child {node [l] at (-0.5, 1.0) { $\times$ }  
}
child {node [b] at (0.5, 1.0) {  }  
}
;
}}

\newcommand{\forestXM}{
\tikz[planar forest ] {

\node [b] at (0.0, 0.0) {  } 
;
}}

\newcommand{\forestYM}{
\tikz[planar forest ] {

\node [b] at (0.0, 0.0) {  } 
child {node [l] at (0.0, 1.0) { $\times$ }  
}
;
}}

\newcommand{\forestAN}{
\tikz[planar forest ] {

\node [b] at (0.0, 0.0) {  } 
child {node [l] at (0.0, 1.0) { $\times$ }  
}
;
}}

\newcommand{\forestBN}{
\tikz[planar forest ] {

\node [b] at (0.0, 0.0) {  } 
;
}}

\newcommand{\forestCN}{
\tikz[planar forest ] {

\node [b] at (0.0, 0.0) {  } 
;
}}

\newcommand{\forestDN}{
\tikz[planar forest ] {

\node [b] at (0.0, 0.0) {  } 
child {node [b] at (0.0, 1.0) {  }  
}
;
}}

\newcommand{\forestEN}{
\tikz[planar forest ] {

\node [b] at (0.0, 0.0) {  } 
;
}}

\newcommand{\forestFN}{
\tikz[planar forest ] {

\node [b] at (0.0, 0.0) {  } 
child {node [b] at (0.0, 1.0) {  }  
child {node [b] at (0.0, 1.0) {  }  
}
}
;
}}

\newcommand{\forestGN}{
\tikz[planar forest ] {

\node [b] at (0.0, 0.0) {  } 
child {node [l] at (-0.5, 1.0) { $\times$ }  
}
child {node [b] at (0.5, 1.0) {  }  
}
;
}}

\newcommand{\forestHN}{
\tikz[planar forest ] {

\node [b] at (0.0, 0.0) {  } 
;
}}

\newcommand{\forestIN}{
\tikz[planar forest ] {

\node [b] at (0.0, 0.0) {  } 
child {node [l] at (0.0, 1.0) { $\times$ }  
}
;
}}

\newcommand{\forestJN}{
\tikz[planar forest ] {

\node [b] at (0.0, 0.0) {  } 
child {node [l] at (0.0, 1.0) { $\times$ }  
}
;
}}

\newcommand{\forestKN}{
\tikz[planar forest ] {

\node [b] at (0.0, 0.0) {  } 
;
}}

\newcommand{\forestLN}{
\tikz[planar forest ] {

\node [b] at (0.0, 0.0) {  } 
;
}}

\newcommand{\forestMN}{
\tikz[planar forest ] {

\node [b] at (0.0, 0.0) {  } 
child {node [b] at (0.0, 1.0) {  }  
}
;
}}

\newcommand{\forestNN}{
\tikz[planar forest ] {

\node [b] at (0.0, 0.0) {  } 
child {node [b] at (0.0, 1.0) {  }  
child {node [b] at (0.0, 1.0) {  }  
child {node [b] at (0.0, 1.0) {  }  
}
}
}
;
}}

\newcommand{\forestON}{
\tikz[planar forest ] {

\node [b] at (0.0, 0.0) {  } 
child {node [b] at (-0.5, 1.0) {  }  
}
child {node [b] at (0.5, 1.0) {  }  
child {node [b] at (0.0, 1.0) {  }  
}
}
;
}}

\newcommand{\forestPN}{
\tikz[planar forest ] {

\node [b] at (0.0, 0.0) {  } 
child {node [b] at (-0.5, 1.0) {  }  
child {node [b] at (0.0, 1.0) {  }  
}
}
child {node [b] at (0.5, 1.0) {  }  
}
;
}}

\newcommand{\forestQN}{
\tikz[planar forest ] {

\node [b] at (0.0, 0.0) {  } 
child {node [l] at (-1.0, 1.0) { $\times$ }  
}
child {node [b] at (0.0, 1.0) {  }  
}
child {node [b] at (1.0, 1.0) {  }  
}
;
}}

\newcommand{\forestRN}{
\tikz[planar forest ] {

\node [b] at (0.0, 0.0) {  } 
;
}}

\newcommand{\forestSN}{
\tikz[planar forest ] {

\node [b] at (0.0, 0.0) {  } 
child {node [b] at (-1.0, 1.0) {  }  
}
child {node [l] at (0.0, 1.0) { $\times$ }  
}
child {node [b] at (1.0, 1.0) {  }  
}
;
}}

\newcommand{\forestTN}{
\tikz[planar forest ] {

\node [b] at (0.0, 0.0) {  } 
;
}}

\newcommand{\forestUN}{
\tikz[planar forest ] {

\node [b] at (0.0, 0.0) {  } 
child {node [l] at (0.0, 1.0) { $\times$ }  
}
;
}}

\newcommand{\forestVN}{
\tikz[planar forest ] {

\node [b] at (0.0, 0.0) {  } 
child {node [l] at (-0.5, 1.0) { $\times$ }  
}
child {node [b] at (0.5, 1.0) {  }  
}
;
}}

\newcommand{\forestWN}{
\tikz[planar forest ] {

\node [b] at (0.0, 0.0) {  } 
;
}}

\newcommand{\forestXN}{
\tikz[planar forest ] {

\node [b] at (0.0, 0.0) {  } 
child {node [l] at (0.0, 1.0) { $\times$ }  
}
;
}}

\newcommand{\forestYN}{
\tikz[planar forest ] {

\node [b] at (0.0, 0.0) {  } 
child {node [b] at (-0.5, 1.0) {  }  
}
child {node [l] at (0.5, 1.0) { $\times$ }  
}
;
}}

\newcommand{\forestAO}{
\tikz[planar forest ] {

\node [b] at (0.0, 0.0) {  } 
;
}}

\newcommand{\forestBO}{
\tikz[planar forest ] {

\node [b] at (0.0, 0.0) {  } 
child {node [l] at (0.0, 1.0) { $\times$ }  
}
;
}}

\newcommand{\forestCO}{
\tikz[planar forest ] {

\node [b] at (0.0, 0.0) {  } 
child {node [l] at (0.0, 1.0) { $\times$ }  
}
;
}}

\newcommand{\forestDO}{
\tikz[planar forest ] {

\node [b] at (0.0, 0.0) {  } 
child {node [l] at (0.0, 1.0) { $\times$ }  
}
;
}}

\newcommand{\forestEO}{
\tikz[planar forest ] {

\node [b] at (0.0, 0.0) {  } 
;
}}

\newcommand{\forestFO}{
\tikz[planar forest ] {

\node [b] at (0.0, 0.0) {  } 
;
}}

\newcommand{\forestGO}{
\tikz[planar forest ] {

\node [b] at (0.0, 0.0) {  } 
child {node [b] at (0.0, 1.0) {  }  
child {node [b] at (0.0, 1.0) {  }  
}
}
;
}}

\newcommand{\forestHO}{
\tikz[planar forest ] {

\node [b] at (0.0, 0.0) {  } 
;
}}

\newcommand{\forestIO}{
\tikz[planar forest ] {

\node [b] at (0.0, 0.0) {  } 
;
}}

\newcommand{\forestJO}{
\tikz[planar forest ] {

\node [b] at (0.0, 0.0) {  } 
child {node [b] at (0.0, 1.0) {  }  
}
;
}}

\newcommand{\forestKO}{
\tikz[planar forest ] {

\node [b] at (0.0, 0.0) {  } 
child {node [b] at (0.0, 1.0) {  }  
child {node [b] at (0.0, 1.0) {  }  
child {node [b] at (0.0, 1.0) {  }  
}
}
}
;
}}

\newcommand{\forestLO}{
\tikz[planar forest ] {

\node [b] at (0.0, 0.0) {  } 
child {node [l] at (0.0, 1.0) { $\times$ }  
}
;
}}

\newcommand{\forestMO}{
\tikz[planar forest ] {

\node [b] at (0.0, 0.0) {  } 
child {node [l] at (0.0, 1.0) { $\times$ }  
}
;
}}

\newcommand{\forestNO}{
\tikz[planar forest ] {

\node [b] at (0.0, 0.0) {  } 
child {node [l] at (0.0, 1.0) { $\times$ }  
}
;
}}

\newcommand{\forestOO}{
\tikz[planar forest ] {

\node [b] at (0.0, 0.0) {  } 
child {node [l] at (0.0, 1.0) { $\times$ }  
}
;
}}

\newcommand{\forestPO}{
\tikz[planar forest ] {

\node [b] at (0.0, 0.0) {  } 
child {node [l] at (0.0, 1.0) { $\times$ }  
}
;
}}

\newcommand{\forestQO}{
\tikz[planar forest ] {

\node [b] at (0.0, 0.0) {  } 
child {node [l] at (0.0, 1.0) { $\times$ }  
}
;
}}

\newcommand{\forestRO}{
\tikz[planar forest ] {

\node [b] at (0.0, 0.0) {  } 
child {node [l] at (-0.5, 1.0) { $\times$ }  
}
child {node [b] at (0.5, 1.0) {  }  
}
;
}}

\newcommand{\forestSO}{
\tikz[planar forest ] {

\node [b] at (0.0, 0.0) {  } 
child {node [l] at (0.0, 1.0) { $\times$ }  
}
;
}}

\newcommand{\forestTO}{
\tikz[planar forest ] {

\node [b] at (0.0, 0.0) {  } 
child {node [l] at (-0.5, 1.0) { $\times$ }  
}
child {node [b] at (0.5, 1.0) {  }  
child {node [b] at (0.0, 1.0) {  }  
}
}
;
}}

\newcommand{\forestUO}{
\tikz[planar forest ] {

\node [b] at (0.0, 0.0) {  } 
child {node [l] at (-0.5, 1.0) { $\times$ }  
}
child {node [b] at (0.5, 1.0) {  }  
child {node [b] at (0.0, 1.0) {  }  
child {node [b] at (0.0, 1.0) {  }  
}
}
}
;
}}

\newcommand{\forestVO}{
\tikz[planar forest ] {

\node [b] at (0.0, 0.0) {  } 
child {node [b] at (0.0, 1.0) {  }  
child {node [b] at (-0.5, 1.0) {  }  
}
child {node [b] at (0.5, 1.0) {  }  
}
}
;
}}

\newcommand{\forestWO}{
\tikz[planar forest ] {

\node [b] at (0.0, 0.0) {  } 
child {node [l] at (0.0, 1.0) { $\times$ }  
}
;
}}

\newcommand{\forestXO}{
\tikz[planar forest ] {

\node [b] at (0.0, 0.0) {  } 
child {node [b] at (-0.5, 1.0) {  }  
}
child {node [l] at (0.5, 1.0) { $\times$ }  
}
;
}}

\newcommand{\forestYO}{
\tikz[planar forest ] {

\node [b] at (0.0, 0.0) {  } 
child {node [l] at (0.0, 1.0) { $\times$ }  
}
;
}}

\newcommand{\forestAP}{
\tikz[planar forest ] {

\node [b] at (0.0, 0.0) {  } 
child {node [l] at (0.0, 1.0) { $\times$ }  
}
;
}}

\newcommand{\forestBP}{
\tikz[planar forest ] {

\node [b] at (0.0, 0.0) {  } 
child {node [l] at (-0.5, 1.0) { $\times$ }  
}
child {node [b] at (0.5, 1.0) {  }  
}
;
}}

\newcommand{\forestCP}{
\tikz[planar forest ] {

\node [b] at (0.0, 0.0) {  } 
child {node [l] at (0.0, 1.0) { $\times$ }  
}
;
}}

\newcommand{\forestDP}{
\tikz[planar forest ] {

\node [b] at (0.0, 0.0) {  } 
child {node [l] at (0.0, 1.0) { $\times$ }  
}
;
}}

\newcommand{\forestEP}{
\tikz[planar forest ] {

\node [b] at (0.0, 0.0) {  } 
child {node [l] at (-1.0, 1.0) { $\times$ }  
}
child {node [b] at (0.0, 1.0) {  }  
}
child {node [b] at (1.0, 1.0) {  }  
}
;
}}

\newcommand{\forestFP}{
\tikz[planar forest ] {

\node [b] at (0.0, 0.0) {  } 
child {node [l] at (-0.5, 1.0) { $\times$ }  
}
child {node [b] at (0.5, 1.0) {  }  
child {node [b] at (-0.5, 1.0) {  }  
}
child {node [b] at (0.5, 1.0) {  }  
}
}
;
}}

\newcommand{\forestGP}{
\tikz[planar forest ] {

\node [b] at (0.0, 0.0) {  } 
child {node [b] at (-0.5, 1.0) {  }  
}
child {node [b] at (0.5, 1.0) {  }  
child {node [b] at (0.0, 1.0) {  }  
}
}
;
}}

\newcommand{\forestHP}{
\tikz[planar forest ] {

\node [b] at (0.0, 0.0) {  } 
child {node [b] at (-0.5, 1.0) {  }  
}
child {node [l] at (0.5, 1.0) { $\times$ }  
}
;
}}

\newcommand{\forestIP}{
\tikz[planar forest ] {

\node [b] at (0.0, 0.0) {  } 
child {node [l] at (0.0, 1.0) { $\times$ }  
}
;
}}

\newcommand{\forestJP}{
\tikz[planar forest ] {

\node [b] at (0.0, 0.0) {  } 
child {node [l] at (0.0, 1.0) { $\times$ }  
}
;
}}

\newcommand{\forestKP}{
\tikz[planar forest ] {

\node [b] at (0.0, 0.0) {  } 
child {node [b] at (-0.5, 1.0) {  }  
}
child {node [l] at (0.5, 1.0) { $\times$ }  
}
;
}}

\newcommand{\forestLP}{
\tikz[planar forest ] {

\node [b] at (0.0, 0.0) {  } 
child {node [l] at (-0.5, 1.0) { $\times$ }  
}
child {node [b] at (0.5, 1.0) {  }  
}
;
}}

\newcommand{\forestMP}{
\tikz[planar forest ] {

\node [b] at (0.0, 0.0) {  } 
child {node [l] at (-0.5, 1.0) { $\times$ }  
}
child {node [b] at (0.5, 1.0) {  }  
child {node [b] at (0.0, 1.0) {  }  
}
}
;
}}

\newcommand{\forestNP}{
\tikz[planar forest ] {

\node [b] at (0.0, 0.0) {  } 
child {node [l] at (0.0, 1.0) { $\times$ }  
}
;
}}

\newcommand{\forestOP}{
\tikz[planar forest ] {

\node [b] at (0.0, 0.0) {  } 
child {node [l] at (-1.0, 1.0) { $\times$ }  
}
child {node [b] at (0.0, 1.0) {  }  
}
child {node [b] at (1.0, 1.0) {  }  
child {node [b] at (0.0, 1.0) {  }  
}
}
;
}}

\newcommand{\forestPP}{
\tikz[planar forest ] {

\node [b] at (0.0, 0.0) {  } 
child {node [b] at (-0.5, 1.0) {  }  
child {node [b] at (0.0, 1.0) {  }  
}
}
child {node [b] at (0.5, 1.0) {  }  
}
;
}}

\newcommand{\forestQP}{
\tikz[planar forest ] {

\node [b] at (0.0, 0.0) {  } 
child {node [b] at (-0.5, 1.0) {  }  
child {node [b] at (0.0, 1.0) {  }  
}
}
child {node [l] at (0.5, 1.0) { $\times$ }  
}
;
}}

\newcommand{\forestRP}{
\tikz[planar forest ] {

\node [b] at (0.0, 0.0) {  } 
child {node [l] at (0.0, 1.0) { $\times$ }  
}
;
}}

\newcommand{\forestSP}{
\tikz[planar forest ] {

\node [b] at (0.0, 0.0) {  } 
child {node [l] at (-0.5, 1.0) { $\times$ }  
}
child {node [b] at (0.5, 1.0) {  }  
}
;
}}

\newcommand{\forestTP}{
\tikz[planar forest ] {

\node [b] at (0.0, 0.0) {  } 
child {node [l] at (0.0, 1.0) { $\times$ }  
}
;
}}

\newcommand{\forestUP}{
\tikz[planar forest ] {

\node [b] at (0.0, 0.0) {  } 
child {node [l] at (0.0, 1.0) { $\times$ }  
}
;
}}

\newcommand{\forestVP}{
\tikz[planar forest ] {

\node [b] at (0.0, 0.0) {  } 
child {node [l] at (-0.5, 1.0) { $\times$ }  
}
child {node [b] at (0.5, 1.0) {  }  
}
;
}}

\newcommand{\forestWP}{
\tikz[planar forest ] {

\node [b] at (0.0, 0.0) {  } 
child {node [l] at (-0.5, 1.0) { $\times$ }  
}
child {node [b] at (0.5, 1.0) {  }  
}
;
}}

\newcommand{\forestXP}{
\tikz[planar forest ] {

\node [b] at (0.0, 0.0) {  } 
child {node [l] at (-1.0, 1.0) { $\times$ }  
}
child {node [b] at (0.0, 1.0) {  }  
child {node [b] at (0.0, 1.0) {  }  
}
}
child {node [b] at (1.0, 1.0) {  }  
}
;
}}

\newcommand{\forestYP}{
\tikz[planar forest ] {

\node [b] at (0.0, 0.0) {  } 
child {node [b] at (-1.0, 1.0) {  }  
}
child {node [b] at (0.0, 1.0) {  }  
}
child {node [b] at (1.0, 1.0) {  }  
}
;
}}

\newcommand{\forestAQ}{
\tikz[planar forest ] {

\node [b] at (0.0, 0.0) {  } 
child {node [b] at (-1.0, 1.0) {  }  
}
child {node [b] at (0.0, 1.0) {  }  
}
child {node [l] at (1.0, 1.0) { $\times$ }  
}
;
}}

\newcommand{\forestBQ}{
\tikz[planar forest ] {

\node [b] at (0.0, 0.0) {  } 
child {node [l] at (0.0, 1.0) { $\times$ }  
}
;
}}

\newcommand{\forestCQ}{
\tikz[planar forest ] {

\node [b] at (0.0, 0.0) {  } 
child {node [b] at (-1.0, 1.0) {  }  
}
child {node [l] at (0.0, 1.0) { $\times$ }  
}
child {node [b] at (1.0, 1.0) {  }  
}
;
}}

\newcommand{\forestDQ}{
\tikz[planar forest ] {

\node [b] at (0.0, 0.0) {  } 
child {node [l] at (0.0, 1.0) { $\times$ }  
}
;
}}

\newcommand{\forestEQ}{
\tikz[planar forest ] {

\node [b] at (0.0, 0.0) {  } 
child {node [l] at (-1.0, 1.0) { $\times$ }  
}
child {node [b] at (0.0, 1.0) {  }  
}
child {node [b] at (1.0, 1.0) {  }  
}
;
}}

\newcommand{\forestFQ}{
\tikz[planar forest ] {

\node [b] at (0.0, 0.0) {  } 
child {node [l] at (0.0, 1.0) { $\times$ }  
}
;
}}

\newcommand{\forestGQ}{
\tikz[planar forest ] {

\node [b] at (0.0, 0.0) {  } 
child {node [l] at (-1.5, 1.0) { $\times$ }  
}
child {node [b] at (-0.5, 1.0) {  }  
}
child {node [b] at (0.5, 1.0) {  }  
}
child {node [b] at (1.5, 1.0) {  }  
}
;
}}

\newcommand{\forestHQ}{
\tikz[planar forest ] {

\node [b] at (0.0, 0.0) {  } 
child {node [l] at (0.0, 1.0) { $\times$ }  
}
;
}}

\newcommand{\forestIQ}{
\tikz[planar forest ] {

\node [b] at (0.0, 0.0) {  } 
child {node [b] at (0.0, 1.0) {  }  
child {node [b] at (0.0, 1.0) {  }  
}
}
;
}}

\newcommand{\forestJQ}{
\tikz[planar forest ] {

\node [b] at (0.0, 0.0) {  } 
child {node [l] at (0.0, 1.0) { $\times$ }  
}
;
}}

\newcommand{\forestKQ}{
\tikz[planar forest ] {

\node [b] at (0.0, 0.0) {  } 
child {node [l] at (0.0, 1.0) { $\times$ }  
}
;
}}

\newcommand{\forestLQ}{
\tikz[planar forest ] {

\node [b] at (0.0, 0.0) {  } 
child {node [l] at (0.0, 1.0) { $\times$ }  
}
;
}}

\newcommand{\forestMQ}{
\tikz[planar forest ] {

\node [b] at (0.0, 0.0) {  } 
child {node [l] at (0.0, 1.0) { $\times$ }  
}
;
}}

\newcommand{\forestNQ}{
\tikz[planar forest ] {

\node [b] at (0.0, 0.0) {  } 
child {node [l] at (0.0, 1.0) { $\times$ }  
}
;
}}

\newcommand{\forestOQ}{
\tikz[planar forest ] {

\node [b] at (0.0, 0.0) {  } 
child {node [l] at (0.0, 1.0) { $\times$ }  
}
;
}}

\newcommand{\forestPQ}{
\tikz[planar forest ] {

\node [b] at (0.0, 0.0) {  } 
child {node [l] at (-0.5, 1.0) { $\times$ }  
}
child {node [b] at (0.5, 1.0) {  }  
}
;
}}

\newcommand{\forestQQ}{
\tikz[planar forest ] {

\node [b] at (0.0, 0.0) {  } 
child {node [l] at (0.0, 1.0) { $\times$ }  
}
;
}}

\newcommand{\forestRQ}{
\tikz[planar forest ] {

\node [b] at (0.0, 0.0) {  } 
child {node [l] at (-0.5, 1.0) { $\times$ }  
}
child {node [b] at (0.5, 1.0) {  }  
child {node [b] at (0.0, 1.0) {  }  
}
}
;
}}

\newcommand{\forestSQ}{
\tikz[planar forest ] {

\node [b] at (0.0, 0.0) {  } 
child {node [b] at (-0.5, 1.0) {  }  
child {node [b] at (0.0, 1.0) {  }  
child {node [b] at (0.0, 1.0) {  }  
}
}
}
child {node [l] at (0.5, 1.0) { $\times$ }  
}
;
}}

\newcommand{\forestTQ}{
\tikz[planar forest ] {

\node [b] at (0.0, 0.0) {  } 
child {node [l] at (0.0, 1.0) { $\times$ }  
}
;
}}

\newcommand{\forestUQ}{
\tikz[planar forest ] {

\node [b] at (0.0, 0.0) {  } 
child {node [b] at (-0.5, 1.0) {  }  
}
child {node [b] at (0.5, 1.0) {  }  
}
;
}}

\newcommand{\forestVQ}{
\tikz[planar forest ] {

\node [b] at (0.0, 0.0) {  } 
child {node [l] at (0.0, 1.0) { $\times$ }  
}
;
}}

\newcommand{\forestWQ}{
\tikz[planar forest ] {

\node [b] at (0.0, 0.0) {  } 
child {node [b] at (-0.5, 1.0) {  }  
}
child {node [l] at (0.5, 1.0) { $\times$ }  
}
;
}}

\newcommand{\forestXQ}{
\tikz[planar forest ] {

\node [b] at (0.0, 0.0) {  } 
child {node [l] at (0.0, 1.0) { $\times$ }  
}
;
}}

\newcommand{\forestYQ}{
\tikz[planar forest ] {

\node [b] at (0.0, 0.0) {  } 
child {node [l] at (0.0, 1.0) { $\times$ }  
}
;
}}

\newcommand{\forestAR}{
\tikz[planar forest ] {

\node [b] at (0.0, 0.0) {  } 
child {node [l] at (-0.5, 1.0) { $\times$ }  
}
child {node [b] at (0.5, 1.0) {  }  
}
;
}}

\newcommand{\forestBR}{
\tikz[planar forest ] {

\node [b] at (0.0, 0.0) {  } 
child {node [l] at (0.0, 1.0) { $\times$ }  
}
;
}}

\newcommand{\forestCR}{
\tikz[planar forest ] {

\node [b] at (0.0, 0.0) {  } 
child {node [l] at (0.0, 1.0) { $\times$ }  
}
;
}}

\newcommand{\forestDR}{
\tikz[planar forest ] {

\node [b] at (0.0, 0.0) {  } 
child {node [l] at (-1.0, 1.0) { $\times$ }  
}
child {node [b] at (0.0, 1.0) {  }  
}
child {node [b] at (1.0, 1.0) {  }  
}
;
}}

\newcommand{\forestER}{
\tikz[planar forest ] {

\node [b] at (0.0, 0.0) {  } 
child {node [b] at (-0.5, 1.0) {  }  
child {node [b] at (-0.5, 1.0) {  }  
}
child {node [b] at (0.5, 1.0) {  }  
}
}
child {node [l] at (0.5, 1.0) { $\times$ }  
}
;
}}

\newcommand{\forestFR}{
\tikz[planar forest ] {

\node [b] at (0.0, 0.0) {  } 
child {node [l] at (-0.5, 1.0) { $\times$ }  
}
child {node [b] at (0.5, 1.0) {  }  
}
;
}}

\newcommand{\forestGR}{
\tikz[planar forest ] {

\node [b] at (0.0, 0.0) {  } 
child {node [b] at (0.0, 1.0) {  }  
}
;
}}

\newcommand{\forestHR}{
\tikz[planar forest ] {

\node [b] at (0.0, 0.0) {  } 
child {node [l] at (-0.5, 1.0) { $\times$ }  
}
child {node [b] at (0.5, 1.0) {  }  
}
;
}}

\newcommand{\forestIR}{
\tikz[planar forest ] {

\node [b] at (0.0, 0.0) {  } 
child {node [l] at (0.0, 1.0) { $\times$ }  
}
;
}}

\newcommand{\forestJR}{
\tikz[planar forest ] {

\node [b] at (0.0, 0.0) {  } 
child {node [l] at (0.0, 1.0) { $\times$ }  
}
;
}}

\newcommand{\forestKR}{
\tikz[planar forest ] {

\node [b] at (0.0, 0.0) {  } 
child {node [l] at (-0.5, 1.0) { $\times$ }  
}
child {node [b] at (0.5, 1.0) {  }  
}
;
}}

\newcommand{\forestLR}{
\tikz[planar forest ] {

\node [b] at (0.0, 0.0) {  } 
child {node [l] at (-0.5, 1.0) { $\times$ }  
}
child {node [b] at (0.5, 1.0) {  }  
}
;
}}

\newcommand{\forestMR}{
\tikz[planar forest ] {

\node [b] at (0.0, 0.0) {  } 
child {node [l] at (-0.5, 1.0) { $\times$ }  
}
child {node [b] at (0.5, 1.0) {  }  
child {node [b] at (0.0, 1.0) {  }  
child {node [b] at (0.0, 1.0) {  }  
}
}
}
;
}}

\newcommand{\forestNR}{
\tikz[planar forest ] {

\node [b] at (0.0, 0.0) {  } 
child {node [b] at (-1.0, 1.0) {  }  
child {node [b] at (0.0, 1.0) {  }  
}
}
child {node [l] at (0.0, 1.0) { $\times$ }  
}
child {node [b] at (1.0, 1.0) {  }  
}
;
}}

\newcommand{\forestOR}{
\tikz[planar forest ] {

\node [b] at (0.0, 0.0) {  } 
child {node [b] at (-0.5, 1.0) {  }  
}
child {node [l] at (0.5, 1.0) { $\times$ }  
}
;
}}

\newcommand{\forestPR}{
\tikz[planar forest ] {

\node [b] at (0.0, 0.0) {  } 
child {node [b] at (0.0, 1.0) {  }  
}
;
}}

\newcommand{\forestQR}{
\tikz[planar forest ] {

\node [b] at (0.0, 0.0) {  } 
child {node [b] at (-0.5, 1.0) {  }  
}
child {node [l] at (0.5, 1.0) { $\times$ }  
}
;
}}

\newcommand{\forestRR}{
\tikz[planar forest ] {

\node [b] at (0.0, 0.0) {  } 
child {node [l] at (0.0, 1.0) { $\times$ }  
}
;
}}

\newcommand{\forestSR}{
\tikz[planar forest ] {

\node [b] at (0.0, 0.0) {  } 
child {node [l] at (0.0, 1.0) { $\times$ }  
}
;
}}

\newcommand{\forestTR}{
\tikz[planar forest ] {

\node [b] at (0.0, 0.0) {  } 
child {node [b] at (-0.5, 1.0) {  }  
}
child {node [l] at (0.5, 1.0) { $\times$ }  
}
;
}}

\newcommand{\forestUR}{
\tikz[planar forest ] {

\node [b] at (0.0, 0.0) {  } 
child {node [l] at (-0.5, 1.0) { $\times$ }  
}
child {node [b] at (0.5, 1.0) {  }  
}
;
}}

\newcommand{\forestVR}{
\tikz[planar forest ] {

\node [b] at (0.0, 0.0) {  } 
child {node [b] at (-0.5, 1.0) {  }  
child {node [b] at (0.0, 1.0) {  }  
child {node [b] at (0.0, 1.0) {  }  
}
}
}
child {node [l] at (0.5, 1.0) { $\times$ }  
}
;
}}

\newcommand{\forestWR}{
\tikz[planar forest ] {

\node [b] at (0.0, 0.0) {  } 
child {node [b] at (-1.0, 1.0) {  }  
child {node [b] at (0.0, 1.0) {  }  
}
}
child {node [b] at (0.0, 1.0) {  }  
}
child {node [l] at (1.0, 1.0) { $\times$ }  
}
;
}}

\newcommand{\forestXR}{
\tikz[planar forest ] {

\node [b] at (0.0, 0.0) {  } 
child {node [l] at (0.0, 1.0) { $\times$ }  
}
;
}}

\newcommand{\forestYR}{
\tikz[planar forest ] {

\node [b] at (0.0, 0.0) {  } 
child {node [l] at (0.0, 1.0) { $\times$ }  
}
;
}}

\newcommand{\forestAS}{
\tikz[planar forest ] {

\node [b] at (0.0, 0.0) {  } 
child {node [b] at (0.0, 1.0) {  }  
}
;
}}

\newcommand{\forestBS}{
\tikz[planar forest ] {

\node [b] at (0.0, 0.0) {  } 
child {node [l] at (0.0, 1.0) { $\times$ }  
}
;
}}

\newcommand{\forestCS}{
\tikz[planar forest ] {

\node [b] at (0.0, 0.0) {  } 
child {node [l] at (0.0, 1.0) { $\times$ }  
}
;
}}

\newcommand{\forestDS}{
\tikz[planar forest ] {

\node [b] at (0.0, 0.0) {  } 
child {node [l] at (0.0, 1.0) { $\times$ }  
}
;
}}

\newcommand{\forestES}{
\tikz[planar forest ] {

\node [b] at (0.0, 0.0) {  } 
child {node [l] at (0.0, 1.0) { $\times$ }  
}
;
}}

\newcommand{\forestFS}{
\tikz[planar forest ] {

\node [b] at (0.0, 0.0) {  } 
child {node [l] at (0.0, 1.0) { $\times$ }  
}
;
}}

\newcommand{\forestGS}{
\tikz[planar forest ] {

\node [b] at (0.0, 0.0) {  } 
child {node [l] at (0.0, 1.0) { $\times$ }  
}
;
}}

\newcommand{\forestHS}{
\tikz[planar forest ] {

\node [b] at (0.0, 0.0) {  } 
child {node [l] at (-0.5, 1.0) { $\times$ }  
}
child {node [b] at (0.5, 1.0) {  }  
}
;
}}

\newcommand{\forestIS}{
\tikz[planar forest ] {

\node [b] at (0.0, 0.0) {  } 
child {node [l] at (0.0, 1.0) { $\times$ }  
}
;
}}

\newcommand{\forestJS}{
\tikz[planar forest ] {

\node [b] at (0.0, 0.0) {  } 
child {node [b] at (-0.5, 1.0) {  }  
child {node [b] at (0.0, 1.0) {  }  
}
}
child {node [l] at (0.5, 1.0) { $\times$ }  
}
;
}}

\newcommand{\forestKS}{
\tikz[planar forest ] {

\node [b] at (0.0, 0.0) {  } 
child {node [l] at (0.0, 1.0) { $\times$ }  
}
;
}}

\newcommand{\forestLS}{
\tikz[planar forest ] {

\node [b] at (0.0, 0.0) {  } 
child {node [l] at (0.0, 1.0) { $\times$ }  
}
;
}}

\newcommand{\forestMS}{
\tikz[planar forest ] {

\node [b] at (0.0, 0.0) {  } 
child {node [b] at (0.0, 1.0) {  }  
}
;
}}

\newcommand{\forestNS}{
\tikz[planar forest ] {

\node [b] at (0.0, 0.0) {  } 
child {node [l] at (0.0, 1.0) { $\times$ }  
}
;
}}

\newcommand{\forestOS}{
\tikz[planar forest ] {

\node [b] at (0.0, 0.0) {  } 
child {node [l] at (0.0, 1.0) { $\times$ }  
}
;
}}

\newcommand{\forestPS}{
\tikz[planar forest ] {

\node [b] at (0.0, 0.0) {  } 
child {node [l] at (0.0, 1.0) { $\times$ }  
}
;
}}

\newcommand{\forestQS}{
\tikz[planar forest ] {

\node [b] at (0.0, 0.0) {  } 
child {node [l] at (0.0, 1.0) { $\times$ }  
}
;
}}

\newcommand{\forestRS}{
\tikz[planar forest ] {

\node [b] at (0.0, 0.0) {  } 
child {node [l] at (0.0, 1.0) { $\times$ }  
}
;
}}

\newcommand{\forestSS}{
\tikz[planar forest ] {

\node [b] at (0.0, 0.0) {  } 
child {node [l] at (0.0, 1.0) { $\times$ }  
}
;
}}

\newcommand{\forestTS}{
\tikz[planar forest ] {

\node [b] at (0.0, 0.0) {  } 
child {node [l] at (-0.5, 1.0) { $\times$ }  
}
child {node [b] at (0.5, 1.0) {  }  
}
;
}}

\newcommand{\forestUS}{
\tikz[planar forest ] {

\node [b] at (0.0, 0.0) {  } 
child {node [l] at (0.0, 1.0) { $\times$ }  
}
;
}}

\newcommand{\forestVS}{
\tikz[planar forest ] {

\node [b] at (0.0, 0.0) {  } 
child {node [b] at (-0.5, 1.0) {  }  
child {node [b] at (0.0, 1.0) {  }  
}
}
child {node [l] at (0.5, 1.0) { $\times$ }  
}
;
}}

\newcommand{\forestWS}{
\tikz[planar forest ] {

\node [b] at (0.0, 0.0) {  } 
child {node [l] at (-0.5, 1.0) { $\times$ }  
}
child {node [b] at (0.5, 1.0) {  }  
child {node [b] at (0.0, 1.0) {  }  
}
}
;
}}

\newcommand{\forestXS}{
\tikz[planar forest ] {

\node [b] at (0.0, 0.0) {  } 
;
}}

\newcommand{\forestYS}{
\tikz[planar forest ] {

\node [b] at (0.0, 0.0) {  } 
child {node [l] at (-0.5, 1.0) { $\times$ }  
}
child {node [b] at (0.5, 1.0) {  }  
child {node [b] at (0.0, 1.0) {  }  
}
}
;
}}

\newcommand{\forestAT}{
\tikz[planar forest ] {

\node [b] at (0.0, 0.0) {  } 
child {node [l] at (0.0, 1.0) { $\times$ }  
}
;
}}

\newcommand{\forestBT}{
\tikz[planar forest ] {

\node [b] at (0.0, 0.0) {  } 
child {node [l] at (-0.5, 1.0) { $\times$ }  
}
child {node [b] at (0.5, 1.0) {  }  
child {node [b] at (0.0, 1.0) {  }  
child {node [b] at (0.0, 1.0) {  }  
}
}
}
;
}}

\newcommand{\forestCT}{
\tikz[planar forest ] {

\node [b] at (0.0, 0.0) {  } 
child {node [l] at (-0.5, 1.0) { $\times$ }  
}
child {node [b] at (0.5, 1.0) {  }  
child {node [b] at (-0.5, 1.0) {  }  
}
child {node [b] at (0.5, 1.0) {  }  
}
}
;
}}

\newcommand{\forestDT}{
\tikz[planar forest ] {

\node [b] at (0.0, 0.0) {  } 
child {node [b] at (-1.0, 1.0) {  }  
}
child {node [l] at (0.0, 1.0) { $\times$ }  
}
child {node [b] at (1.0, 1.0) {  }  
child {node [b] at (0.0, 1.0) {  }  
}
}
;
}}

\newcommand{\forestET}{
\tikz[planar forest ] {

\node [b] at (0.0, 0.0) {  } 
child {node [b] at (-0.5, 1.0) {  }  
child {node [b] at (0.0, 1.0) {  }  
}
}
child {node [l] at (0.5, 1.0) { $\times$ }  
}
;
}}

\newcommand{\forestFT}{
\tikz[planar forest ] {

\node [b] at (0.0, 0.0) {  } 
;
}}

\newcommand{\forestGT}{
\tikz[planar forest ] {

\node [b] at (0.0, 0.0) {  } 
child {node [b] at (-0.5, 1.0) {  }  
child {node [b] at (0.0, 1.0) {  }  
}
}
child {node [l] at (0.5, 1.0) { $\times$ }  
}
;
}}

\newcommand{\forestHT}{
\tikz[planar forest ] {

\node [b] at (0.0, 0.0) {  } 
child {node [l] at (0.0, 1.0) { $\times$ }  
}
;
}}

\newcommand{\forestIT}{
\tikz[planar forest ] {

\node [b] at (0.0, 0.0) {  } 
child {node [b] at (-0.5, 1.0) {  }  
child {node [b] at (0.0, 1.0) {  }  
child {node [b] at (0.0, 1.0) {  }  
}
}
}
child {node [l] at (0.5, 1.0) { $\times$ }  
}
;
}}

\newcommand{\forestJT}{
\tikz[planar forest ] {

\node [b] at (0.0, 0.0) {  } 
child {node [b] at (-0.5, 1.0) {  }  
child {node [b] at (-0.5, 1.0) {  }  
}
child {node [b] at (0.5, 1.0) {  }  
}
}
child {node [l] at (0.5, 1.0) { $\times$ }  
}
;
}}

\newcommand{\forestKT}{
\tikz[planar forest ] {

\node [b] at (0.0, 0.0) {  } 
child {node [b] at (-1.0, 1.0) {  }  
}
child {node [b] at (0.0, 1.0) {  }  
child {node [b] at (0.0, 1.0) {  }  
}
}
child {node [l] at (1.0, 1.0) { $\times$ }  
}
;
}}

\newcommand{\forestLT}{
\tikz[planar forest ] {

\node [b] at (0.0, 0.0) {  } 
child {node [l] at (-1.0, 1.0) { $\times$ }  
}
child {node [b] at (0.0, 1.0) {  }  
}
child {node [b] at (1.0, 1.0) {  }  
}
;
}}

\newcommand{\forestMT}{
\tikz[planar forest ] {

\node [b] at (0.0, 0.0) {  } 
;
}}

\newcommand{\forestNT}{
\tikz[planar forest ] {

\node [b] at (0.0, 0.0) {  } 
child {node [l] at (-1.0, 1.0) { $\times$ }  
}
child {node [b] at (0.0, 1.0) {  }  
}
child {node [b] at (1.0, 1.0) {  }  
}
;
}}

\newcommand{\forestOT}{
\tikz[planar forest ] {

\node [b] at (0.0, 0.0) {  } 
child {node [l] at (0.0, 1.0) { $\times$ }  
}
;
}}

\newcommand{\forestPT}{
\tikz[planar forest ] {

\node [b] at (0.0, 0.0) {  } 
child {node [l] at (-1.0, 1.0) { $\times$ }  
}
child {node [b] at (0.0, 1.0) {  }  
}
child {node [b] at (1.0, 1.0) {  }  
child {node [b] at (0.0, 1.0) {  }  
}
}
;
}}

\newcommand{\forestQT}{
\tikz[planar forest ] {

\node [b] at (0.0, 0.0) {  } 
child {node [l] at (-1.0, 1.0) { $\times$ }  
}
child {node [b] at (0.0, 1.0) {  }  
child {node [b] at (0.0, 1.0) {  }  
}
}
child {node [b] at (1.0, 1.0) {  }  
}
;
}}

\newcommand{\forestRT}{
\tikz[planar forest ] {

\node [b] at (0.0, 0.0) {  } 
child {node [b] at (-1.5, 1.0) {  }  
}
child {node [l] at (-0.5, 1.0) { $\times$ }  
}
child {node [b] at (0.5, 1.0) {  }  
}
child {node [b] at (1.5, 1.0) {  }  
}
;
}}

\newcommand{\forestST}{
\tikz[planar forest ] {

\node [b] at (0.0, 0.0) {  } 
child {node [b] at (-1.0, 1.0) {  }  
}
child {node [l] at (0.0, 1.0) { $\times$ }  
}
child {node [b] at (1.0, 1.0) {  }  
}
;
}}

\newcommand{\forestTT}{
\tikz[planar forest ] {

\node [b] at (0.0, 0.0) {  } 
;
}}

\newcommand{\forestUT}{
\tikz[planar forest ] {

\node [b] at (0.0, 0.0) {  } 
child {node [b] at (-1.0, 1.0) {  }  
}
child {node [l] at (0.0, 1.0) { $\times$ }  
}
child {node [b] at (1.0, 1.0) {  }  
}
;
}}

\newcommand{\forestVT}{
\tikz[planar forest ] {

\node [b] at (0.0, 0.0) {  } 
child {node [l] at (0.0, 1.0) { $\times$ }  
}
;
}}

\newcommand{\forestWT}{
\tikz[planar forest ] {

\node [b] at (0.0, 0.0) {  } 
child {node [b] at (-1.0, 1.0) {  }  
}
child {node [l] at (0.0, 1.0) { $\times$ }  
}
child {node [b] at (1.0, 1.0) {  }  
child {node [b] at (0.0, 1.0) {  }  
}
}
;
}}

\newcommand{\forestXT}{
\tikz[planar forest ] {

\node [b] at (0.0, 0.0) {  } 
child {node [b] at (-1.0, 1.0) {  }  
child {node [b] at (0.0, 1.0) {  }  
}
}
child {node [l] at (0.0, 1.0) { $\times$ }  
}
child {node [b] at (1.0, 1.0) {  }  
}
;
}}

\newcommand{\forestYT}{
\tikz[planar forest ] {

\node [b] at (0.0, 0.0) {  } 
child {node [b] at (-1.5, 1.0) {  }  
}
child {node [b] at (-0.5, 1.0) {  }  
}
child {node [l] at (0.5, 1.0) { $\times$ }  
}
child {node [b] at (1.5, 1.0) {  }  
}
;
}}

\newcommand{\forestAU}{
\tikz[planar forest ] {

\node [b] at (0.0, 0.0) {  } 
child {node [b] at (-1.0, 1.0) {  }  
}
child {node [b] at (0.0, 1.0) {  }  
}
child {node [l] at (1.0, 1.0) { $\times$ }  
}
;
}}

\newcommand{\forestBU}{
\tikz[planar forest ] {

\node [b] at (0.0, 0.0) {  } 
;
}}

\newcommand{\forestCU}{
\tikz[planar forest ] {

\node [b] at (0.0, 0.0) {  } 
child {node [b] at (-1.0, 1.0) {  }  
}
child {node [b] at (0.0, 1.0) {  }  
}
child {node [l] at (1.0, 1.0) { $\times$ }  
}
;
}}

\newcommand{\forestDU}{
\tikz[planar forest ] {

\node [b] at (0.0, 0.0) {  } 
child {node [l] at (0.0, 1.0) { $\times$ }  
}
;
}}

\newcommand{\forestEU}{
\tikz[planar forest ] {

\node [b] at (0.0, 0.0) {  } 
child {node [b] at (-1.0, 1.0) {  }  
}
child {node [b] at (0.0, 1.0) {  }  
child {node [b] at (0.0, 1.0) {  }  
}
}
child {node [l] at (1.0, 1.0) { $\times$ }  
}
;
}}

\newcommand{\forestFU}{
\tikz[planar forest ] {

\node [b] at (0.0, 0.0) {  } 
child {node [b] at (-1.0, 1.0) {  }  
child {node [b] at (0.0, 1.0) {  }  
}
}
child {node [b] at (0.0, 1.0) {  }  
}
child {node [l] at (1.0, 1.0) { $\times$ }  
}
;
}}

\newcommand{\forestGU}{
\tikz[planar forest ] {

\node [b] at (0.0, 0.0) {  } 
child {node [b] at (-1.5, 1.0) {  }  
}
child {node [b] at (-0.5, 1.0) {  }  
}
child {node [b] at (0.5, 1.0) {  }  
}
child {node [l] at (1.5, 1.0) { $\times$ }  
}
;
}}

\newcommand{\forestHU}{
\tikz[planar forest ] {

\node [b] at (0.0, 0.0) {  } 
child {node [l] at (-0.5, 1.0) { $\times$ }  
}
child {node [b] at (0.5, 1.0) {  }  
}
;
}}

\newcommand{\forestIU}{
\tikz[planar forest ] {

\node [b] at (0.0, 0.0) {  } 
child {node [l] at (0.0, 1.0) { $\times$ }  
}
;
}}

\newcommand{\forestJU}{
\tikz[planar forest ] {

\node [b] at (0.0, 0.0) {  } 
;
}}

\newcommand{\forestKU}{
\tikz[planar forest ] {

\node [b] at (0.0, 0.0) {  } 
child {node [l] at (-0.5, 1.0) { $\times$ }  
}
child {node [b] at (0.5, 1.0) {  }  
}
;
}}

\newcommand{\forestLU}{
\tikz[planar forest ] {

\node [b] at (0.0, 0.0) {  } 
child {node [l] at (0.0, 1.0) { $\times$ }  
}
;
}}

\newcommand{\forestMU}{
\tikz[planar forest ] {

\node [b] at (0.0, 0.0) {  } 
child {node [l] at (0.0, 1.0) { $\times$ }  
}
;
}}

\newcommand{\forestNU}{
\tikz[planar forest ] {

\node [b] at (0.0, 0.0) {  } 
child {node [l] at (-0.5, 1.0) { $\times$ }  
}
child {node [b] at (0.5, 1.0) {  }  
}
;
}}

\newcommand{\forestOU}{
\tikz[planar forest ] {

\node [b] at (0.0, 0.0) {  } 
child {node [b] at (-0.5, 1.0) {  }  
}
child {node [l] at (0.5, 1.0) { $\times$ }  
}
;
}}

\newcommand{\forestPU}{
\tikz[planar forest ] {

\node [b] at (0.0, 0.0) {  } 
child {node [l] at (-0.5, 1.0) { $\times$ }  
}
child {node [b] at (0.5, 1.0) {  }  
child {node [b] at (0.0, 1.0) {  }  
}
}
;
}}

\newcommand{\forestQU}{
\tikz[planar forest ] {

\node [b] at (0.0, 0.0) {  } 
child {node [l] at (0.0, 1.0) { $\times$ }  
}
;
}}

\newcommand{\forestRU}{
\tikz[planar forest ] {

\node [b] at (0.0, 0.0) {  } 
child {node [b] at (-1.0, 1.0) {  }  
}
child {node [l] at (0.0, 1.0) { $\times$ }  
}
child {node [b] at (1.0, 1.0) {  }  
}
;
}}

\newcommand{\forestSU}{
\tikz[planar forest ] {

\node [b] at (0.0, 0.0) {  } 
child {node [l] at (0.0, 1.0) { $\times$ }  
}
;
}}

\newcommand{\forestTU}{
\tikz[planar forest ] {

\node [b] at (0.0, 0.0) {  } 
child {node [b] at (-0.5, 1.0) {  }  
}
child {node [l] at (0.5, 1.0) { $\times$ }  
}
;
}}

\newcommand{\forestUU}{
\tikz[planar forest ] {

\node [b] at (0.0, 0.0) {  } 
child {node [l] at (0.0, 1.0) { $\times$ }  
}
;
}}

\newcommand{\forestVU}{
\tikz[planar forest ] {

\node [b] at (0.0, 0.0) {  } 
;
}}

\newcommand{\forestWU}{
\tikz[planar forest ] {

\node [b] at (0.0, 0.0) {  } 
child {node [b] at (-0.5, 1.0) {  }  
}
child {node [l] at (0.5, 1.0) { $\times$ }  
}
;
}}

\newcommand{\forestXU}{
\tikz[planar forest ] {

\node [b] at (0.0, 0.0) {  } 
child {node [l] at (0.0, 1.0) { $\times$ }  
}
;
}}

\newcommand{\forestYU}{
\tikz[planar forest ] {

\node [b] at (0.0, 0.0) {  } 
child {node [l] at (0.0, 1.0) { $\times$ }  
}
;
}}

\newcommand{\forestAV}{
\tikz[planar forest ] {

\node [b] at (0.0, 0.0) {  } 
child {node [b] at (-0.5, 1.0) {  }  
}
child {node [l] at (0.5, 1.0) { $\times$ }  
}
;
}}

\newcommand{\forestBV}{
\tikz[planar forest ] {

\node [b] at (0.0, 0.0) {  } 
child {node [b] at (-0.5, 1.0) {  }  
}
child {node [l] at (0.5, 1.0) { $\times$ }  
}
;
}}

\newcommand{\forestCV}{
\tikz[planar forest ] {

\node [b] at (0.0, 0.0) {  } 
child {node [b] at (-0.5, 1.0) {  }  
child {node [b] at (0.0, 1.0) {  }  
}
}
child {node [l] at (0.5, 1.0) { $\times$ }  
}
;
}}

\newcommand{\forestDV}{
\tikz[planar forest ] {

\node [b] at (0.0, 0.0) {  } 
child {node [l] at (0.0, 1.0) { $\times$ }  
}
;
}}

\newcommand{\forestEV}{
\tikz[planar forest ] {

\node [b] at (0.0, 0.0) {  } 
child {node [b] at (-1.0, 1.0) {  }  
}
child {node [b] at (0.0, 1.0) {  }  
}
child {node [l] at (1.0, 1.0) { $\times$ }  
}
;
}}

\newcommand{\forestFV}{
\tikz[planar forest ] {

\node [b] at (0.0, 0.0) {  } 
child {node [l] at (0.0, 1.0) { $\times$ }  
}
;
}}

\newcommand{\forestGV}{
\tikz[planar forest ] {

\node [b] at (0.0, 0.0) {  } 
child {node [l] at (0.0, 1.0) { $\times$ }  
}
;
}}

\newcommand{\forestHV}{
\tikz[planar forest ] {

\node [b] at (0.0, 0.0) {  } 
child {node [l] at (0.0, 1.0) { $\times$ }  
}
;
}}

\newcommand{\forestIV}{
\tikz[planar forest ] {

\node [b] at (0.0, 0.0) {  } 
child {node [l] at (0.0, 1.0) { $\times$ }  
}
;
}}

\newcommand{\forestJV}{
\tikz[planar forest ] {

\node [b] at (0.0, 0.0) {  } 
;
}}

\newcommand{\forestKV}{
\tikz[planar forest ] {

\node [b] at (0.0, 0.0) {  } 
child {node [l] at (0.0, 1.0) { $\times$ }  
}
;
}}

\newcommand{\forestLV}{
\tikz[planar forest ] {

\node [b] at (0.0, 0.0) {  } 
child {node [l] at (0.0, 1.0) { $\times$ }  
}
;
}}

\newcommand{\forestMV}{
\tikz[planar forest ] {

\node [b] at (0.0, 0.0) {  } 
child {node [l] at (0.0, 1.0) { $\times$ }  
}
;
}}

\newcommand{\forestNV}{
\tikz[planar forest ] {

\node [b] at (0.0, 0.0) {  } 
child {node [l] at (0.0, 1.0) { $\times$ }  
}
;
}}

\newcommand{\forestOV}{
\tikz[planar forest ] {

\node [b] at (0.0, 0.0) {  } 
child {node [l] at (0.0, 1.0) { $\times$ }  
}
;
}}

\newcommand{\forestPV}{
\tikz[planar forest ] {

\node [b] at (0.0, 0.0) {  } 
child {node [l] at (0.0, 1.0) { $\times$ }  
}
;
}}

\newcommand{\forestQV}{
\tikz[planar forest ] {

\node [b] at (0.0, 0.0) {  } 
child {node [b] at (-0.5, 1.0) {  }  
}
child {node [l] at (0.5, 1.0) { $\times$ }  
}
;
}}

\newcommand{\forestRV}{
\tikz[planar forest ] {

\node [b] at (0.0, 0.0) {  } 
child {node [l] at (-0.5, 1.0) { $\times$ }  
}
child {node [b] at (0.5, 1.0) {  }  
}
;
}}

\newcommand{\forestSV}{
\tikz[planar forest ] {

\node [b] at (0.0, 0.0) {  } 
child {node [l] at (0.0, 1.0) { $\times$ }  
}
;
}}

\newcommand{\forestTV}{
\tikz[planar forest ] {

\node [b] at (0.0, 0.0) {  } 
;
}}

\newcommand{\forestUV}{
\tikz[planar forest ] {

\node [b] at (0.0, 0.0) {  } 
child {node [l] at (-0.5, 1.0) { $\times$ }  
}
child {node [b] at (0.5, 1.0) {  }  
}
;
}}

\newcommand{\forestVV}{
\tikz[planar forest ] {

\node [b] at (0.0, 0.0) {  } 
child {node [l] at (0.0, 1.0) { $\times$ }  
}
;
}}

\newcommand{\forestWV}{
\tikz[planar forest ] {

\node [b] at (0.0, 0.0) {  } 
child {node [l] at (0.0, 1.0) { $\times$ }  
}
;
}}

\newcommand{\forestXV}{
\tikz[planar forest ] {

\node [b] at (0.0, 0.0) {  } 
child {node [l] at (-0.5, 1.0) { $\times$ }  
}
child {node [b] at (0.5, 1.0) {  }  
}
;
}}

\newcommand{\forestYV}{
\tikz[planar forest ] {

\node [b] at (0.0, 0.0) {  } 
child {node [b] at (-0.5, 1.0) {  }  
}
child {node [l] at (0.5, 1.0) { $\times$ }  
}
;
}}

\newcommand{\forestAW}{
\tikz[planar forest ] {

\node [b] at (0.0, 0.0) {  } 
child {node [l] at (-0.5, 1.0) { $\times$ }  
}
child {node [b] at (0.5, 1.0) {  }  
child {node [b] at (0.0, 1.0) {  }  
}
}
;
}}

\newcommand{\forestBW}{
\tikz[planar forest ] {

\node [b] at (0.0, 0.0) {  } 
child {node [l] at (0.0, 1.0) { $\times$ }  
}
;
}}

\newcommand{\forestCW}{
\tikz[planar forest ] {

\node [b] at (0.0, 0.0) {  } 
child {node [b] at (-1.0, 1.0) {  }  
}
child {node [l] at (0.0, 1.0) { $\times$ }  
}
child {node [b] at (1.0, 1.0) {  }  
}
;
}}

\newcommand{\forestDW}{
\tikz[planar forest ] {

\node [b] at (0.0, 0.0) {  } 
child {node [l] at (0.0, 1.0) { $\times$ }  
}
;
}}

\newcommand{\forestEW}{
\tikz[planar forest ] {

\node [b] at (0.0, 0.0) {  } 
child {node [b] at (-0.5, 1.0) {  }  
}
child {node [l] at (0.5, 1.0) { $\times$ }  
}
;
}}

\newcommand{\forestFW}{
\tikz[planar forest ] {

\node [b] at (0.0, 0.0) {  } 
child {node [l] at (0.0, 1.0) { $\times$ }  
}
;
}}

\newcommand{\forestGW}{
\tikz[planar forest ] {

\node [b] at (0.0, 0.0) {  } 
;
}}

\newcommand{\forestHW}{
\tikz[planar forest ] {

\node [b] at (0.0, 0.0) {  } 
child {node [b] at (-0.5, 1.0) {  }  
}
child {node [l] at (0.5, 1.0) { $\times$ }  
}
;
}}

\newcommand{\forestIW}{
\tikz[planar forest ] {

\node [b] at (0.0, 0.0) {  } 
child {node [l] at (0.0, 1.0) { $\times$ }  
}
;
}}

\newcommand{\forestJW}{
\tikz[planar forest ] {

\node [b] at (0.0, 0.0) {  } 
child {node [l] at (0.0, 1.0) { $\times$ }  
}
;
}}

\newcommand{\forestKW}{
\tikz[planar forest ] {

\node [b] at (0.0, 0.0) {  } 
child {node [b] at (-0.5, 1.0) {  }  
}
child {node [l] at (0.5, 1.0) { $\times$ }  
}
;
}}

\newcommand{\forestLW}{
\tikz[planar forest ] {

\node [b] at (0.0, 0.0) {  } 
child {node [b] at (-0.5, 1.0) {  }  
}
child {node [l] at (0.5, 1.0) { $\times$ }  
}
;
}}

\newcommand{\forestMW}{
\tikz[planar forest ] {

\node [b] at (0.0, 0.0) {  } 
child {node [b] at (-0.5, 1.0) {  }  
child {node [b] at (0.0, 1.0) {  }  
}
}
child {node [l] at (0.5, 1.0) { $\times$ }  
}
;
}}

\newcommand{\forestNW}{
\tikz[planar forest ] {

\node [b] at (0.0, 0.0) {  } 
child {node [l] at (0.0, 1.0) { $\times$ }  
}
;
}}

\newcommand{\forestOW}{
\tikz[planar forest ] {

\node [b] at (0.0, 0.0) {  } 
child {node [b] at (-1.0, 1.0) {  }  
}
child {node [b] at (0.0, 1.0) {  }  
}
child {node [l] at (1.0, 1.0) { $\times$ }  
}
;
}}

\newcommand{\forestPW}{
\tikz[planar forest ] {

\node [b] at (0.0, 0.0) {  } 
child {node [l] at (0.0, 1.0) { $\times$ }  
}
;
}}

\newcommand{\forestQW}{
\tikz[planar forest ] {

\node [b] at (0.0, 0.0) {  } 
child {node [l] at (0.0, 1.0) { $\times$ }  
}
;
}}

\newcommand{\forestRW}{
\tikz[planar forest ] {

\node [b] at (0.0, 0.0) {  } 
child {node [l] at (0.0, 1.0) { $\times$ }  
}
;
}}

\newcommand{\forestSW}{
\tikz[planar forest ] {

\node [b] at (0.0, 0.0) {  } 
child {node [l] at (0.0, 1.0) { $\times$ }  
}
;
}}

\newcommand{\forestTW}{
\tikz[planar forest ] {

\node [b] at (0.0, 0.0) {  } 
;
}}

\newcommand{\forestUW}{
\tikz[planar forest ] {

\node [b] at (0.0, 0.0) {  } 
child {node [l] at (0.0, 1.0) { $\times$ }  
}
;
}}

\newcommand{\forestVW}{
\tikz[planar forest ] {

\node [b] at (0.0, 0.0) {  } 
child {node [l] at (0.0, 1.0) { $\times$ }  
}
;
}}

\newcommand{\forestWW}{
\tikz[planar forest ] {

\node [b] at (0.0, 0.0) {  } 
child {node [l] at (0.0, 1.0) { $\times$ }  
}
;
}}

\newcommand{\forestXW}{
\tikz[planar forest ] {

\node [b] at (0.0, 0.0) {  } 
child {node [l] at (0.0, 1.0) { $\times$ }  
}
;
}}

\newcommand{\forestYW}{
\tikz[planar forest ] {

\node [b] at (0.0, 0.0) {  } 
child {node [l] at (0.0, 1.0) { $\times$ }  
}
;
}}

\newcommand{\forestAX}{
\tikz[planar forest ] {

\node [b] at (0.0, 0.0) {  } 
child {node [l] at (0.0, 1.0) { $\times$ }  
}
;
}}

\newcommand{\forestBX}{
\tikz[planar forest ] {

\node [b] at (0.0, 0.0) {  } 
child {node [b] at (-0.5, 1.0) {  }  
}
child {node [l] at (0.5, 1.0) { $\times$ }  
}
;
}}

\newcommand{\forestCX}{
\tikz[planar forest ] {

\node [b] at (0.0, 0.0) {  } 
child {node [l] at (0.0, 1.0) { $\times$ }  
}
;
}}

\newcommand{\forestDX}{
\tikz[planar forest ] {

\node [b] at (0.0, 0.0) {  } 
child {node [b] at (-0.5, 1.0) {  }  
}
child {node [l] at (0.5, 1.0) { $\times$ }  
}
;
}}

\newcommand{\forestEX}{
\tikz[planar forest ] {

\node [b] at (0.0, 0.0) {  } 
child {node [l] at (0.0, 1.0) { $\times$ }  
}
;
}}

\newcommand{\forestFX}{
\tikz[planar forest ] {

\node [b] at (0.0, 0.0) {  } 
child {node [l] at (0.0, 1.0) { $\times$ }  
}
;
}}

\newcommand{\forestGX}{
\tikz[planar forest ] {

\node [b] at (0.0, 0.0) {  } 
child {node [l] at (0.0, 1.0) { $\times$ }  
}
;
}}

\newcommand{\forestHX}{
\tikz[planar forest ] {

\node [b] at (0.0, 0.0) {  } 
child {node [l] at (0.0, 1.0) { $\times$ }  
}
;
}}

\newcommand{\forestIX}{
\tikz[planar forest ] {

\node [b] at (0.0, 0.0) {  } 
;
}}

\newcommand{\forestJX}{
\tikz[planar forest ] {

\node [b] at (0.0, 0.0) {  } 
child {node [l] at (0.0, 1.0) { $\times$ }  
}
;
}}

\newcommand{\forestKX}{
\tikz[planar forest ] {

\node [b] at (0.0, 0.0) {  } 
child {node [l] at (0.0, 1.0) { $\times$ }  
}
;
}}

\newcommand{\forestLX}{
\tikz[planar forest ] {

\node [b] at (0.0, 0.0) {  } 
child {node [l] at (0.0, 1.0) { $\times$ }  
}
;
}}

\newcommand{\forestMX}{
\tikz[planar forest ] {

\node [b] at (0.0, 0.0) {  } 
child {node [l] at (0.0, 1.0) { $\times$ }  
}
;
}}

\newcommand{\forestNX}{
\tikz[planar forest ] {

\node [b] at (0.0, 0.0) {  } 
child {node [l] at (0.0, 1.0) { $\times$ }  
}
;
}}

\newcommand{\forestOX}{
\tikz[planar forest ] {

\node [b] at (0.0, 0.0) {  } 
child {node [l] at (0.0, 1.0) { $\times$ }  
}
;
}}

\newcommand{\forestPX}{
\tikz[planar forest ] {

\node [b] at (0.0, 0.0) {  } 
child {node [b] at (-0.5, 1.0) {  }  
}
child {node [l] at (0.5, 1.0) { $\times$ }  
}
;
}}

\newcommand{\forestQX}{
\tikz[planar forest ] {

\node [b] at (0.0, 0.0) {  } 
;
}}

\newcommand{\forestRX}{
\tikz[planar forest ] {

\node [b] at (0.0, 0.0) {  } 
child {node [b] at (0.0, 1.0) {  }  
child {node [b] at (0.0, 1.0) {  }  
}
}
;
}}

\newcommand{\forestSX}{
\tikz[planar forest ] {

\node [b] at (0.0, 0.0) {  } 
child {node [l] at (0.0, 1.0) { $\times$ }  
}
;
}}

\newcommand{\forestTX}{
\tikz[planar forest ] {

\node [b] at (0.0, 0.0) {  } 
child {node [l] at (0.0, 1.0) { $\times$ }  
}
;
}}

\newcommand{\forestUX}{
\tikz[planar forest ] {

\node [b] at (0.0, 0.0) {  } 
child {node [b] at (-0.5, 1.0) {  }  
}
child {node [l] at (0.5, 1.0) { $\times$ }  
}
;
}}

\newcommand{\forestVX}{
\tikz[planar forest ] {

\node [b] at (0.0, 0.0) {  } 
child {node [l] at (0.0, 1.0) { $\times$ }  
}
;
}}

\newcommand{\forestWX}{
\tikz[planar forest ] {

\node [b] at (0.0, 0.0) {  } 
child {node [l] at (0.0, 1.0) { $\times$ }  
}
;
}}

\newcommand{\forestXX}{
\tikz[planar forest ] {

\node [b] at (0.0, 0.0) {  } 
child {node [l] at (-0.5, 1.0) { $\times$ }  
}
child {node [b] at (0.5, 1.0) {  }  
}
;
}}

\newcommand{\forestYX}{
\tikz[planar forest ] {

\node [b] at (0.0, 0.0) {  } 
child {node [l] at (0.0, 1.0) { $\times$ }  
}
;
}}

\newcommand{\forestAY}{
\tikz[planar forest ] {

\node [b] at (0.0, 0.0) {  } 
child {node [b] at (-1.0, 1.0) {  }  
}
child {node [l] at (0.0, 1.0) { $\times$ }  
}
child {node [b] at (1.0, 1.0) {  }  
}
;
}}

\newcommand{\forestBY}{
\tikz[planar forest ] {

\node [b] at (0.0, 0.0) {  } 
child {node [l] at (0.0, 1.0) { $\times$ }  
}
;
}}

\newcommand{\forestCY}{
\tikz[planar forest ] {

\node [b] at (0.0, 0.0) {  } 
child {node [l] at (-1.0, 1.0) { $\times$ }  
}
child {node [b] at (0.0, 1.0) {  }  
}
child {node [b] at (1.0, 1.0) {  }  
}
;
}}

\newcommand{\forestDY}{
\tikz[planar forest ] {

\node [b] at (0.0, 0.0) {  } 
child {node [b] at (-1.0, 1.0) {  }  
}
child {node [l] at (0.0, 1.0) { $\times$ }  
}
child {node [b] at (1.0, 1.0) {  }  
child {node [b] at (0.0, 1.0) {  }  
}
}
;
}}

\newcommand{\forestEY}{
\tikz[planar forest ] {

\node [b] at (0.0, 0.0) {  } 
child {node [l] at (-1.0, 1.0) { $\times$ }  
}
child {node [b] at (0.0, 1.0) {  }  
}
child {node [b] at (1.0, 1.0) {  }  
child {node [b] at (0.0, 1.0) {  }  
}
}
;
}}

\newcommand{\forestFY}{
\tikz[planar forest ] {

\node [b] at (0.0, 0.0) {  } 
child {node [l] at (-0.5, 1.0) { $\times$ }  
}
child {node [b] at (0.5, 1.0) {  }  
child {node [b] at (0.0, 1.0) {  }  
child {node [b] at (0.0, 1.0) {  }  
}
}
}
;
}}

\newcommand{\forestGY}{
\tikz[planar forest ] {

\node [b] at (0.0, 0.0) {  } 
child {node [b] at (-0.5, 1.0) {  }  
child {node [b] at (0.0, 1.0) {  }  
child {node [b] at (0.0, 1.0) {  }  
}
}
}
child {node [l] at (0.5, 1.0) { $\times$ }  
}
;
}}

\newcommand{\forestHY}{
\tikz[planar forest ] {

\node [b] at (0.0, 0.0) {  } 
;
}}

\newcommand{\forestIY}{
\tikz[planar forest ] {

\node [b] at (0.0, 0.0) {  } 
child {node [b] at (-0.5, 1.0) {  }  
}
child {node [b] at (0.5, 1.0) {  }  
}
;
}}

\newcommand{\forestJY}{
\tikz[planar forest ] {

\node [b] at (0.0, 0.0) {  } 
child {node [b] at (-0.5, 1.0) {  }  
}
child {node [l] at (0.5, 1.0) { $\times$ }  
}
;
}}

\newcommand{\forestKY}{
\tikz[planar forest ] {

\node [b] at (0.0, 0.0) {  } 
child {node [b] at (-0.5, 1.0) {  }  
}
child {node [l] at (0.5, 1.0) { $\times$ }  
}
;
}}

\newcommand{\forestLY}{
\tikz[planar forest ] {

\node [b] at (0.0, 0.0) {  } 
child {node [l] at (-0.5, 1.0) { $\times$ }  
}
child {node [b] at (0.5, 1.0) {  }  
}
;
}}

\newcommand{\forestMY}{
\tikz[planar forest ] {

\node [b] at (0.0, 0.0) {  } 
child {node [l] at (-0.5, 1.0) { $\times$ }  
}
child {node [b] at (0.5, 1.0) {  }  
}
;
}}

\newcommand{\forestNY}{
\tikz[planar forest ] {

\node [b] at (0.0, 0.0) {  } 
child {node [b] at (-1.5, 1.0) {  }  
}
child {node [l] at (-0.5, 1.0) { $\times$ }  
}
child {node [b] at (0.5, 1.0) {  }  
}
child {node [b] at (1.5, 1.0) {  }  
}
;
}}

\newcommand{\forestOY}{
\tikz[planar forest ] {

\node [b] at (0.0, 0.0) {  } 
child {node [l] at (-1.5, 1.0) { $\times$ }  
}
child {node [b] at (-0.5, 1.0) {  }  
}
child {node [b] at (0.5, 1.0) {  }  
}
child {node [b] at (1.5, 1.0) {  }  
}
;
}}

\newcommand{\forestPY}{
\tikz[planar forest ] {

\node [b] at (0.0, 0.0) {  } 
child {node [l] at (-0.5, 1.0) { $\times$ }  
}
child {node [b] at (0.5, 1.0) {  }  
child {node [b] at (-0.5, 1.0) {  }  
}
child {node [b] at (0.5, 1.0) {  }  
}
}
;
}}

\newcommand{\forestQY}{
\tikz[planar forest ] {

\node [b] at (0.0, 0.0) {  } 
child {node [b] at (-0.5, 1.0) {  }  
child {node [b] at (-0.5, 1.0) {  }  
}
child {node [b] at (0.5, 1.0) {  }  
}
}
child {node [l] at (0.5, 1.0) { $\times$ }  
}
;
}}

\newcommand{\forestRY}{
\tikz[planar forest ] {

\node [b] at (0.0, 0.0) {  } 
child {node [l] at (0.0, 1.0) { $\times$ }  
}
;
}}

\newcommand{\forestSY}{
\tikz[planar forest ] {

\node [b] at (0.0, 0.0) {  } 
;
}}

\newcommand{\forestTY}{
\tikz[planar forest ] {

\node [b] at (0.0, 0.0) {  } 
child {node [b] at (0.0, 1.0) {  }  
}
;
}}

\newcommand{\forestUY}{
\tikz[planar forest ] {

\node [b] at (0.0, 0.0) {  } 
child {node [l] at (0.0, 1.0) { $\times$ }  
}
;
}}

\newcommand{\forestVY}{
\tikz[planar forest ] {

\node [b] at (0.0, 0.0) {  } 
child {node [l] at (0.0, 1.0) { $\times$ }  
}
;
}}

\newcommand{\forestWY}{
\tikz[planar forest ] {

\node [b] at (0.0, 0.0) {  } 
child {node [b] at (-0.5, 1.0) {  }  
}
child {node [l] at (0.5, 1.0) { $\times$ }  
}
;
}}

\newcommand{\forestXY}{
\tikz[planar forest ] {

\node [b] at (0.0, 0.0) {  } 
child {node [l] at (0.0, 1.0) { $\times$ }  
}
;
}}

\newcommand{\forestYY}{
\tikz[planar forest ] {

\node [b] at (0.0, 0.0) {  } 
child {node [l] at (0.0, 1.0) { $\times$ }  
}
;
}}

\newcommand{\forestAAB}{
\tikz[planar forest ] {

\node [b] at (0.0, 0.0) {  } 
child {node [l] at (-0.5, 1.0) { $\times$ }  
}
child {node [b] at (0.5, 1.0) {  }  
}
;
}}

\newcommand{\forestBAB}{
\tikz[planar forest ] {

\node [b] at (0.0, 0.0) {  } 
child {node [l] at (0.0, 1.0) { $\times$ }  
}
;
}}

\newcommand{\forestCAB}{
\tikz[planar forest ] {

\node [b] at (0.0, 0.0) {  } 
child {node [b] at (-1.0, 1.0) {  }  
}
child {node [l] at (0.0, 1.0) { $\times$ }  
}
child {node [b] at (1.0, 1.0) {  }  
}
;
}}

\newcommand{\forestDAB}{
\tikz[planar forest ] {

\node [b] at (0.0, 0.0) {  } 
child {node [l] at (0.0, 1.0) { $\times$ }  
}
;
}}

\newcommand{\forestEAB}{
\tikz[planar forest ] {

\node [b] at (0.0, 0.0) {  } 
child {node [l] at (-1.0, 1.0) { $\times$ }  
}
child {node [b] at (0.0, 1.0) {  }  
}
child {node [b] at (1.0, 1.0) {  }  
}
;
}}

\newcommand{\forestFAB}{
\tikz[planar forest ] {

\node [b] at (0.0, 0.0) {  } 
child {node [l] at (0.0, 1.0) { $\times$ }  
}
;
}}

\newcommand{\forestGAB}{
\tikz[planar forest ] {

\node [b] at (0.0, 0.0) {  } 
child {node [l] at (-0.5, 1.0) { $\times$ }  
}
child {node [b] at (0.5, 1.0) {  }  
child {node [b] at (0.0, 1.0) {  }  
}
}
;
}}

\newcommand{\forestHAB}{
\tikz[planar forest ] {

\node [b] at (0.0, 0.0) {  } 
child {node [l] at (0.0, 1.0) { $\times$ }  
}
;
}}

\newcommand{\forestIAB}{
\tikz[planar forest ] {

\node [b] at (0.0, 0.0) {  } 
child {node [b] at (-0.5, 1.0) {  }  
child {node [b] at (0.0, 1.0) {  }  
}
}
child {node [l] at (0.5, 1.0) { $\times$ }  
}
;
}}

\newcommand{\forestJAB}{
\tikz[planar forest ] {

\node [b] at (0.0, 0.0) {  } 
child {node [b] at (-1.0, 1.0) {  }  
}
child {node [b] at (0.0, 1.0) {  }  
child {node [b] at (0.0, 1.0) {  }  
}
}
child {node [l] at (1.0, 1.0) { $\times$ }  
}
;
}}

\newcommand{\forestKAB}{
\tikz[planar forest ] {

\node [b] at (0.0, 0.0) {  } 
child {node [b] at (-1.0, 1.0) {  }  
child {node [b] at (0.0, 1.0) {  }  
}
}
child {node [b] at (0.0, 1.0) {  }  
}
child {node [l] at (1.0, 1.0) { $\times$ }  
}
;
}}

\newcommand{\forestLAB}{
\tikz[planar forest ] {

\node [b] at (0.0, 0.0) {  } 
;
}}

\newcommand{\forestMAB}{
\tikz[planar forest ] {

\node [b] at (0.0, 0.0) {  } 
;
}}

\newcommand{\forestNAB}{
\tikz[planar forest ] {

\node [b] at (0.0, 0.0) {  } 
child {node [b] at (0.0, 1.0) {  }  
}
;
}}

\newcommand{\forestOAB}{
\tikz[planar forest ] {

\node [b] at (0.0, 0.0) {  } 
child {node [b] at (-1.0, 1.0) {  }  
child {node [b] at (0.0, 1.0) {  }  
}
}
child {node [b] at (0.0, 1.0) {  }  
}
child {node [l] at (1.0, 1.0) { $\times$ }  
}
;
}}

\newcommand{\forestPAB}{
\tikz[planar forest ] {

\node [b] at (0.0, 0.0) {  } 
child {node [b] at (-1.0, 1.0) {  }  
child {node [b] at (0.0, 1.0) {  }  
}
}
child {node [l] at (0.0, 1.0) { $\times$ }  
}
child {node [b] at (1.0, 1.0) {  }  
}
;
}}

\newcommand{\forestQAB}{
\tikz[planar forest ] {

\node [b] at (0.0, 0.0) {  } 
child {node [b] at (-1.0, 1.0) {  }  
}
child {node [l] at (0.0, 1.0) { $\times$ }  
}
child {node [b] at (1.0, 1.0) {  }  
child {node [b] at (0.0, 1.0) {  }  
}
}
;
}}

\newcommand{\forestRAB}{
\tikz[planar forest ] {

\node [b] at (0.0, 0.0) {  } 
child {node [l] at (-1.0, 1.0) { $\times$ }  
}
child {node [b] at (0.0, 1.0) {  }  
}
child {node [b] at (1.0, 1.0) {  }  
child {node [b] at (0.0, 1.0) {  }  
}
}
;
}}

\newcommand{\forestSAB}{
\tikz[planar forest ] {

\node [b] at (0.0, 0.0) {  } 
child {node [b] at (-1.0, 1.0) {  }  
}
child {node [b] at (0.0, 1.0) {  }  
child {node [b] at (0.0, 1.0) {  }  
}
}
child {node [l] at (1.0, 1.0) { $\times$ }  
}
;
}}

\newcommand{\forestTAB}{
\tikz[planar forest ] {

\node [b] at (0.0, 0.0) {  } 
child {node [l] at (-1.0, 1.0) { $\times$ }  
}
child {node [b] at (0.0, 1.0) {  }  
child {node [b] at (0.0, 1.0) {  }  
}
}
child {node [b] at (1.0, 1.0) {  }  
}
;
}}

\newcommand{\forestUAB}{
\tikz[planar forest ] {

\node [b] at (0.0, 0.0) {  } 
child {node [b] at (0.0, 1.0) {  }  
child {node [b] at (-1.0, 1.0) {  }  
}
child {node [b] at (0.0, 1.0) {  }  
}
child {node [l] at (1.0, 1.0) { $\times$ }  
}
}
;
}}

\newcommand{\forestVAB}{
\tikz[planar forest ] {

\node [b] at (0.0, 0.0) {  } 
child {node [b] at (0.0, 1.0) {  }  
child {node [l] at (-1.0, 1.0) { $\times$ }  
}
child {node [b] at (0.0, 1.0) {  }  
}
child {node [b] at (1.0, 1.0) {  }  
}
}
;
}}

\newcommand{\forestWAB}{
\tikz[planar forest ] {

\node [b] at (0.0, 0.0) {  } 
child {node [b] at (0.0, 1.0) {  }  
child {node [b] at (-1.0, 1.0) {  }  
}
child {node [l] at (0.0, 1.0) { $\times$ }  
}
child {node [b] at (1.0, 1.0) {  }  
}
}
;
}}

\newcommand{\forestXAB}{
\tikz[planar forest ] {

\node [b] at (0.0, 0.0) {  } 
child {node [l] at (-1.5, 1.0) { $\times$ }  
}
child {node [b] at (-0.5, 1.0) {  }  
}
child {node [b] at (0.5, 1.0) {  }  
}
child {node [b] at (1.5, 1.0) {  }  
}
;
}}

\newcommand{\forestYAB}{
\tikz[planar forest ] {

\node [b] at (0.0, 0.0) {  } 
child {node [b] at (-1.5, 1.0) {  }  
}
child {node [b] at (-0.5, 1.0) {  }  
}
child {node [l] at (0.5, 1.0) { $\times$ }  
}
child {node [b] at (1.5, 1.0) {  }  
}
;
}}

\newcommand{\forestABB}{
\tikz[planar forest ] {

\node [b] at (0.0, 0.0) {  } 
child {node [b] at (-1.5, 1.0) {  }  
}
child {node [l] at (-0.5, 1.0) { $\times$ }  
}
child {node [b] at (0.5, 1.0) {  }  
}
child {node [b] at (1.5, 1.0) {  }  
}
;
}}

\DeclareMathOperator{\Vol}{Vol}
\DeclareMathOperator{\Lie}{Lie}

\newcommand{\graft}{\curvearrowright}

\newcommand{\XM}{\mathfrak{X}(\MM)}
\newcommand{\FM}{\mathcal{C^\infty}(\MM)}
\newcommand{\tri}{\triangleright}

\newcommand{\Tfe}{\TT_C^0}

\allowdisplaybreaks

\newtheorem{theorem}{Theorem}[section]
\newtheorem{definition}[theorem]{Definition}
\newtheorem*{definition*}{Definition}

\newtheorem{proposition}[theorem]{Proposition}
\newtheorem{lemma}[theorem]{Lemma}
\newtheorem{remark}[theorem]{Remark}
\newtheorem*{remark*}{Remark}

\newtheorem*{remarks*}{Remarks}

\newtheorem*{notation*}{Notation}
\newtheorem{ex}[theorem]{Example}
\newtheorem*{ex*}{Example}

\newtheorem*{exs*}{Examples}

\newtheorem*{app*}{Application}
\newtheorem{conjecture*}{Conjecture}

\newcommand{\ppr}{\odot}
\newcommand{\TM}{\operatorname{TM}}
\renewcommand{\tr}{\triangleright}
\newcommand{\g}{\mathfrak{g}}

\title{The free tracial post-Lie-Rinehart algebra of planar aromatic trees for the design of divergence-free Lie-group methods
}

\author{
Adrien Busnot Laurent\textsuperscript{1}, Hans Munthe-Kaas\textsuperscript{2} and Venkatesh~G.~S.\textsuperscript{2}
}

\begin{document}
\footnotetext[1]{
Univ Rennes, INRIA (Research team MINGuS), IRMAR (CNRS UMR 6625) and ENS Rennes, France.
Adrien.Busnot-Laurent@inria.fr.}
\footnotetext[2]{Department of Mathematics and Statistics, UiT – The Arctic University of Norway, Tromsø, Norway.
Hans.Munthe-Kaas@uib.no, Subbarao.V.Guggilam@uit.no.}

\maketitle

\begin{abstract}
Aromatic Butcher series were successfully introduced for the study and design of numerical integrators that preserve volume while solving differential equations in Euclidean spaces. They are naturally associated to pre-Lie-Rinehart algebras and pre-Hopf algebroids structures, and aromatic trees were shown to form the free tracial pre-Lie-Rinehart algebra.
In this paper, we present the generalisation of aromatic trees for the study of divergence-free integrators on manifolds. We introduce planar aromatic trees, prove that they span the free tracial post-Lie-Rinehart algebra, and apply them for deriving new Lie-group methods that preserve geometric divergence-free features up to a high order of accuracy.

\smallskip

\smallskip

\end{abstract}


\section{Introduction}

Butcher  trees and series were introduced in the 60's for the creation of high-order integration methods \cite{Butcher63cft, Butcher69teo, Butcher72aat, Hairer06gni}. Several extensions of the initial formalism were then introduced for the development of geometric numerical integration.
In particular, planar trees and Lie-Butcher series were used for designing Lie-group methods \cite{Iserles00lgm} and aromatic trees and aromatic B-series showed to be crucial tools for the study of volume-preserving integrators \cite{Chartier07pfi, Iserles07bsm}.
The creation of numerical methods on manifolds that preserve geometric invariants is an active field of research. In particular, there is, to the best of our knowledge, no existing work discussing the preservation of volume for intrinsic methods on manifolds.
In general, the Lie-group methods (Lie-Runge-Kutta, Crouch-Grossman or Runge-Kutta-Munthe-Kaas methods) and more generally LB-series methods, do not preserve volume, even in the Euclidean case \cite{Chartier07pfi,Iserles07bsm}.
The characterisation of Euclidean numerical volume-preservation is done through the use of backward error analysis \cite{Chartier05asl, Hairer06gni, Chartier10aso, Calaque11tih} and rewrites as the design of methods whose modified vector field satisfies $\Div(\tilde{f})=0$.
It is natural to consider aromatic modifications of LB-series methods to obtain divergence-free features for the numerical method at least up to a high order. A simple way to achieve high order of volume-preservation is obtained by considering high order-methods, but this approach is costly and sub-optimal.
Approaches for Euclidean pseudo-volume-preservation are first discussed in \cite{Bogfjellmo19aso, MuntheKaas16abs}.
There are other Euclidean techniques in the literature for building volume-preserving methods \cite{Shang94cov, Quispel95vpi}, but their complexity blows up with the dimension of the problem, opposite to the aromatic B-series approach.
The extension of the Euclidean results to intrinsic methods on manifolds have not been considered to the best of our knowledge.
In this paper, we introduce planar aromatic trees, study their algebraic structure, and apply them for the study and design of volume-preserving Lie-group methods. The analysis is successfully applied for the design of aromatic Lie-group methods with low convergence order and high order of volume preservation.

To motivate the algebraic formalism below, and its relation to geometry and numerical integration, we recall briefly the basic setup of Lie group integration and related numerical methods. Let $E\rightarrow M$ be a vector bundle over a manifold and $R=C^\infty(M,\R)$ and $L=\Gamma(E)$ the sections of the bundle (e.g.\ vector fields, tensor fields).  If $L$ is also equipped with a Lie bracket $[-,-]$ and a Lie algebra homomorphism $\rho\colon L\rightarrow \XM$ to the vector fields with the Jacobi bracket, we call $L$ a \emph{Lie algebroid}. An important example is the action algebroid obtained from the action of a Lie group on a manifold, which is the generic setup of numerical Lie group integration~\cite{MuntheKaas99hor}. In numerical algorithms we  define basic flows on $M$ which can be computed exactly, from which we develop more advanced integration algorithms. Such basic flows can often be described as geodesics of a connection. For the action algebroid, the canonical connection on $L$ is invariant with zero curvature and  parallel torsion. This is called a \emph{post-Lie algebroid}~\cite{MuntheKaas13opl}. In~\cite{MuntheKaas20icl} an intimate relationship between post-Lie algebroids and action algebroids is developed, and~\cite{AlKaabi22aao} explains why post-Lie algebroids are also fundamental in the understanding of geodesic flows of general non-invariant connections. 

There is a dual picture on the algebraic side, where the manifold $M$ is replaced by the commutative ring of scalar functions $R=C^\infty(M,\R)$, and vector bundles are studied through their space of sections $L=\Gamma(E)$, which  is algebraically described as an $R$-module, with the product of a scalar function with a section defined point wise. The Serre-Swan theorem states that the category of finitely generated projective modules over $R$ is a faithful representation of the category of vector bundles $E\rightarrow M$. Informally: \emph{`projective modules are like vector bundles'}. 
The geometric concept of Lie algebroids are described on the algebraic side as Lie-Rinehart algebras, 
and
post-Lie algebroids as \emph{tracial post-Lie-Rinehart algebras}, which is the main algebraic object of study in this paper. The trace condition, defined below,  is fulfilled for any finitely generated projective module, as those arising in the Serre-Swan theorem.

From the algebraic point of view, planar trees generate the free post-Lie algebra \cite{MuntheKaas13opl}, with numerous applications in various fields (see, for instance, \cite{Vallette07hog, Bai10ngl, Jacques23pla}). On the other hand, aromatic trees yield the free tracial pre-Lie-Rinehart algebra \cite{Floystad20tup}, and the use of aromas was recently extended to a variety of contexts \cite{Laurent21ocf, Zhu25aac}.
The natural generalisation of these structures is post-Lie-Rinehart algebras, and have been recently mentioned in \cite{Guo21iag, Jacques23pla, Busnot25pha, Rahm26tup}.
In numerical analysis, such structures are naturally used with their universal enveloping algebras. Post-Lie and pre-Lie-Rinehart algebras respectively give rise to post-Hopf algebras \cite{MuntheKaas08oth} and pre-Hopf algebroids \cite{Bogfjellmo19aso, Bronasco22cef}.
For post-Lie-Rinehart algebras, the universal enveloping algebra yields a post-Hopf algebroid \cite{Busnot25pha}.
We show in the present paper that the newly defined planar aromatic trees generate the free tracial post-Lie-Rinehart algebra.

The article is organized as follows.
In Section \ref{section:LR_algebra}, we recall the definition and properties of Lie-Rinehart algebra. We specify such structures given an affine connection to define tracial Lie-Rinehart algebras and the divergence in Section \ref{section:connection}.
In Section \ref{section:pLR_algebra}, we define the main algebraic structures of this article, that are, post-Lie-Rinehart algebras.
Section \ref{sec:planar_AT} introduces the planar aromatic trees, their algebraic properties, and shows that they are the free tracial post-Lie-Rinehart algebra.
The approach is applied successfully in Section \ref{sec:num} for the design of Lie-group methods that preserve divergence-free features up to a high order.
Future works are discussed in Section \ref{sec:Conclusion}.

\section{Lie-Rinehart Algebras, trace and divergence}\label{section:LR_algebra}
This section introduces the algebraic framework
needed to define divergence and trace in a coordinate-free setting.
In particular, we aim to encode the notion of divergence of vector
fields (and later, volume preservation) in purely algebraic terms
within Lie--Rinehart algebras. These constructions generalise
the classical notion of divergence of vector fields on a manifold.

\subsection{Lie-Rinehart Algebras}

Let $k$ be a field of characteristic $0$.

\begin{definition}\label{de:LR} A $k$-Lie—Rinehart algebra is a tuple $(R,L,\rho)$ where $R$ is a commutative $k$-algebra and $\left(L,\left[\cdot,\cdot\right]\right)$ is a $k$-Lie algebra such that:
\begin{enumerate}[label=(\roman*)]
\item $L$ is a $R$-module,

\item $R$ is a $L$-module where the action of $L$ on $R$ is given by the {\em anchor map} $\rho$ viz, a Lie morphism,
\begin{align*}
\rho : L \longrightarrow Der(R,R)	,
\end{align*}
where $Der(R,R)$ stands for Lie algebra of derivations from $R$ to $R$,

\item (Leibniz Rule) : For all $f \in R$ and $X,Y \in L$,
\begin{align*}
	\left[X, fY\right] = (X.f)Y + f\left[X,Y\right].
\end{align*}
\end{enumerate}
\end{definition}

\begin{remark}
	For the sake of brevity, the action of $L$ on $R$ via the anchor map is denoted as the following:
	\begin{align*}
		X.f = \rho(X)(f),\quad X \in L,\quad f \in R.
	\end{align*}
\end{remark}


\subsection{Connections}
\label{section:connection}

We refer to Appendix~A for a general discussion of  connections. An affine/linear connection on the Lie-Rinehart algebra can alternatively be described in a more elementary manner following~\cite{Floystad20tup} which the current paper also adheres to.

\begin{definition}\label{de:conn} A linear connection on a $R$-module $N$ with respect to Lie-Rinehart algebra $(R,L)$ is a $R$-linear map
\begin{align}\label{eq:conn}
\nabla : L &\longrightarrow \End_k(N)\\
         X &\longmapsto \nabla_X
\end{align}
such that for all $f \in R $ and $X,Y \in L$:
\begin{align}
\nabla_X(fY) = (X.f)Y + f\nabla_XY.
\end{align}
\end{definition}

Observe that the associative algebra $\End(N)$ can be endowed a Lie algebra structure via the commutator $\llbracket\cdot, \cdot\rrbracket$.

\begin{definition}\label{de:parallel-tensor} Let $N$ be a $R$-module equipped with a connection $\nabla$  with respect to Lie-Rinehart algebra $(R,L)$. An $m$ rank $k$-multilinear map $\mathcal{W} : N^{\otimes m} \longrightarrow M$ is called parallel if for all $X \in L$ and $Y_1,Y_2,\ldots, Y_m$
\begin{align}\label{eq:parallel-tensor}
\nabla_X(\mathcal{W}(Y_1,Y_2,\ldots,Y_m))&= \mathcal{W}(\nabla_XY_1,Y_2,\ldots,Y_m) + \mathcal{W}(Y_1,\nabla_XY_2,\ldots,Y_m) + \cdots  \\ 
&\hspace*{1.5 in} \cdots + \mathcal{W}(Y_1,Y_2,\ldots,\nabla_XY_m) \nonumber.
\end{align}
\end{definition}

\begin{definition}\label{de:curvature}
	The curvature of the connection $\nabla$ on the $R$-module $N$ is the alternating bilinear map defined as 
	\begin{align*}
		\mathcal{R} : L \wedge L  &\longrightarrow End(N) \\
		(X,Y) &\longmapsto \mathcal{R}(X,Y),
	\end{align*}  
	where $\mathcal{R}(X,Y) = \llbracket \nabla_X, \nabla_Y\rrbracket - \nabla_{\left[X,Y\right]}$.
\end{definition}

\begin{remark} The curvature $\mathcal{R} = 0 $ if and only if $N$ is a module over the Lie algebra $L$ and the action is given by the connection $\nabla$. Since $N$ is acted by both $R$ and $L$ of the Lie-Rinehart pair $(R,L)$ in this case, $N$ is termed as a module over the Lie-Rinehart algebra $(R,L)$.  
\end{remark}

\begin{definition}[Lie-Rinehart module] A \(R\)-module \(N\) is a module over the Lie-Rinehart pair \((R,L)\) or simply called an \((R,L)\)-module), with a connection \(\nabla\)  if
\begin{itemize}
\item \(N\) is a module over the Lie algebra \(L\) where the action is given by the connection \(\nabla\).
\[
\nabla_{[X_1,X_2]}Y = (\nabla X_1 \circ \nabla X_2 -  \nabla X_2 \circ \nabla X_1) Y
\] 
for all \(X_1, X_2 \in L\) and \(Y \in N\).

\item The action of the commutaive algebra \(R\) on \(N\) and the action on Lie algebra \(L\) on \(N\) satisfy the Lie-Rinehart pair compatibility.  For all \(X_1, X_2 \in L\), \(f \in R\) and \(Y \in N\):
\[
\nabla_{[X_1,fX_2]}Y = (X_1.f)\nabla_{X_2}Y + f \nabla_{[X_1,X_2]} Y
\]
\end{itemize}
\end{definition}

\begin{ex}
	For a Lie-Rinehart pair $(R,L)$, a primary example of Lie-Rinehart module over $(R,L)$ is the module $R$. Looking at $R$ as $R$-module, define the connection 
	\begin{align*}
		\nabla : L &\longrightarrow \End_k(R) \\
		X &\longmapsto \nabla_X,
	\end{align*}
	where $\nabla_X(f) = X.f $ for all $X \in L$ and $f \in R$. The flatness of the connection $\nabla$ follows from the definition of $(R,L)$ being a Lie-Rinehart pair. Hence, $R$ is a Lie-Rinehart module over $(R,L)$.
\end{ex}

\begin{definition}[Morphism of Lie-Rinehart modules] Let \(M,N\) be \((R,L)\)-modules with connection \(\nabla\), \(\nabla'\) respectively. A \(R\)-module morphism \(\alpha : M \longrightarrow N\) is a morphism of \((R,L)\)-modules if \(\alpha\) is a morphism of \(L\)-modules viz. for all \(X \in L\) and \(Y \in M\)
	\[
	 \alpha(\nabla_X Y) = \nabla'_X (\alpha(Y)).
	\]	
\end{definition}
\medskip

For a connection $\nabla$ defined on the $R$-module $L$ with respect to Lie-Rinehart pair $(R,L)$, the torsion $T$ of the connection $\nabla$ is defined as
\begin{align*}
	T : L\wedge L &\longrightarrow L \\
	(X,Y) &\longmapsto T(X,Y),
\end{align*}
where  $T(X,Y) =  \nabla_XY - \nabla_YX -\left[X,Y\right]$.

Let $N$ be a $R$-module endowed with a connection $\nabla$ with respect to the Lie-Rinehart pair $(R,L)$. The $R$-module $\End_R(N)$ is canonically equipped with the connection $\hat{\nabla}$(with respect to $(R,L)$), defined as 
\begin{align}\label{eq:conn-on-endo}
	(\hat{\nabla}_X \nu)(Y) = \nabla_X(\nu(Y)) - \nu(\nabla_XY),  
\end{align} 
for all $X \in L, Y \in N$ and $\nu \in \End_R(N)$.
From $(\ref{eq:conn-on-endo})$, it is straightforward to verify that for all $X \in L$ the connection $\hat{\nabla}_X$ acts a derivation on the algebra $\End_R(N)$ viz., for all $\nu, \mu \in \End_R(N)$:
\begin{align}\label{eq: Leibnitz-on-End}
	\hat{\nabla}_X(\nu \circ \mu) &= \hat{\nabla}_X(\nu) \circ \mu + \nu \circ \hat{\nabla}_X(\mu).
\end{align}

\begin{proposition}\cite{Floystad20tup}
If $N$ is a Lie-Rinehart module with respect to the Lie-Rinehart pair $(R,L)$, then $\End_R(N)$ is also a Lie-Rinehart module over $(R,L)$.
\end{proposition}

For completeness, we recall the Chevalley—Eilenberg complex associated to Lie—Rinehart algebras in Appendix~A.

\subsection{Tracial Lie-Rinehart Algebras}
We now introduce algebraic analogues of covariant differentiation
and divergence. The map $d$ plays the role of taking the covariant derivative
of a section, while the trace map $\tau$ will encode the contraction
leading to divergence.

Given a connection $\nabla$ on $L$, the assignment
$X \mapsto dX$ associates to each element $X \in L$
an $R$-linear endomorphism of $L$, analogous to the
covariant derivative viewed as a $(1,1)$-tensor.
 
 We proceed in three steps: first defining an algebraic covariant
derivative, then introducing the algebra of elementary endomorphisms,
and finally defining trace and divergence.

\begin{definition}\label{de:d}
Let $L$ be the $R$-module endowed with connection $\nabla$ over the Lie-Rinehart algebra $(R,L)$ and then define $k$-linear maps 
\begin{align*}
	d : L &\longrightarrow \End_{R}(L) \\
	X &\longmapsto dX,
\end{align*} 
where $dX(Z) = \nabla_ZX$ for all $Z,X \in L$ and 
\begin{align*}
\delta : L &\longrightarrow \End_R(L) \\
X &\longmapsto \delta X,
\end{align*}
where $\delta X(Y) = T(Y,X)$ for all $X,Y \in L$.
\end{definition}

The $k$-linear map $\delta$ satisfies the following properties
\begin{enumerate}
  \item For all $X,Y \in L$ : $\delta X(Y) +\delta Y(X) = 0$,
  \item If \(T\) is parallel, then for all $X,Y_1,Y_2$ 
  \begin{align*}
  d(\delta Y_1(Y_2))(X) = \delta(dY_1(X))(Y_2) -\delta(dY_2(X))(Y_1).
  \end{align*}
\end{enumerate}
	
Define $\El_R(L)$ as the $R$-algebra of {\em elementary $R$-module endomorphisms} generated by the set of $R$-linear maps 
\begin{align*}
\big\lbrace\hat{\nabla}_{X_1}\hat{\nabla}_{X_2}\cdots\hat{\nabla}_{X_k}dY, \hat{\nabla}_{X_1}\hat{\nabla}_{X_2}\cdots\hat{\nabla}_{X_k}\delta Y \,:\; k \geq 0 \, , \, X_1,X_2,\ldots,X_k,Y \in L \big\rbrace.
\end{align*}
By definition, $\El_R(L)$ is a subalgebra of the algebra of $\End_R(L)$. If $L$ is a Lie-Rinehart module, $\El_R(L)$ is a Lie-Rinehart submodule of the Lie-Rinehart module $\End_R(L)$. This algebra collects operators generated by iterated covariant
derivatives and torsion. To define divergence, we require a notion of trace on such endomorphisms.  $\El_R(L)$ serves as the domain for the trace. Note that it is not possible to define a trace on the full group of endomorphisms.

\begin{definition}\label{de:tracial-LR} A Lie-Rinehart algebra $(R,L)$ is {\em tracial} if $L$ is a Lie-Rinehart module over $(R,L)$ with connection $\nabla$ and a $R$-linear map $\tau: \El_R(L) \longrightarrow R$ such that
	\begin{enumerate}[label=(\roman*)]
		\item $\tau$ is a homomorphism of $(R,L)$-modules from $(\El_R, \hat{\nabla})$  to $(R, \nabla)$, that is, for \(nu \in \El_R\) and \(X \in L\):
		\[
		 \tau(\hat\nabla_X \nu) = \nabla_X (\tau(\nu)) = X. (\tau(\nu))
		\]
		
		\item for all $\nu, \mu \in \El_R(L)$
		\begin{equation*}
			\tau(\nu \circ \mu) = \tau(\mu \circ \nu),
		\end{equation*}
viz., $\tau$ is invariant under the action of cyclic permutations.
	\end{enumerate}
\end{definition}

We now arrive at the central notion of this section.
\begin{definition}\label{de:div} The {\em divergence} map on $(L,\nabla)$ in a tracial Lie-Rinehart pair $(R,L,\tau)$ is
\begin{align}\label{eq:div}
\Div : L &\longrightarrow R \nonumber \\
X &\longmapsto \Div(X) = (\tau \circ d)(X).
\end{align}
Therefore,  we have $\Div = \tau \circ d$.
\end{definition}

This notion of divergence will be realised explicitly in terms of planar
aromatic trees in Section~5.

\begin{remark}[Geometric interpretation of the trace and divergence]
In the classical geometric setting where $R = C^\infty(M)$ and $L = \mathfrak{X}(M)$ for a smooth manifold $M$, a nowhere-vanishing volume form $\mu$ induces a trace functional $\tau$ on suitable endomorphisms of $L$. In this case, the divergence defined by
\[
\operatorname{div}(X) = \tau(dX)
\]
coincides with the standard divergence associated to $\mu$, characterized by
\[
d(\iota_X \mu) = (\operatorname{div} X)\,\mu.
\]
Thus, tracial Lie--Rinehart algebras can be viewed as an algebraic abstraction of divergence relative to a volume form.

From this perspective, the trace map $\tau$ captures the infinitesimal
change of volume induced by the connection $\nabla$, i.e. its trace
(determinant) component. Thus, divergence depends on both $\nabla$ and
$\tau$, corresponding geometrically to the choice of a connection together
with a volume form.
\end{remark}

\section{Post-Lie-Rinehart algebras}
\label{section:pLR_algebra}

Let $V$ be a $k$-vector space and let $\triangleright : V^{\otimes 2} \longrightarrow V$ be a magmatic product defined on $V$. Then, the {\em Associator} map is defined as
\begin{align*}
	\Ass_{\triangleright} : V^{\otimes 3} &\longrightarrow V \\
	(a,b,c) &\longmapsto a \triangleright (b \triangleright c) - (a \triangleright b) \triangleright c.
\end{align*}
Post-Lie-Rinehart algebras were introduced  in \cite{Guo21iag, Jacques23pla} and further studied in \cite{Busnot25pha}, where their link with post-Hopf algebroids is discussed.

\begin{definition}\label{de:post-LR} A post-Lie-Rinehart algebra is a Lie-Rinehart algebra $(R,L)$ endowed with flat connection $\nabla : L \longrightarrow  \End_k(L)$ with parallel torsion or equivalently 
\begin{align*}
\nabla: L \otimes L &\longrightarrow L \\
(X,Y) &\longmapsto X \triangleright Y := \nabla_X Y.
\end{align*}
\end{definition}
Hence for all $X,Y,Z \in L$:
\begin{enumerate}[label=(\roman*)]\label{enum:post-Lie}
\item $ \Ass_{\triangleright}(X,Y,Z) - \Ass_{\triangleright}(Y,X,Z) = T(Y,X) \triangleright Z$ \quad \text{($\mathcal{R}=0 \Leftrightarrow L$ is a Lie-Rinehart module)},

\item $X \triangleright T(Y,Z) = T((X \triangleright Y),Z) + T(Y, (X \triangleright Z))$ \quad \text{(torsion $T$ is parallel)}.
\end{enumerate}

\begin{ex}[The covariant derivation algebra]
This example illustrates how post-Lie-Rinehart structures arise naturally
from connections on manifolds and govern the algebra of flows and geodesics.
Let $(M,\nabla)$ be a manifold with a general $\TM$-connection $\nabla$ (not necessarily  flat nor constant torsion). Associated with this geometric space is a natural post-Lie-Rinehart algebra which is governing the algebra of flows of vector fields and geodesics on $M$. We give a brief introduction and refer to~\cite{AlKaabi22aao} for details. 
Define the tensor algebra of vector fields 
\[
T(\XM) := \R \oplus \XM \oplus (\XM\ppr\XM)\oplus (\XM\ppr\XM\ppr\XM)\oplus\cdots = \bigoplus_{r=0}^\infty T_0^r,
\]
where ${\XM\ppr\XM := \XM\otimes_R \XM}$. We follow~\cite{Oudom08otl,Ebrahimi15otl} and extend the connection to an $\R$-linear product $\tr$ on $T(\XM)$, as the Guin-Odoum construction:
for $\varphi\in \R$, $X, Y\in \XM$ and $W_1,W_2\in T(\XM)$ let
\begin{align*}
X\tr Y &= \nabla_X Y,\\
\varphi\tr W_1 &= \varphi W_1,\\
X\tr (W_1\ppr W_2) &= (X\tr W_1)\ppr W_2 + W_1\ppr (X\tr W_2),\\
(X\ppr W_1)\tr W_2 &= X\tr(W_1\tr W_2)-(X\tr W_1)\tr W_2. 
\end{align*}
The product $\tr$ is $\R$ multi-linear in left argument and $\RR$ linear in the right. \\
Let 
$\Delta \colon T(\XM)\rightarrow T(\XM)\otimes T(\XM)$ denote the de-shuffle co-product and define the associative Grossman—Larson product
$*\colon T(\XM)\otimes T(\XM)\rightarrow T(\XM)$ as
\[A*B = \sum_{\Delta(A)}A_{(1)}\ppr \left(A_{(2)}\tr B\right).\]
Then $W_1\tr\left(W_2\tr W_3\right) = \left(W_1 * W_2\right)\tr W_3$. 

We call $\D(\M) := (T(\XM), \ppr\ , \tr, \Delta, *)$ the \emph{covariant derivation algebra} on $(\M,\nabla)$. It acts naturally as higher order covariant derivations on the tensor bundles $T^r_s$ over $M$. The primitive elements 
\[\g = \{W\in \D(\M)\, \colon\, \Delta(W) = W\otimes I + I\otimes W\}\]
acts as first order covariant derivations, in particular $\g$ acts on $R=T^0_0$ as derivations, which defines
an anchor map 
\[\rho\colon \g\rightarrow \Der(R,R)=\XM.\]
The lifted connection $\tr$ on $\g$ is flat with parallel torsion, hence $(\g,[-,-],\tr)$ is a post-Lie algebra with 
$[A,B]= A\ppr B-B\ppr A$, and $(\g,\llbracket-,-\rrbracket)$ is a Lie algebra with 
\[\llbracket A,B\rrbracket = A*B-B*A = [A,B] + A\tr B-B\tr A.\]
The tuple $(R,L\!=\!\g,\rho,\tr)$ is a tracial post-Lie-Rinehart algebra with enveloping algebra $\D(\M)$. 
The exponential with respect to $\ppr$ yields pullback series along geodesics, while the exponential with respect to $*$ yields pullback series along the exact flow of a vector field. 

It might be a surprise that a general $\TM$-connection $\nabla$ with torsion $T^\nabla$ and curvature $\RR^\nabla$ lifts to a flat post-Lie connection $\tr$ on $\g$. The curvature  $\RR^\nabla$  and torsion $T^\nabla$ appear through a split exact sequence of Lie algebras. Let  $[-,-]_J$ be the Jacobi bracket on vector fields $\XM$ and for $X,Y\in \XM$
define the curvature form
\[s(X,Y) := \llbracket X,Y\rrbracket - [X,Y]_J, \]
thus $s$ is a tensor of mixed type $s\in T^2_2\oplus T^1_2$. 
Then $\RR^\nabla(X,Y)Z = s(X,Y)\tr Z$.

 Let ${{\mathcal I} = \langle s(X,Y) \colon X,Y\in \XM\rangle}$ be the two-sided ideal of the Grossman—Larson Lie algebra $(\g,\llbracket-,-\rrbracket)$ generated by the curvature forms. Then there is a split short exact sequence
\[
\begin{tikzcd}
0 \arrow[r]
& {\mathcal I} \arrow[r]
& \mathfrak{g} \arrow[r, "\rho"]
& {\XM} \arrow[r] \arrow[l, bend left=50,  "i"]
& 0.
\end{tikzcd}
\]
Thus the G-L Lie algebra $(\g,\llbracket-,-\rrbracket)$ decomposes in a semi direct product
\[\g \cong {\mathcal I}\rtimes \XM, \]
where the left part contains the curvature forms and the right part the torsion forms of $\nabla$. 
As an example, we have for $X,Y\in \XM$ that  $\rho([X,Y]) = -T^\nabla(X,Y)$ and $(1-\rho)([X,Y]) = s(X,Y)$. 
We refer to~\cite{AlKaabi22aao} and forthcoming papers for more details. 
\end{ex}

\begin{ex} Let $M^n$ be an $n$-dimensional real manifold and $U$ be an open subset of $M^n$. Let $R = \mathcal{C}^{\infty}(U)$ and $L = \mathbf{\Gamma}(TM_{|U})$ be the space of local sections of the tangent bundle restricted to $U$ with nonholonomic basis $\left(e_1,e_2,\ldots,e_n \right)$ equipped with a flat connection with parallel torsion 
\begin{align*}
\nabla : L^{\otimes 2} &\longrightarrow L \\
(X,Y) &\longmapsto X \triangleright Y,
\end{align*}
such that for all $i,j,k = 1,2,\ldots,n$,
\begin{align*}
\left[e_i,e_j\right] = \sum_{k=1}^nc_{ij}^ke_k,  
\end{align*}
where $c_{ij}^k$ are {\bf structure constants} and 
\begin{align*}
e_i \triangleright e_j = \nabla_{e_i}(e_j) = \sum_{k=1}^n\Gamma_{ij}^k e_k,
\end{align*}
where $\Gamma_{ij}^k$ are {\bf connection coefficients}. The torsion tensor is hence computed as
\begin{align*}
T(e_i,e_j) = \sum_{k=1}^n (\Gamma_{ij}^k - \Gamma_{ji}^k - c_{ij}^k)e_k.
\end{align*}
Thus $\left(\mathcal{C}^{\infty}(U), \mathbf{\Gamma}(TU), \nabla\right)$ is a post-Lie-Rinehart algebra and
for all $X,Y \in L$ with $X = \sum_{i=1}^n X^i e_i$ and $Y = \sum_{j=1}^n Y^j e_j$ where $X^i, Y^j \in R$: 
\begin{align*}
X \triangleright Y = \sum_{j=1}^n \left(\sum_{i=1}X^i(e_i.Y^j)\right)e_j + \sum_{k=1}^n \left(\sum_{i=1}^n \sum_{j=1}^n X^iY^j \Gamma_{ij}^k\right) e_k.
\end{align*} 
Thus for a vector field $Y = \sum_{j=1}^n Y_j e_j$, the corresponding endomorphism $dY$ maps 
\begin{align*}
e_i \longmapsto \sum_{j=1}^n \left(e_i.Y^j\right)e_j + \sum_{k=1}^n \left(\sum_{j=1}^n Y^j \Gamma_{ij}^k\right) e_k.
\end{align*}
Therefore the divergence of $Y$ is given by
\begin{align*}
\Div (Y) = \tau(dY) = \sum_{i=1}^n \left(e_i.Y^i\right) + \sum_{i=1}^n \left(\sum_{j=1}^n \Gamma_{ij}^i Y^j\right).
\end{align*}
\end{ex}

\begin{proposition}\label{prop:gen-El} In a post-Lie-Rinehart algebra $(R,L)$, the algebra of elementary $R$-module endomorphisms, $\El_R(L)$ is generated by $\{dX, \delta X:\, X \in L\}$.	
\end{proposition}

\begin{proof} The proof proceeds by induction on the length $n$ of arbitrary generators of $\El_R(L)$ of the form $\hat{\nabla}_{Y_n}\hat{\nabla}_{Y_{n-1}}\cdots \hat{\nabla}_{Y_1}dX$ and $\hat{\nabla}_{Y_n}\hat{\nabla}_{Y_{n-1}}\cdots \hat{\nabla}_{Y_1}\delta X$ for all $X,Y_1,Y_2,\ldots,Y_n \in L$.

\underline{(Base Case)$n=1$}: For $X,Y,Z \in L$
\begin{enumerate}[label=(\roman*)]
\item \begin{align*}
(\hat{\nabla}_{Y}dX)(Z) &= \nabla_{Y}(dX(Z)) - dX(\nabla_YZ)\\
&= Y \triangleright(Z \triangleright X) - (Y \triangleright Z) \triangleright X \\
&= Z \triangleright (Y \triangleright X) - (Z \triangleright Y) \triangleright X + T(Z,Y) \triangleright X\\
&= \Big\lbrace d(Y \triangleright X) - dX \circ dY + dX \circ \delta Y\Big\rbrace(Z) \\
&= \Big\lbrace d(Y \triangleright X) - dX \circ (dY - \delta Y)\Big\rbrace(Z).
\end{align*}

\item \begin{align*}
(\hat{\nabla}_Y \delta X)(Z) &= \nabla_{Y}(\delta X(Z)) - \delta X(\nabla_{Y}(Z)) \\
&= \nabla_{Y} (T(Z,X)) - T(\nabla_Y Z, X) \\
&= Y \triangleright T(Z,X) - T( Y \triangleright Z, X) \\
&= T(Z, Y\triangleright X) = \delta(dX(Y)) (Z).
\end{align*}
\end{enumerate}

To check for $\underline{n = 2}$, 
\begin{align*}
\hat{\nabla}_{Y_2}\hat{\nabla}_{Y_1}dX &= \hat{\nabla}_{Y_2}(d(Y \triangleright X) - dX \circ (dY - \delta Y) \quad (\text{from Case}\, n=1) \\
&\overset{(\ref{eq: Leibnitz-on-End})}{=}\hat{\nabla}_{Y_2}d(Y\triangleright X) - \hat{\nabla}_{Y_2}dX \circ (dY - \delta Y), 
\end{align*}
where all the terms are generated by $\{dX,\delta X : X\in L \}$ in view of the case $n=1$. Successively, the cases for $n=2,3,\ldots $ can be proven in the same way. For,  
\begin{align*}
\hat{\nabla}_{Y_2}\hat{\nabla}_{Y_1}\delta X &= \hat{\nabla}_{Y_2}(\delta(dX(Y_1)),
\end{align*}
and is reduced to Case $n=1$. Likewise, the cases for $n=2,3,\ldots$ can iteratively be reduced to the Case $n=1$. 
\end{proof}

\section{Planar aromatic trees and the free tracial post-Lie-Rinehart algebra}
\label{sec:planar_AT}

This section reviews planar rooted trees and introduces planar aromatic trees, which provide an explicit description of the free tracial post-Lie-Rinehart algebra.

\subsection{Planar rooted trees and free post-Lie-Rinehart algebra}

The planar rooted trees are a convenient graphical representation of the free magma \cite{MuntheKaas13opl}.
\begin{definition}
\label{def:planar_trees}
A planar tree, gathered in the set $T$, is defined by induction:
\[
\forestA\in T,\quad (t_1\cdots t_p)\graft \forestB\in T,\quad t_i\in T,
\]
where $(t_1\cdots t_p)\graft$ is an abstract notation for the ordered list of sub-trees $t_1,\dots, t_p$.
The notation $\forestC$ stands for a vertex of the tree, that are arbitrarily labeled (with different labels). The set of vertices $V(t)$ of a tree $t=(t_1\cdots t_p)\graft \forestD$ contains the root $\forestE$ and the vertices of its subtrees $t_k$. The number of vertices of a tree, denoted $|t|$ is called the order of the tree.

Let $C$ be a finite set of decorations, also called colours in numerics.
A decorated planar tree is a tuple $(t,\phi)$ with a tree $t$ and a map $\phi : V \longrightarrow C$. The decorated trees are gathered in the set $T_C$ and the space $\TT_C=\Span(T_C)$. Similarly, the subset/subspace of decorated planar trees of order $N$ are denoted $T_C^N$ and $\TT_C^N$.
\end{definition}

The index $C$ is omitted for examples and numerical applications, where one is mainly interested in one colour $C=\{\forestF\}$.
Planar trees are drawn with the root at the bottom. The planar trees of order up to five in $T$ are the following:
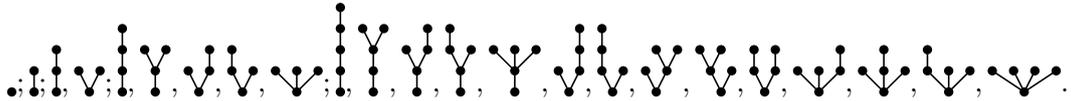
\[
\forestG;
\forestH;
\forestI,\forestJ;
\forestK,\forestL,\forestM,\forestN,\forestO;
\forestP,\forestQ,\forestR,\forestS,\forestT,\forestU,\forestV,\forestW,\forestX,\forestY,\forestAB,\forestBB,\forestCB,\forestDB.
\]
We emphasize that the following trees are different as they are defined by induction as
\[
\forestEB=(\forestFB\cdot(\forestGB\graft\forestHB))\graft\forestIB\,, \quad \forestJB=((\forestKB\graft\forestLB)\cdot\forestMB)\graft\forestNB.
\]

The left grafting product $\graft\colon \TT_C\times \TT_C\rightarrow \TT_C$ is defined on trees and then extended by bilinearity as
\[t_2\graft t_1=\sum_{v\in V(t_1)} t_2\graft_v t_1,\]
where $t_2\graft_v t_1$ attaches the root of $t_2$ to the node $v$ of \(t_2\) on the left.
For instance, we find
\[
\forestOB\graft \forestPB=\forestQB+\forestRB+\forestSB+\forestTB.
\]
Let us consider the free Lie algebra spanned by decorated planar trees $(\Lie(\TT_C),[-,-])$. The grafting product is extended on $\Lie(\TT_C)$ by
\begin{align*}
t_3\graft [t_2,t_1]&=[t_3\graft t_2,t_1]+[t_2,t_3\graft t_1],\\
[t_3,t_2]\graft t_1&=\Ass_{\graft}(t_3,t_2,t_1)-\Ass_{\graft}(t_2,t_3,t_1).
\end{align*}

\begin{proposition}[\cite{MuntheKaas13opl}]
\label{prop:free_post_Lie}
The space $(\Lie(\TT_C),[-,-],\graft)$ is the free post-Lie algebra over the set $C$, that is, for every set map \(F : C \longrightarrow L\) where \((L, [-,-], \triangleright)\) is a post-Lie algebra, there exists a unique post-Lie (epi-)morphism \(\F: \Lie(\TT_C) \longrightarrow L\) extending the map \(F\).
\[\begin{tikzcd}
	C & {( L, [-,-] ,\triangleright)} \\
	{(\Lie(T_C), [-,-],\curvearrowright)}
	\arrow["F", from=1-1, to=1-2]
	\arrow["\iota"', hook, from=1-1, to=2-1]
	\arrow["{\F}"', dashed, two heads, from=2-1, to=1-2]
\end{tikzcd}\]
\end{proposition}

\begin{remark}
Let us now consider the universal enveloping algebra $\mathcal U(\Lie(\TT_C))$, equipped with its product. It naturally is a cocommutative Hopf algebra with the unshuffle coproduct \(\Delta_{\shuffle}\) and is isomorphic to the Hopf algebra $(\mathbb T(\TT_C),\cdot,\Delta_\shuffle)$ by the Cartier-Quillen-Milnor-Moore theorem \cite{Milnor65ots}.
The grafting product is extended to $\mathbb T(\TT_C)$ using the (post-Lie) Guin-Oudum construction \cite{Oudom08otl, Ebrahimi15otl} and $(\mathbb T(\TT_C),\oast,\Delta_\shuffle)$, where \(\oast\) is the Grossman-Larson product. The enveloping algebra of a post-Lie algebra endows a cointeracting structure. The tuple \((\mathbb T (\TT_C), \cdot, \Delta_\shuffle)\) is a Hopf algebra in the category of \((\mathbb T (\TT_C), \oast, \Delta_{\shuffle})\)-comodules. This cointeracting structure is termed as a post-Hopf algebra \cite{Li23pha}.
Note that this construction is compatible with the notation for the inductive definition of planar trees in Definition \ref{def:planar_trees}.
The Hopf algebras associated to planar trees have been extensively studied in the numerical literature \cite{Iserles00lgm, MuntheKaas08oth, Lundervold11hao, Ebrahimi24aso}, where they are used for representing the Taylor expansion of Lie-group methods.
\end{remark}

For the Lie algebra \(\Lie(\TT_C)\), the left adjoint Lie morphism for \(t \in \Lie(\TT_C)\) is given by
\begin{align*}
\delta_{t} : \Lie(\TT_C) &\longrightarrow \End(\Lie(\TT_C)) \\
\eta &\longmapsto \left[t, \eta\right]
\end{align*}

Let $T_C^0$ be the set of decorated planar trees with one free edge. We denote $\hat{T}_C^0$ the subset of decorated planar trees with the free edge attached to the root. For instance, we find
\[
\forestUB\in T_{\{\forestVB,\forestWB\}}^0,\quad
\forestXB\in \hat{T}_{\{\forestYB,\forestAC\}}^0\subset T_{\{\forestBC,\forestCC\}}^0.
\]
The trees with a free edge can be seen as endomorphisms in $\End(\Lie(\TT_C))$ that act on trees by grafting on the free edge, 
\[
\forestDC\colon \forestEC \mapsto \forestFC
\]
The (endomorphism) algebra generated by the \(\Lie(\TT_C)\) (via the left adjoint \(\delta_{\cdot}\)) and  \(\mathcal T_C^{0}\) is defined by
\[
\forestGC\circ\forestHC=\forestIC,\quad
\forestJC\circ\delta_{\forestKC}=\forestLC-\forestMC.
\]
This gives us in particular
\[
\forestNC\colon [\forestOC,\forestPC] \mapsto \forestQC-\forestRC.
\]
This yields that $\Tfe=\Span(T_C^0)$ equipped with the composition $\circ$ is a subalgebra of $(\End(\TT_\CC),\circ)$.
Moreover, the subalgebra $(\Tfe,\circ)$ is generated by $(\hat{\TT}_C^0,\circ)$ (with $\hat{\TT}_C^0=\Span(\hat{T}_C^0)$) as there is a unique way to decompose an element of $T_C^0$ into a composition of elements of $\hat{T}_C^0$:
\[
\forestSC=\forestTC\circ \forestUC\circ \forestVC.
\]

Let the map $d\colon \TT_C\rightarrow \Tfe$ that sums over all possible ways to graft from the left a free edge to the input:
\begin{align}\label{eq:differential-d}
d\forestWC=\forestXC+\forestYC+\forestAD+\forestBD.
\end{align}

\begin{lemma}
\label{lemma:elem_end}
The algebra of elementary module morphisms \(\El_k(\Lie(\TT_C))\) for the post-Lie-Rinehart algebra \((k,\Lie(\TT_C), \left[-,-\right], \graft)\) is the algebra generated by \(\hat{\TT}_C^0\), the subalgebra of trees with a free edge at the root and the Lie subalgebra \(\delta(\Lie(\TT_C))\). 
\end{lemma}

\begin{proof}
This follows from (mimicking) Proposition~\ref{prop:gen-El}. 
\end{proof}
\medskip

The grafting operation \(\graft\) on \(\Lie(\TT_C)\) is extended as the action of the Lie algebra on the trees with a free edge
\[
\graft : \Lie(\TT_C) \times  \TT_C^0 \longrightarrow \TT_C^0
\]
as for all \(t_1, t_2 \in \Lie(\TT_C)\) and \(t \in \TT_C^0\) as
\begin{align}\label{eq:graft-planar-tree-tree-with free-edge}
	t_1 \graft a &= \sum_{v \in V(a)} t_1 \graft_v t, \nonumber \\
	[t_1,t_2] \graft t &= \Ass_{\graft}(t_1,t_2,t) - \Ass_{\graft}(t_2,t_1,t).
\end{align}
The grafting of a planar tree \(t_1\) on a planar tree with a free edge \(t\) is the sum of all possible grafting \(t_1\) on all the vertices of \(t\). The important remark is that {\em the free edge of \(a\) is not disturbed at all.}
For example, one finds
\begin{align*}
\forestCD \graft \forestDD = \forestED + \forestFD + \forestGD.
\end{align*}

\subsection{Planar aromatic trees and free tracial post-Lie-Rinehart algebras}

\begin{definition}
A planar aroma is a list $(t_1,\dots,t_n)_\circlearrowleft$ with $t_i\in \hat{T}_C^0$ that has cyclic invariance:
\[(t_1,\dots,t_n)_\circlearrowleft=(t_2,\dots,t_n,t_1)_\circlearrowleft=\dots=(t_n,t_1,\dots,t_{n-1})_\circlearrowleft.\]
We adopt the convention that the free edge in each $t_p$ is indicated with a cross $\forestHD$.
Let $\AA_C$ be the symmetric algebra generated by planar aromas, whose elements are called multi-aromas.
The $\AA_C$-module of aromatic trees is $\AA\TT_C=\AA_C\otimes \TT_\CC$ and the $\AA_C$-Lie algebra is $\Lie_{\AA_C}(\AA\TT_C)=\AA_C\otimes \Lie(\TT_C)$ equipped with the $\AA_C$-bilinear bracket $[-,-]$.
These vector spaces are graded by the number of vertices.
\end{definition}

An example of planar aroma is the following:
\[
(\forestID,\forestJD,\forestKD)_\circlearrowleft
=(\forestLD,\forestMD,\forestND)_\circlearrowleft
=(\forestOD,\forestPD,\forestQD)_\circlearrowleft.\]
In the Euclidean setting where there are no planar features, the location of the cross would not matter.
We added the notation $\forestRD$ for the sake of clarity as we shall now extend the map $d$ to aromas, which shall yield additional free edges.

The planar aromatic trees are the main object of interest for our numerical interests. For $a\in \AA_C$, $t\in \Lie(\TT_C)$, the associated aromatic tree \(a \otimes t\) is written as $at$ for simplicity.
The grafting product is extended by the Leibniz rule:
\begin{align*}
\graft&\colon \Lie_{\AA_C}(\AA\TT_C)\times \Lie_{\AA_C}(\AA\TT_C)\rightarrow\Lie_{\AA_C}(\AA\TT_C)\\
(a_1 t_1)\graft (a_2 t_2)&=a_1 a_2 (t_1\graft t_2)+a_1 (t_1\graft a_2) t_2,
\end{align*}
where the grafting of $t\in \Lie(\TT_C)$ on a multiaroma satisfies
\begin{align*}
t\graft\textbf{1}&=0,\\
t\graft(t_1,\dots, t_p)_\circlearrowleft&=(t\graft t_1,\dots, t_p)_\circlearrowleft+\dots +(t_1,\dots, t\graft t_p)_\circlearrowleft,\\
[t_1,t_2]\graft a&=\Ass_{\graft}(t_1,t_2,a)-\Ass_{\graft}(t_2,t_1,a),\\
t\graft(a_1a_2)&=(t\graft a_1)a_2+a_1 (t\graft a_2).
\end{align*}
The associated anchor map is
\[
\rho\colon \Lie_{\AA_C}(\AA\TT_C)\rightarrow \Der(\AA_C),\quad
\rho(x)(y)= x\graft y.
\]

The following proposition is an embodiment of the discussion so far in this subsection

\begin{proposition}
The space of planar aromatic forests $(\AA_C, \Lie_{\AA_C}(\AA\TT_C),[-,-], \rho,\graft)$ is a post-Lie-Rinehart algebra.
\end{proposition}

Say the tree $t_1$ is with a free edge (other than the crossed free edge defining the aroma), then we have an aroma with a free edge.
\[
(t_1,\dots, t_p)_\circlearrowleft\colon t\in\Lie(\TT_C) \rightarrow (t_1(t),\dots, t_p)_\circlearrowleft.\]
Here \(t_1(t)\) denotes the grafting of \(t\) on the non-crossed free edge of \(t_1\). For instance, we find
\[
(\forestSD,\forestTD,\forestUD)_\circlearrowleft(\forestVD)=(\forestWD,\forestXD,\forestYD)_\circlearrowleft,
\]
where we recall that the free edges with a cross are the one defining the aroma and cannot be grafted upon. Thus hereafter unless otherwise mentioned, free edges for aromatic planar trees should imply that edges are non-crossed.
\smallskip

Let $\Lie_{\AA_C}^0(\AA\TT_C)$ be the Lie algebra of aromatic trees where exactly one aroma or one tree has a free edge.
The algebra $(\Lie_{\AA_C}^0(\AA\TT_C),\circ)$ forms a subalgebra of $(\End_{\AA_C}(\Lie_{\AA_C}(\AA\TT_C)),\circ)$ by imposing $\AA_C$-linearity.

The map $d$ extends on aromatic trees by the Leibniz rule:
\[
d\colon \Lie_{\AA_C}(\AA\TT_C)\rightarrow \Lie_{\AA_C}^0(\AA\TT_C) \subset \End_{\AA_C}(\Lie_{\AA_C}(\AA\TT_C)),\quad
d(at)=d(a)t+ad(t),
\]
where we extend the map \(d\) defined in (\ref{eq:differential-d}) on multiaromas and commutators by
\begin{align*}
d(\textbf{1})&=0,\\
d(a_1a_2)&=d(a_1)a_2 +a_1d(a_2),\\
d((t_1,\dots,t_n)_\circlearrowleft)&=(d(t_1),\dots,t_n)_\circlearrowleft+\dots+(t_1,\dots,d(t_n))_\circlearrowleft,\\
d([t_1,t_2])&=[d(t_1),t_2] +[t_1,d(t_2)].
\end{align*}
Here \(d(t_1)\)  is the sum of all possible graftings of a  free edge from the left onto all the vertices of the planar tree \(t_1\) and the {\em the crossed free edge is not disturbed at all.} For example, 
\begin{align*}
	d(\forestAE) = \forestBE
\end{align*}

\begin{lemma}
Extend $\delta$ to $\Lie_{\AA_C}(\AA\TT_C)$ by $\AA_C$-linearity. Then, the elementary endomorphisms of the post-Lie-Rinehart algebra of planar aromatic forests are given by
\[\El_{\AA_C}(\Lie_{\AA_C}(\AA\TT_C))= \Lie_{\AA_C}^0(\AA\TT_C)\oplus\delta(\Lie_{\AA_C}(\AA\TT_C)).\]
\end{lemma}

Using the unique decomposition of elements of $T_C^0$ into compositions of elements of $\hat{T}_C^0$, we define the trace and the divergence maps. Examples of computations are presented in Table \ref{table:ex_planar_aroma_tree_comput} and Appendix \ref{app:AT_examples}.

We now define the trace combinatorially on planar aromatic trees,
in a way that reflects the cyclic invariance of the algebra.
\begin{definition}
The trace is the $\AA_C$-linear map $\tau\colon \El_{\AA_C}(\Lie_{\AA_C}(\AA\TT_C))\rightarrow \AA_C$ that vanishes on vanishes on $\delta(\Lie_{\AA_C}(\AA\TT_C))$ and is given on $\Lie_{\AA_C}^0(\AA\TT_C)$ by the following.
If the free edge is on a tree $t$:
\[
\tau(t)=(t_1,\dots,t_n)_\circlearrowleft,\quad t=t_1\circ\dots\circ t_n, \quad t_1,\dots, t_n\in \hat{T}_C^0.
\]
If the free edge is on an aroma:
\[
\tau((t_1,\dots,t_n)_\circlearrowleft t)=(t_1,\dots,t_n)_\circlearrowleft(t).
\]
If the free edge is in a commutator:
\[
\tau([t_1,t_2])=\hat{\tau}([t_1,t_2])(\id),\]
where the auxiliary map $\hat{\tau}$ satisfies
\begin{align*}
\hat{\tau}(t_1\circ\dots\circ t_n)(\delta_x)&=(t_1,\dots,t_n\circ \delta_x)_\circlearrowleft t, \quad t_1,\dots, t_n\in \hat{T}_C^0,\quad x\in \Lie(\TT_C),\\
\hat{\tau}([t_1,t_2])(\delta_x)&=\hat{\tau}(t_1)(\delta_x\circ\delta_{t_2}),\quad t_1\in \Lie^0(\TT_C),\quad t_2\in \Lie(\TT_C).
\end{align*}
The divergence map is then defined by
\[
\Div\colon \Lie_{\AA_C}(\AA\TT_C)\rightarrow\AA_C,\quad \Div=\tau\circ d.
\]
\end{definition}

\begin{figure}[ht]
\begin{longtable}{|C|C|C|}
\hline
t\in \Lie_{\AA}(\AA\TT) & d(t) & \Div(t) \\\hline
\forestCE & \forestDE & (\forestEE)_\circlearrowleft \\\hline
\forestFE & \forestGE+\forestHE & (\forestIE,\forestJE)_\circlearrowleft+(\forestKE)_\circlearrowleft \\
(\forestLE)_\circlearrowleft \forestME & (\forestNE)_\circlearrowleft \forestOE+(\forestPE)_\circlearrowleft \forestQE & (\forestRE)_\circlearrowleft (\forestSE)_\circlearrowleft+(\forestTE)_\circlearrowleft \\\hline
\forestUE & \forestVE+\forestWE+\forestXE & (\forestYE,\forestAF,\forestBF)_\circlearrowleft+(\forestCF,\forestDF)_\circlearrowleft+(\forestEF)_\circlearrowleft \\
\forestFF & \forestGF+\forestHF+\forestIF & (\forestJF,\forestKF)_\circlearrowleft+(\forestLF,\forestMF)_\circlearrowleft+(\forestNF)_\circlearrowleft \\
(\forestOF)_\circlearrowleft \forestPF & (\forestQF)_\circlearrowleft \forestRF+(\forestSF)_\circlearrowleft \forestTF+(\forestUF)_\circlearrowleft \forestVF & (\forestWF)_\circlearrowleft (\forestXF,\forestYF)_\circlearrowleft+(\forestAG)_\circlearrowleft (\forestBG)_\circlearrowleft+(\forestCG)_\circlearrowleft \\
(\forestDG)_\circlearrowleft \forestEG & (\forestFG)_\circlearrowleft \forestGG+(\forestHG)_\circlearrowleft \forestIG+(\forestJG)_\circlearrowleft \forestKG & (\forestLG)_\circlearrowleft(\forestMG)_\circlearrowleft+(\forestNG)_\circlearrowleft+(\forestOG)_\circlearrowleft \\
(\forestPG)_\circlearrowleft \forestQG & (\forestRG)_\circlearrowleft \forestSG+(\forestTG)_\circlearrowleft \forestUG+(\forestVG)_\circlearrowleft \forestWG & (\forestXG)_\circlearrowleft(\forestYG)_\circlearrowleft+(\forestAH)_\circlearrowleft+(\forestBH)_\circlearrowleft \\
(\forestCH,\forestDH)_\circlearrowleft \forestEH & (\forestFH,\forestGH)_\circlearrowleft \forestHH+2(\forestIH,\forestJH)_\circlearrowleft \forestKH & (\forestLH,\forestMH)_\circlearrowleft (\forestNH)_\circlearrowleft+2(\forestOH,\forestPH)_\circlearrowleft \\
(\forestQH)_\circlearrowleft (\forestRH)_\circlearrowleft \forestSH & (\forestTH)_\circlearrowleft (\forestUH)_\circlearrowleft \forestVH+2(\forestWH)_\circlearrowleft (\forestXH)_\circlearrowleft \forestYH & (\forestAI)_\circlearrowleft (\forestBI)_\circlearrowleft (\forestCI)_\circlearrowleft+2(\forestDI)_\circlearrowleft (\forestEI)_\circlearrowleft \\
{}[\forestFI,\forestGI] & [\forestHI,\forestII]+[\forestJI,\forestKI]+[\forestLI,\forestMI] & (\forestNI,\forestOI)_\circlearrowleft-(\forestPI,\forestQI)_\circlearrowleft+(\forestRI)_\circlearrowleft\\&&-(\forestSI)_\circlearrowleft+(\forestTI)_\circlearrowleft-(\forestUI)_\circlearrowleft \\
\hline
\caption{Divergence of elements of $\Lie_{\AA}(\AA\TT)$ of order up to three (see also Appendix \ref{app:AT_examples}).}
\label{table:ex_planar_aroma_tree_comput}
\end{longtable}
\end{figure}

Planar aromatic trees yield an important example of tracial post-Lie-Rinehart algebra.
\begin{proposition}
The tuple $(\AA_C, \Lie_{\AA_C}(\AA\TT_C),[-,-], \rho,\graft,\tau)$ is a \(k\)-tracial post-Lie-Rinehart algebra.
\end{proposition}

The main result of this section is that the Lie algebra of planar aromatic trees is a free object, which generalises \cite{Floystad20tup} to the post-Lie context.
\begin{theorem}
\label{thm:free_tpLR}
The free \(k\)-tracial post-Lie-Rinehart algebra over the set $C$ is the space of planar aromatic trees \((\AA_C, \Lie_{\AA_C}(\AA\TT_C),[-,-], \rho,\graft,\tau)\), that is, for every set map \(F : C \longrightarrow L\) where \((R, L,[-,-], \rho,\rhd,\tau)\) is a \(k\)-tracial post-Lie-Rinehart algebra, there exists a unique (epi-)morphism of \(k\)-tracial post-Lie-Rinehart algebra \(\F: (\AA_C, \Lie_{\AA_C}(\AA\TT_C)) \longrightarrow (R,L)\) extending the map \(F\).
\[
\begin{tikzcd}
	C & {(R, L,[-,-], \rho,\rhd,\tau)} \\
	{(\AA_C, \Lie_{\AA_C}(\AA\TT_C),[-,-], \rho,\graft,\tau)}
	\arrow["F", from=1-1, to=1-2]
	\arrow["\iota"', hook, from=1-1, to=2-1]
	\arrow["{\F}"', dashed, two heads, from=2-1, to=1-2]
\end{tikzcd}
\]
\end{theorem}

Let us now present a concise proof of Theorem \ref{thm:free_tpLR}. We refer to \cite{Rahm26tup} for a fully detailed proof of this result.
\begin{proof}
As $\Lie(\TT_C)$ is the free post-Lie algebra, there exists a unique post-Lie algebra morphism $\F\colon \Lie(\TT_C) \rightarrow L$ (see Proposition \ref{prop:free_post_Lie}).
Following the proof of \cite{MuntheKaas13opl}, it also naturally yields a unique algebra morphism $\psi\colon \El_k(\Lie(\TT_C))\rightarrow \El_R(L)$, which naturally induces the following map (analogously to the pre-Lie-Rinehart case \cite{Floystad20tup}):
\[
\hat \psi\colon \El_k(\Lie(\TT_C))/[\TT_C^0,\TT_C^0]_\circ \rightarrow \El_R(L)/[\El_R(L),\El_R(L)]_\circ,
\]
where $[u,v]_\circ=u\circ v-v\circ u$.
We thus obtain
\[
\F\colon \TT_C^0/[\TT_C^0,\TT_C^0]_\circ \rightarrow R, \quad \F=\hat \tau\circ \hat \psi,
\]
where we factor the trace $\hat \tau\colon \El_R(L)/[\El_R(L),\El_R(L)]_\circ\rightarrow R$.
As $\TT_C^0/[\TT_C^0,\TT_C^0]_\circ$ is exactly the space of planar aromas, this defines $\F$ on aromas, and one thus extends $\F$ to the symmetric algebra of multiaromas $\AA_C$ as an algebra morphism $\F\colon \AA_C \rightarrow R$.
Then, we obtain the unique \(k\)-tracial post-Lie-Rinehart algebra $\F$ by
\[
\F\colon (\AA_C, \Lie_{\AA_C}(\AA\TT_C))\rightarrow (R,L), \quad \F(at)=\F(a)\F(t).
\]
Hence the result.
\end{proof}

\section{Numerical preservation of divergence-free properties on manifolds}
\label{sec:num}

\subsection{Post-Lie-Rinehart structure of the connection algebra}

Consider a manifold $\MM$ on $k$ equipped with a frame basis $(E_d)$, defined globally for simplicity\footnote{The operations used in Lie-group methods only require a local frame family (see \cite{Owren06ocf, Bronasco25hoi}).}.
The manifold is naturally equipped with the Weitzenböck connection, that we write as a product on the vector fields $\XM$:
\[
Y \tri X= Y[x^i] E_i,\quad X=x^i E_i,
\]
where we use the Einstein summation notation.
The connection $\tri$ is naturally curvature-free $R=0$, but does not have constant torsion in general.
Thus, we further assume that the $(E_d)$ span a Lie algebra, or equivalently that $\MM$ is locally a Lie group, or equivalently that we have constant torsion $\nabla T=0$ (see \cite{Nomizu54iac}).
Let the Jacobi bracket $\llbracket -,-\rrbracket_J$ and the bracket derived from the torsion $[-,-]=-T$, given explicitly for the Weitzenböck connection by
\[
[X,Y]=x^i y^j \llbracket E_i,E_j\rrbracket_J,\quad X=x^i E_i,\quad Y=y^j E_j.
\]
Note that the two brackets are linked by the identity
\[
\llbracket X,Y \rrbracket_J=[X,Y]+X\tri Y- Y\tri X.
\]
Then, the connection algebra $(\XM,[-,-],\tri)$ is a post-Lie algebra \cite{MuntheKaas13opl}.

A post-Lie-Rinehart structure on the $\FM$-module $\XM$ is conveniently defined when using the Weitzenböck connection.
Let the anchor map
\[
\rho(X)(\phi)=X[\phi],\quad X=x^i E_i,\quad \phi\in \FM,
\]
the trace
\[
\tau(u)=u^i_i,\quad u\in\End_{\FM}(\XM), \quad u(E_j)=u^i_j E_i,
\]
and the divergence
\[
\Div(X)=E_i[x^i], \quad X=x^i E_i.
\]
\begin{proposition}
Assume that $\tri$ has constant torsion, then $(\FM,\XM,[-,-],\rho,\tri,\tau)$ is a $k$-tracial post-Lie-Rinehart algebra.
\end{proposition}

\begin{remark}
\label{rk:frozen_prod}
The Taylor expansions of numerical methods are not expressed directly in $\XM$ but rather in the enveloping algebra of $\XM$, equipped with the frozen product $\cdot$.
Four our choice of connection, the frozen product $\cdot$ acts on vector fields and functions as the differential operator
\[
(X_1\cdots X_p)\tri Y=((X_1\cdots X_p)\tri y^j) E_j,\quad
(X_1\cdots X_p)\tri \phi=x_1^{i_1}\dots x_p^{i_p} E_{i_1}[\dots E_{i_p}[\phi]\dots].
\]
It is showed in particular in \cite{Ebrahimi15otl, Busnot25pha} that the universal enveloping algebra of post-Lie algebras (respectively post-Lie-Rinehart algebras) are, under technical assumptions, post-Hopf algebras (respectively post-Hopf algebroids).
\end{remark}

\subsection{Divergence-free vector fields, volume-preservation, and planar aromatic trees}

Let a smooth ordinary differential equation on $\MM$ led by a divergence-free vector field
\begin{equation}
\label{eq:ODE}
y'(t)=F(y(t)),\quad y(0)=y_0,\quad F=x^i E_i,\quad \Div(F)=0,
\end{equation}
where $F$ is a smooth Lipschitz vector field.
The divergence characterises geometric properties in specific cases.
\begin{proposition}
\label{prop:vol_pres_div_free}
Assume that $\MM$ is compact and equipped with a bi-invatiant Riemannian metric and the corresponding volume form $d\vol$.
Let the volume of a measurable set $\Omega$ of $\MM$ be $\Vol(\Omega)=\int_\Omega d\vol$.
We say that a flow $\varphi_t$ is volume-preserving if for all measurable $\Omega$ and all $t>0$, we have $\Vol(\varphi_t(\Omega))=\Vol(\Omega)$.
Then, the flow of \eqref{eq:ODE} is volume-preserving if and only if $\Div(F)=0$.
\end{proposition}

\begin{proof}
The preservation of volume is characterised by the divergence for the Levi-Civita connection \cite[Liouville's theorem]{Lee18itr}, which coincides with the Weitzenböck divergence $\Div$ in the case of a bi-invariant metric.
\end{proof}

\begin{ex}
Consider the differential equation on the special orthogonal group $\SO_d(\R)$:
\[y'(t)=A(y(t))y(t),\quad A\colon \SO_d(\R)\rightarrow \mathfrak{so}_d(\R).\]
Given a frame basis of right-invariant vector fields $E_i(y)=A_i y$ and $(A_i)$ basis of $\mathfrak{so}_d(\R)$, the Weitzenböck connection becomes the standard right-invariant connection on $\SO_d(\R)$.
If $A(y)= f^i(y) A_i$ and $E_i[f^i]=0$, then the flow preserves volume.
\end{ex}

\begin{remark}
Finding integrators that preserve volume in a general geometric context is a challenging open problem, for which this work is a first step. Indeed, this requires the design of a high-order analysis and the study of algebraic structures for representing the connection algebra for a general connection. We cite in particular the recent works \cite{Stava23tca, Stava24oca, MuntheKaas24gio} that introduce integrators and algebraic tools for connections with vanishing torsion and constant curvature.
The general case will be studied in future works.
\end{remark}

Planar aromatic trees represent specific vector fields and functions through the application of the elementary differential map.
\begin{definition}
Given a smooth vector field $F\in \XM$, the elementary differential map is defined on $\Lie_k(\TT)$ by
\[
\F^F(\forestVI)=F,\quad
\F^F((t_1\cdots t_p)\graft \forestWI)=(\F^F(t_1)\cdots\F^F(t_p))\tri F,\quad
\F^F([t_1,t_2])=[\F^F(t_1),\F^F(t_2)],
\]
where we use the frozen product of Remark \ref{rk:frozen_prod}.
The map extends to $\TT^0$ by
\[
\F^F(\forestXI)(X)=X,\quad
\F^F((t_1\cdots t_p)\graft \forestYI)(X)=(\F^F(t_1)(X)\cdots\F^F(t_p)(X))\tri F,
\]
and on aromas by
\[
\F^F(\textbf{1})=1,\quad
\F^F((t_1,\dots,t_n)_\circlearrowleft)=\F^F(t_1)^{i_1}(E_{i_2}) \cdots \F^F(t_n)^{i_n}(E_{i_1}).
\]
Finally, $\F^F$ extends into $\F^F\colon \Lie_{\AA}(\AA\TT)\rightarrow \XM$ as a morphism
\[
\F^F(a_1\dots a_n t)=\F^F(a_1)\dots \F^F(a_n) \F^F(t).
\]
\end{definition}

\begin{proposition}
The elementary differential $\F^F$ is a morphism of post-Lie-Rinehart algebras:
\[
\F^F\colon (\AA, \Lie_{\AA}(\AA\TT),[-,-], \rho,\graft,\tau) \rightarrow (\FM,\XM,[-,-],\rho,\tri,\tau).
\]
In particular, $\F^F$ commutes with the divergence:
\[
\Div\circ \F^F=\F^F\circ \Div.
\]
\end{proposition}

The divergence-free assumption of $F$ yields degeneracies.
\begin{lemma}
\label{lemma:degeneracies}
Assume that $\Div(F)=0$. Then, the differential associated to the following aromas vanishes:
\[
\F^F((t_1\cdots t_p \cdot \times)\graft \forestAJ)_\circlearrowleft)=0.
\]
\end{lemma}

The degeneracies of Lemma \ref{lemma:degeneracies} prove useful in the design of efficient pseudo-divergence-free methods as it allows to use only first derivatives to represent some differential operators of higher order in general. 
In particular, one obtains with $\Div(F)=0$,
\[
\F^F((\times\cdot t)\graft \forestBJ)_\circlearrowleft)=\llbracket E_i,\F^F(t)\rrbracket_J [f^i],
\]
where we recall that $\llbracket E_i,\F^F(t)\rrbracket_J$ is a vector field and thus acts as a first order differential operator.

\subsection{Divergence-free Lie-group methods with aromatic Lie-Butcher series}

Our aim is to construct Lie-group integrators \cite{Iserles00lgm} that satisfy similar properties as the flow of \eqref{eq:ODE}. Exact volume-preservation is an open problem already in the Euclidean setting, so that we propose a methodology for the design of pseudo-divergence-free methods.

Let the exponential $\exp(tX)p$ for $X\in \XM$ be the solution of the ODE
\[
y'(t)=X(y(t)),\quad y(0)=p.
\]
Freezing the components at a point $q\in \MM$ yields the frozen exponential $\exp(tx^d(q)E_d) p$, which gives the geodesics for the connection $\tri$.
We consider Lie-Runge-Kutta methods of the form
\begin{align}
Y_i&=\exp(h\sum_j a_{ij}f^d(Y_j)E_d)y_n,\nonumber\\
\label{eq:LRK_methods}
y_{n+1}&=\exp(h\sum_i b_{i}f^d(Y_i)E_d)y_n,
\end{align}
where $A\in \R^{s\times s}$ and $b\in \R^s$ are the coefficients of the methods, $h$ is the timestep of the method, and $s$ is the number of stages.
Our approach extends straightforwardly to RKMK methods \cite{MuntheKaas95lbt, MuntheKaas97nio, MuntheKaas98rkm, MuntheKaas99hor} and frozen-flow methods \cite{Crouch93nio, Owren99rkm, Celledoni03cfl, Owren06ocf}, and we restrict our approach to the methods \eqref{eq:LRK_methods} for simplicity.
As exact volume preservation is impossible for methods of the form \eqref{eq:LRK_methods} (see \cite{Chartier07pfi, Iserles07bsm}), we design preprocessors to have a high-order of volume-preservation with a low order of convergence, similarly to the method proposed in \cite{Bogfjellmo19aso}. We recall in particular that the methods \eqref{eq:LRK_methods} cannot be of order more than two in general.

Lie-group methods are naturally described by planar Butcher trees and forests.
\begin{definition}
A Lie-Butcher (LB) series is a formal series indexed by planar trees
\[
B^F(a)=\sum_{t\in T} a(t)\F^F(t), \quad a\in \TT^*.
\]
Analogously, an aromatic LB series is a formal series indexed by planar aromatic trees
\[
B^F(a)=\sum_{t\in AT} a(t)\F^F(t), \quad a\in \AA\TT^*.
\]
\end{definition}

Thanks to Proposition \ref{prop:vol_pres_div_free}, the backward error analysis of Lie-group methods allows us to characterise methods that preserve the features of the system \eqref{eq:ODE} (see also \cite{Hairer06gni, Chartier10aso, Calaque11tih}).
\begin{proposition}[\cite{Lundervold13bea, Rahm22aoa}]
Consider a method of the form \eqref{eq:LRK_methods}.
Then, its Taylor expansion coincides with the one of the exact flow of the modified ODE
\[
y'(t)=\tilde{F}_h(y(t)),\quad h\tilde{F}_h=B^{hF}(b),
\]
where the modified vector field $\tilde{F}_h$ is given by a LB series.
More generally, if the method \eqref{eq:LRK_methods} is applied to a preprocessed vector field $h\hat{F}_h=B^{hF}(b)$ given by an aromatic LB series, then the modified vector field of the resulting method $\tilde{F}_h$ is given by a LB series.
The numerical method is called divergence-free if its modified vector field satisfies $\Div(\tilde{F}_h)=0$ and is called pseudo-divergence-free of order $p$ if $\Div(\tilde{F}_h)=\OO(h^p)$.
\end{proposition}

\begin{ex}
Consider the Lie-Euler method
\begin{equation}
\label{eq:Lie_Euler}
y_{n+1}=\exp(hf^d(y_n)E_d)y_n.
\end{equation}
Then, the first terms of its modified vector field are
\begin{align*}
h\tilde{F}_h&=\F^{hF}\Big(
\forestCJ
-\frac{1}{2}\forestDJ
+\frac{1}{3}\forestEJ
+\frac{1}{12}\forestFJ
-\frac{1}{12}[\forestGJ,\forestHJ]
+\dots \Big).
\end{align*}
\end{ex}

The combinations of planar aromatic forests of order up to four of vanishing divergence $\Div(\Psi_k)=0$ are listed below. Note that we recover the solenoidal forms from \cite{Laurent23tab, Laurent23tld} when removing the commutators and using non-planar aromatic trees.
\begin{align*}
\Psi_1&=\forestIJ
+(\forestJJ)_\circlearrowleft \forestKJ
-(\forestLJ)_\circlearrowleft \forestMJ
-(\forestNJ,\forestOJ)_\circlearrowleft \forestPJ
+[\forestQJ,\forestRJ],\\
\Psi_2&=\forestSJ
+\forestTJ
+(\forestUJ)_\circlearrowleft \forestVJ
-(\forestWJ)_\circlearrowleft \forestXJ
-(\forestYJ,\forestAK)_\circlearrowleft \forestBK
-(\forestCK,\forestDK,\forestEK)_\circlearrowleft \forestFK
+[\forestGK,\forestHK],\\
\Psi_3&=\forestIK
+\forestJK
+\forestKK
+(\forestLK)_\circlearrowleft \forestMK
-\forestNK
-(\forestOK)_\circlearrowleft \forestPK
-(\forestQK,\forestRK)_\circlearrowleft \forestSK
-(\forestTK,\forestUK)_\circlearrowleft \forestVK
+[\forestWK,\forestXK],\\
\Psi_4&=(\forestYK)_\circlearrowleft \forestAL
+(\forestBL)_\circlearrowleft \forestCL
+(\forestDL)_\circlearrowleft (\forestEL)_\circlearrowleft \forestFL
-(\forestGL)_\circlearrowleft \forestHL
-(\forestIL)_\circlearrowleft (\forestJL)_\circlearrowleft \forestKL
-(\forestLL)_\circlearrowleft (\forestML,\forestNL)_\circlearrowleft \forestOL
+(\forestPL)_\circlearrowleft [\forestQL,\forestRL],\\
\Psi_5&=\forestSL
+(\forestTL)_\circlearrowleft \forestUL
+(\forestVL)_\circlearrowleft \forestWL
+(\forestXL,\forestYL)_\circlearrowleft \forestAM
-\forestBM
-(\forestCM)_\circlearrowleft \forestDM
-(\forestEM)_\circlearrowleft \forestFM
-(\forestGM,\forestHM)_\circlearrowleft \forestIM\\&
+[\forestJM,\forestKM]
+[\forestLM,\forestMM]
+(\forestNM)_\circlearrowleft [\forestOM,\forestPM]
+ [\forestQM,[\forestRM,\forestSM]].
\end{align*}

Our analysis with planar aromatic trees allows us to tackle the tedious calculations for deriving pseudo-divergence-free methods of high order.
In particular, we propose the following Lie-Runge-Kutta methods, where we use preprocessed aromatic vector fields that only require the first derivative of the $f^d$.
\begin{proposition}
Consider the Lie-Euler method \eqref{eq:Lie_Euler} for solving \eqref{eq:ODE} with $\Div(F)=0$.
Apply the method with the preprocessed vector field
\begin{align*}
h\hat{F}_h&=\F^{hF}\Big(\forestTM
+\frac{1}{2}\forestUM
-\frac{1}{3}\forestVM
-\frac{1}{12} (\forestWM)_\circlearrowleft \forestXM
-\frac{1}{12} (\forestYM,\forestAN)_\circlearrowleft \forestBN
+\frac{1}{6}[\forestCN,\forestDN]\Big)\\
&=hf^i E_i
+\frac{h^2}{2}f^j E_j[f^i] E_i
-\frac{h^3}{3} f^k E_k[f^j] E_j[f^i] E_i
-\frac{h^3}{12} f^k \llbracket E_j,E_k\rrbracket_J[f^j] f^i E_i\\&
-\frac{h^3}{12} E_j[f^k] E_k[f^j] f^i E_i
+\frac{h^3}{6} f^k f^j E_j[f^i] \llbracket E_k,E_i\rrbracket_J.
\end{align*}
Then, the resulting method has order two of convergence and is pseudo-divergence-free of third order.
\end{proposition}

\begin{proposition}
Consider the Lie-Runge-Kutta method \eqref{eq:LRK_methods} with the following Butcher tableau for solving \eqref{eq:ODE} with $\Div(F)=0$.
\[
\begin{array}
{c|ccc}
0 & 0 & 0 & 0\\
\frac{-1+\sqrt{5}}{12} & \frac{-1+\sqrt{5}}{12} & 0 & 0\\
\frac{3+\sqrt{5}}{12} & \frac{-9-5\sqrt{5}}{12} & \frac{2+\sqrt{5}}{2} & 0\\\hline
& \frac{-7+3\sqrt{5}}{2} & 0 & \frac{9-3\sqrt{5}}{2}
\end{array}
\]
Apply the integrator to the preprocessed vector field
\begin{align*}
h\hat{F}_h
&=\F^{hF}\Big(\forestEN
-\frac{1}{12}\forestFN
+\frac{1-\sqrt{5}}{24} (\forestGN)_\circlearrowleft \forestHN
+\frac{1-\sqrt{5}}{24} (\forestIN,\forestJN)_\circlearrowleft \forestKN
+\frac{1+\sqrt{5}}{24}[\forestLN,\forestMN]\\&
+\frac{1}{8}\forestNN
-\frac{1}{12}\big(\forestON-\forestPN\big)
+\frac{1}{36}\big((\forestQN)_\circlearrowleft \forestRN -2(\forestSN)_\circlearrowleft \forestTN\big)
+\frac{1}{12}\big((\forestUN,\forestVN)_\circlearrowleft \forestWN -(\forestXN,\forestYN)_\circlearrowleft \forestAO\big)\\&
+\frac{1}{18}(\forestBO,\forestCO,\forestDO)_\circlearrowleft \forestEO
-\frac{1}{12} [\forestFO,\forestGO]
+\frac{1}{18} [\forestHO,[\forestIO,\forestJO]]
\Big)\\
&
=hF
+h^3\Big(
-\frac{1}{12} f^k E_k[f^j] E_j[f^i] E_i
+\frac{1-\sqrt{5}}{24} f^k \llbracket E_j,E_k\rrbracket_J[f^j] f^i E_i\\&
+\frac{1-\sqrt{5}}{24} E_j[f^k] E_k[f^j] f^i E_i
+\frac{1+\sqrt{5}}{24} f^k f^j E_j[f^i] \llbracket E_k,E_i\rrbracket_J
\Big)\\&
+h^4\Big(
\frac{1}{8} f^l E_l[f^k] E_k[f^j] E_j[f^i] E_i
-\frac{1}{12} f^l f^k E_k[f^j] \llbracket E_l,E_j\rrbracket_J[f^i] E_i\\&
+\frac{1}{36} f^l f^k \llbracket \llbracket E_j,E_l\rrbracket_J,E_k\rrbracket_J[f^j] f^i E_i
+\frac{1}{12} f^l E_j[f^k] \llbracket E_k,E_l\rrbracket_J[f^j] f^i E_i\\&
+\frac{1}{18} E_j[f^l] E_l[f^k] E_k[f^j] f^i E_i
-\frac{1}{12} f^l f^k E_k[f^j] E_j[f^i] \llbracket E_l,E_i\rrbracket_J\\&
+\frac{1}{18} f^l f^k f^j E_j[f^i] E_i \llbracket E_l,\llbracket E_k,E_i\rrbracket_J\rrbracket_J
\Big).
\end{align*}
Then, the resulting aromatic method has order two of convergence, but is pseudo-divergence-free of fourth order.
\end{proposition}

\section{Conclusion}
\label{sec:Conclusion}

In this paper, we defined planar aromatic trees and showed that they are the free tracial post-Lie-Rinehart algebra. This algebraic object finds concrete applications in the numerical integration of ODEs on manifolds for the design of pseudo-divergence-free maps via backward error analysis.

This work opens several research avenues that we will explore in future works.
First, the characterisation and existence of divergence-free integrators is an important open problem of geometric integration, even for Euclidean ODEs. This calls for the generalisation of the works \cite{Bogfjellmo22uat, Laurent23tld, Laurent23tab, Dotsenko24vpo} to the manifold case.
The understanding of divergence-free maps for ODEs shows strong links with the discretisations of ergodic stochastic differential equations that sample the invariant measure exactly, with numerous applications in molecular dynamics, stochastic optimisation, and machine learning. This link is uncovered in \cite{Laurent20eab, Bronasco22cef} in $\R^D$ and in \cite{Bronasco25hoi} through the use of exotic Butcher series and extensions, so that generalising the present work to the stochastic context is natural.
On the geometric side, it would be interesting to characterise planar aromatic B-series with universal geometric properties in the spirit of \cite{MuntheKaas16abs, McLachlan16bsm, Laurent23tue}.
On the algebraic side, the structures associated to aromas were recently studied from various point of views in \cite{Laurent23tab, Busnot25pha, Zhu25aac} with possible applications beyond numerical analysis and combinatorial algebra.
Moreover, we recall from \cite{Lundervold11hao, Grong23pla} that post-Lie algebras are associated to the connection algebra in a context of constant torsion and vanishing curvature, which is not the natural framework for working with volume forms. It is important to generalise the concept of aromas to more general connection algebras in order to obtain general statements for volume-preservation on Riemannian manifolds.

\bigskip

\noindent \textbf{Acknowledgements.}\
The authors acknowledge the support of the French program ANR-25-CE40-2862-01 (MaStoC - Manifolds and Stochastic Computations), the Research Council of Norway through project 302831 ``Computational Dynamics and Stochastics on Manifolds'' (CODYSMA), and the UiT--MaSCoT project at UiT, Troms\o.
The first author would like to thank Pauline Baudat for an enlightening discussion that inspired our notation of planar aromas.
The present paper and the work \cite{Rahm26tup} were written simultaneously, unbeknownst to the authors, and both study the free tracial post-Lie-Rinehart algebra. The work \cite{Rahm26tup} focuses on the algebraic structure, while the present paper focuses on the geometric context and the numerical applications.

\bibliographystyle{abbrv}
\bibliography{ma_bibliographie}

\newpage

\begin{appendices}
\section{Additional background on Lie--Rinehart algebras}
This appendix collects background material on Lie--Rinehart algebras
and related geometric constructions that is not required for the main
development, but may be useful for context.

\subsection{Chevalley-Eilenberg Cohomology}
Given a Lie-Rinehart algebra $(R, L)$, let the Chevalley-Eilenberg cochain complex denoted by $\mathfrak{CE}(L, \End_{R}(L), \omega) = ( \bigwedge^{\bullet} \,L^{\ast}, \End_{R}(L))$ where    
\begin{enumerate}[label = (\roman*)]
\item $\End_{R}(L)$ is the $k$-Lie algebra of $R$-linear endomorphisms of the $R$-module $L$. The Lie bracket of $\End_R(L)$ is denoted by $\llbracket \cdot, \cdot \rrbracket$,

\item $\omega \in L^{\ast} \otimes \End_{R}(L)$ and is called a $\End_R(L)$-Lie algebra valued {\em connection} 1-form satisfying 
\begin{align}\label{eq:conn-form}
(\omega(r.g_1) - r.\omega(g_1))(g_2) = (\rho(g_2)(r))g_1,
\end{align}

\item $L^{\ast}$ is the dual of the $R$-module $L$,
\item $\bigwedge^{\bullet}L^{\ast}$ is the differential graded exterior algebra of the $R$-module $L^{\ast}$.
\end{enumerate}
Given $\omega$, then $\End_R(L)$ is a $L$-module with $L$-action: $g.f = \llbracket \omega(\cdot)(g),f(\cdot)\rrbracket$ viz.,
\begin{align*}
(g.f)(g') &= \omega(f(g'))(g) - f((\omega(g')))(g) \\
&= \llbracket \omega(\cdot)(g), f(\cdot)\rrbracket (g'), 
\end{align*}
for $g,g' \in L$ and $f \in \End_R(L)$. It is straightforward to check via Jacobi identity that $g.(h.f) - h.(g.f) = \llbracket g,h\rrbracket. f$  and generalizes the adjoint action {\em twisted} by connection form.

The Chevalley-Eilenberg cochain complex is 
\begin{align*}
\End_{R}(L) \xrightarrow{d} L^{\ast} \otimes \End_R(L) \xrightarrow{d} \bigwedge ^2  L^{\ast} \otimes \End_R(L) \xrightarrow{d} \cdots   
\end{align*}
where the differential $d$ is defined as: for all $\gamma \in \bigwedge^{n}L^{\ast} \otimes \End_R(L)$ and $x_1, x_2,\ldots, x_{n} \in L$:
\begin{align*}
d\gamma (x_1,x_2,\ldots,x_{n+1}) 
&= \sum_{\sigma \in \mathrm{Sh}(1,n)} \mathrm{sgn}(\sigma) \, x_{\sigma(1)} \cdot \gamma(x_{\sigma(2)}, x_{\sigma(3)}, \ldots, x_{\sigma(n+1)}) \\
&- \sum_{\sigma \in \mathrm{Sh}(2,n-1)} \mathrm{sgn}(\sigma) \, \gamma([x_{\sigma(1)}, x_{\sigma(2)}], x_{\sigma(3)}, \ldots, x_{\sigma(n+1)}),
\end{align*} 
where $\text{Sh}(p,q)$ denote $(p,q)$-shuffles in the permutation group of $p+q$ letters. The {\em curvature} is a $2$-form $\mathrm{R} \in \bigwedge^2 L^{\ast} \otimes \End_{R}(L)$ given by 
\begin{align*}
\mathrm{R} = d\omega + \dfrac{1}{2}\left[ \omega \wedge \omega \right].
\end{align*}  

\subsection{Connection forms, reductive splittings, and torsion}
Say $\End_{R}(L)$ has a reductive decomposition viz., $\End_{R}(L) = \mathfrak{h} \oplus \mathfrak{m}$ as $R$-modules such that  
\begin{enumerate}[label = (\roman*)]
\item $\mathfrak{h}$ is a Lie subalgebra of $L$,
\item $\mathfrak{m}$ is a $\mathfrak{h}$-module under adjoint action of $\mathfrak{h}$ i.e, $\llbracket  \mathfrak{h}, \mathfrak{m} \rrbracket \subseteq \mathfrak{m}$.
\end{enumerate}
The $1$-form whence splits as $\omega = \omega_{\mathfrak{h}} + \omega_{\mathfrak{m}}$ where  $\omega_{\mathfrak{h}} \in L^{\ast} \otimes \mathfrak{h}$ and $\omega_{\mathfrak{m}} \in L^{\ast} \otimes \mathfrak{m}$ respectively. The $\omega_{\mathfrak{h}}$ is called {\em principal connection form} and $\omega_{\mathfrak{m}}$ is called {\em soldering form} or {\em vielbein}.
Consequently, the Lie algebra valued $2$-form curvature $\mathrm R$ splits as 
\begin{align*}
\mathrm{R} &= d\omega + \dfrac{1}{2} \left[ \omega \wedge \omega \right] \\
&= d(\omega_{\mathfrak{h}} + \omega_{\mathfrak{m}}) + \dfrac{1}{2} \left[ (\omega_{\mathfrak{h}} + \omega_{\mathfrak{m}})\wedge(\omega_{\mathfrak{h}} + \omega_{\mathfrak{m}})\right] \\
&= d\omega_{\mathfrak{h}} + d\omega_{\mathfrak{m}} + \dfrac{1}{2} \{ \left[\omega_{\mathfrak{h}} \wedge \omega_{\mathfrak{h}}\right] + 2 \left[ \omega_{\mathfrak{h}} \wedge \omega_{\mathfrak{m}}\right] + \left[\omega_{\mathfrak{m}}\wedge\omega_{\mathfrak{m}}\right] \}.  
\end{align*}
Since $\mathfrak{m}$ is not necessarily a Lie subalgebra, the term $\left[\omega_{\mathfrak{m}}, \omega_{\mathfrak{m}}\right]$ can have values in both $\mathfrak{h}$ and $\mathfrak{m}$. We obtain
\begin{align*}
\left[\omega_{\mathfrak{m}}\wedge\omega_{\mathfrak{m}}\right] &= \left[\omega_{\mathfrak{m}}\wedge \omega_{\mathfrak{m}}\right]_{\mathfrak{h}} + \left[\omega_{\mathfrak{m}}\wedge\omega_{\mathfrak{m}}\right]_{\mathfrak{m}}.
\end{align*}
The curvature form $\mathrm{R}$ thus splits into terms in $\mathfrak{h}$ and $\mathfrak{m}$ as
\begin{align*}
\mathrm{R} &= \underbrace{d\omega_{\mathfrak{h}} + \dfrac{1}{2}\left[\omega_{\mathfrak{h}}\wedge \omega_{\mathfrak{h}}\right] + \left[\omega_{\mathfrak{m}}\wedge\omega_{\mathfrak{m}}\right]_{\mathfrak{h}} }_{\text{terms in}\,\mathfrak{h}} + \underbrace{d\omega_{\mathfrak{m}} + \left[\omega_{\mathfrak{h}}\wedge\omega_{\mathfrak{m}}\right] + \dfrac{1}{2}\left[\omega_{\mathfrak{m}} \wedge \omega_{\mathfrak{m}}\right]_{\mathfrak{m}}}_{\text{terms in}\, \mathfrak{m}} \\
&= \mathcal{R} + \mathcal{T}.
\end{align*}
The terms $\mathcal{R}$ is the {\em intrinsic curvature} $2$-form and $\mathcal{T}$ is the {\em torsion}. In the case where the Lie ideal \(\mathfrak m\) is an Abelian Lie algebra, then  the expressions for intrinsic curvature and torsion reduce respectively to 
\begin{align*}
\mathcal{R} &= d\omega_{\mathfrak{h}} + \dfrac{1}{2}\left[\omega_{\mathfrak{h}}\wedge \omega_{\mathfrak{h}}\right], \\
\mathcal{T} &= d\omega_{\mathfrak{m}} + \left[\omega_{\mathfrak{h}}\wedge\omega_{\mathfrak{m}}\right].
\end{align*}

\subsection{Levi-Civita and Weitzenb\"ock Connection}
The connection form \(\omega \in L^{\ast} \otimes \End_R(L) \) is an affine connection form when $\End_R(L) \cong \mathfrak{aff}(n)$ where \(\mathfrak{aff}(n)\) is the Lie algebra of affine group. Note that, \( \mathfrak {aff}(n)\) has a reductive decomposition:
\[
\mathfrak {aff}(n) = \mathfrak{gl}(n) \oplus \mathbb R^n.
\]

The connection form \(\omega \in L^{\ast} \otimes \End_R(L) \) with \(\End_R(L) \cong \mathfrak {iso}(n)\), the Lie algebra of isometries of the affine space \(\mathbb R^n\). Recall that 
\(\mathfrak {iso}(n) \hookrightarrow \mathfrak {aff}(n)\) and has reductive decomposition as 
\[
\mathfrak{iso} (n) = \mathfrak{so}(n) \oplus \mathbb R^n,
\]  
where \(\mathfrak {so}(n)\) is the Lie subalgebra and \(\mathbb R^n\) is \(\mathfrak {so}(n)\) ideal. The both Levi-Civita and Weitzenb\"ock connections are isometric connections which are explained shortly.  The 2-form curvature in isometric connection splits into intrinsic curvature and torsion as 

\[
\mathrm{R} = \underbrace{\mathcal R}_{\text{valued in \,} \mathfrak{so} (n)} + \underbrace{\mathcal T}_{\text{valued in \,} \mathbb R^n}.
\]
The {\em Levi-Civita connection} is the isometric connection where \(\mathcal T = 0\) (torsion vanishes) and {\em Weitzenb\"ock connection} is the isometric connection where \(\mathcal R = 0\) (intrinsic curvature vanishes).
The Bianchi identity for the Levi-Civita connection whence reduces to  
\[
d\mathrm{R} = d \mathcal{R} = 0,
\]    
while for the Weitzenb\"ock connection the Bianchi identity is
\[ 
d\mathrm{R} = d \mathcal T = 0.
\]

\newpage

\section{Computations on planar aromatic trees of order four}
\label{app:AT_examples}

\begin{longtable}{|C|C|}
\hline
t\in \Lie_{\AA}(\AA\TT) & \Div(t) \\\hline
\forestKO & (\forestLO,\forestMO,\forestNO,\forestOO)_\circlearrowleft+(\forestPO,\forestQO,\forestRO)_\circlearrowleft+(\forestSO,\forestTO)_\circlearrowleft+(\forestUO)_\circlearrowleft \\
\forestVO & (\forestWO,\forestXO,\forestYO)_\circlearrowleft+(\forestAP,\forestBP,\forestCP)_\circlearrowleft+(\forestDP,\forestEP)_\circlearrowleft+(\forestFP)_\circlearrowleft \\
\forestGP & (\forestHP,\forestIP,\forestJP)_\circlearrowleft+(\forestKP,\forestLP)_\circlearrowleft+(\forestMP,\forestNP)_\circlearrowleft+(\forestOP)_\circlearrowleft \\
\forestPP & (\forestQP,\forestRP)_\circlearrowleft+(\forestSP,\forestTP,\forestUP)_\circlearrowleft+(\forestVP,\forestWP)_\circlearrowleft+(\forestXP)_\circlearrowleft \\
\forestYP & (\forestAQ,\forestBQ)_\circlearrowleft+(\forestCQ,\forestDQ)_\circlearrowleft+(\forestEQ,\forestFQ)_\circlearrowleft+(\forestGQ)_\circlearrowleft \\
(\forestHQ)_\circlearrowleft \forestIQ & (\forestJQ)_\circlearrowleft(\forestKQ,\forestLQ,\forestMQ)_\circlearrowleft +(\forestNQ)_\circlearrowleft(\forestOQ,\forestPQ)_\circlearrowleft +(\forestQQ)_\circlearrowleft(\forestRQ)_\circlearrowleft +(\forestSQ)_\circlearrowleft \\
(\forestTQ)_\circlearrowleft \forestUQ & (\forestVQ)_\circlearrowleft(\forestWQ,\forestXQ)_\circlearrowleft +(\forestYQ)_\circlearrowleft(\forestAR,\forestBR)_\circlearrowleft +(\forestCR)_\circlearrowleft(\forestDR)_\circlearrowleft +(\forestER)_\circlearrowleft \\
(\forestFR)_\circlearrowleft \forestGR & (\forestHR)_\circlearrowleft (\forestIR,\forestJR)_\circlearrowleft +(\forestKR)_\circlearrowleft (\forestLR)_\circlearrowleft +(\forestMR)_\circlearrowleft +(\forestNR)_\circlearrowleft \\
(\forestOR)_\circlearrowleft \forestPR & (\forestQR)_\circlearrowleft (\forestRR,\forestSR)_\circlearrowleft +(\forestTR)_\circlearrowleft (\forestUR)_\circlearrowleft +(\forestVR)_\circlearrowleft +(\forestWR)_\circlearrowleft \\
(\forestXR,\forestYR)_\circlearrowleft \forestAS & (\forestBS,\forestCS)_\circlearrowleft (\forestDS,\forestES)_\circlearrowleft +(\forestFS,\forestGS)_\circlearrowleft (\forestHS)_\circlearrowleft +2(\forestIS,\forestJS)_\circlearrowleft \\
(\forestKS)_\circlearrowleft (\forestLS)_\circlearrowleft \forestMS & (\forestNS)_\circlearrowleft (\forestOS)_\circlearrowleft (\forestPS,\forestQS)_\circlearrowleft +(\forestRS)_\circlearrowleft (\forestSS)_\circlearrowleft (\forestTS)_\circlearrowleft +2(\forestUS)_\circlearrowleft(\forestVS)_\circlearrowleft \\
(\forestWS)_\circlearrowleft \forestXS & (\forestYS)_\circlearrowleft (\forestAT)_\circlearrowleft +(\forestBT)_\circlearrowleft +(\forestCT)_\circlearrowleft +(\forestDT)_\circlearrowleft \\
(\forestET)_\circlearrowleft \forestFT & (\forestGT)_\circlearrowleft (\forestHT)_\circlearrowleft +(\forestIT)_\circlearrowleft +(\forestJT)_\circlearrowleft +(\forestKT)_\circlearrowleft  \\
(\forestLT)_\circlearrowleft \forestMT & (\forestNT)_\circlearrowleft (\forestOT)_\circlearrowleft +(\forestPT)_\circlearrowleft +(\forestQT)_\circlearrowleft +(\forestRT)_\circlearrowleft \\
(\forestST)_\circlearrowleft \forestTT & (\forestUT)_\circlearrowleft (\forestVT)_\circlearrowleft +(\forestWT)_\circlearrowleft +(\forestXT)_\circlearrowleft +(\forestYT)_\circlearrowleft \\
(\forestAU)_\circlearrowleft \forestBU & (\forestCU)_\circlearrowleft (\forestDU)_\circlearrowleft +(\forestEU)_\circlearrowleft +(\forestFU)_\circlearrowleft +(\forestGU)_\circlearrowleft \\
(\forestHU,\forestIU)_\circlearrowleft \forestJU & (\forestKU,\forestLU)_\circlearrowleft (\forestMU)_\circlearrowleft +(\forestNU,\forestOU)_\circlearrowleft +(\forestPU,\forestQU)_\circlearrowleft +(\forestRU,\forestSU)_\circlearrowleft \\
(\forestTU,\forestUU)_\circlearrowleft \forestVU & (\forestWU,\forestXU)_\circlearrowleft (\forestYU)_\circlearrowleft +(\forestAV,\forestBV)_\circlearrowleft +(\forestCV,\forestDV)_\circlearrowleft +(\forestEV,\forestFV)_\circlearrowleft \\
(\forestGV,\forestHV,\forestIV)_\circlearrowleft \forestJV & (\forestKV,\forestLV,\forestMV)_\circlearrowleft (\forestNV)_\circlearrowleft +3(\forestOV,\forestPV,\forestQV)_\circlearrowleft \\
(\forestRV)_\circlearrowleft (\forestSV)_\circlearrowleft \forestTV & (\forestUV)_\circlearrowleft (\forestVV)_\circlearrowleft (\forestWV)_\circlearrowleft +(\forestXV)_\circlearrowleft (\forestYV)_\circlearrowleft +(\forestAW)_\circlearrowleft (\forestBW)_\circlearrowleft +(\forestCW)_\circlearrowleft (\forestDW)_\circlearrowleft \\
(\forestEW)_\circlearrowleft (\forestFW)_\circlearrowleft \forestGW & (\forestHW)_\circlearrowleft (\forestIW)_\circlearrowleft (\forestJW)_\circlearrowleft +(\forestKW)_\circlearrowleft (\forestLW)_\circlearrowleft +(\forestMW)_\circlearrowleft (\forestNW)_\circlearrowleft +(\forestOW)_\circlearrowleft (\forestPW)_\circlearrowleft \\
(\forestQW,\forestRW)_\circlearrowleft (\forestSW)_\circlearrowleft \forestTW & (\forestUW,\forestVW)_\circlearrowleft (\forestWW)_\circlearrowleft (\forestXW)_\circlearrowleft +(\forestYW,\forestAX)_\circlearrowleft (\forestBX)_\circlearrowleft +2(\forestCX,\forestDX)_\circlearrowleft (\forestEX)_\circlearrowleft \\
(\forestFX)_\circlearrowleft (\forestGX)_\circlearrowleft (\forestHX)_\circlearrowleft \forestIX & (\forestJX)_\circlearrowleft (\forestKX)_\circlearrowleft (\forestLX)_\circlearrowleft (\forestMX)_\circlearrowleft +3(\forestNX)_\circlearrowleft (\forestOX)_\circlearrowleft (\forestPX)_\circlearrowleft \\
\hline
\end{longtable}
\newpage

\begin{longtable}{|C|C|}
\hline
t\in \Lie_{\AA}(\AA\TT) & \Div(t) \\\hline
{}[\forestQX,\forestRX] & (\forestSX,\forestTX,\forestUX)_\circlearrowleft-(\forestVX,\forestWX,\forestXX)_\circlearrowleft+(\forestYX,\forestAY)_\circlearrowleft-(\forestBY,\forestCY)_\circlearrowleft\\&+(\forestDY)_\circlearrowleft-(\forestEY)_\circlearrowleft+(\forestFY)_\circlearrowleft-(\forestGY)_\circlearrowleft \\
{}[\forestHY,\forestIY] & (\forestJY,\forestKY)_\circlearrowleft-(\forestLY,\forestMY)_\circlearrowleft+(\forestNY)_\circlearrowleft\\&-(\forestOY)_\circlearrowleft+(\forestPY)_\circlearrowleft-(\forestQY)_\circlearrowleft \\
{}(\forestRY)_\circlearrowleft[\forestSY,\forestTY] & (\forestUY)_\circlearrowleft (\forestVY,\forestWY)_\circlearrowleft-(\forestXY)_\circlearrowleft (\forestYY,\forestAAB)_\circlearrowleft+(\forestBAB)_\circlearrowleft (\forestCAB)_\circlearrowleft-(\forestDAB)_\circlearrowleft (\forestEAB)_\circlearrowleft\\&+(\forestFAB)_\circlearrowleft (\forestGAB)_\circlearrowleft-(\forestHAB)_\circlearrowleft (\forestIAB)_\circlearrowleft +(\forestJAB)_\circlearrowleft-(\forestKAB)_\circlearrowleft\\
{}[\forestLAB,[\forestMAB,\forestNAB]] & (\forestOAB)_\circlearrowleft +(\forestPAB)_\circlearrowleft +(\forestQAB)_\circlearrowleft + (\forestRAB)_\circlearrowleft  -2(\forestSAB)_\circlearrowleft -2(\forestTAB)_\circlearrowleft \\&+(\forestUAB)_\circlearrowleft +(\forestVAB)_\circlearrowleft  -2(\forestWAB)_\circlearrowleft +(\forestXAB)_\circlearrowleft +(\forestYAB)_\circlearrowleft -2(\forestABB)_\circlearrowleft
\\
\hline
\end{longtable}

\end{appendices}

\end{document}